\documentclass{article}
\usepackage{graphicx} 
\usepackage{amsfonts, amsmath,amssymb, amsthm}
\usepackage{tikz}
\usepackage{todonotes}
\usepackage{authblk}
\usepackage{comment}
\usepackage{natbib}
\usepackage{url}
\usepackage{hyperref}
\usepackage{cleveref}
\usepackage{subcaption}
\usepackage{esint}
\usepackage{rotating}
\usepackage{tabularx}
\usepackage[title]{appendix}%

\usepackage[margin = 1in]{geometry}

\usetikzlibrary{shapes.geometric}
\usetikzlibrary{decorations.markings}

\newtheorem{theorem}{Theorem}[section]
\newtheorem{appendixTheorem}{Theorem}[section]
\newtheorem{lemma}[theorem]{Lemma}
\newtheorem{appendixLemma}{Lemma}[section]
\newtheorem{corollary}[theorem]{Corollary}
\newtheorem{appendixCorollary}{Corollary}[section]
\newtheorem{proposition}[theorem]{Proposition}

\newtheorem{example}{Example}%

\DeclareMathOperator*{\argmax}{\text{argmax}}

\DeclareMathOperator{\Rn}{\mathbb{R}^n}
\DeclareMathOperator*{\esssup}{ess\,sup}

\tikzset{test/.style={
    postaction={
        decorate,
        decoration={
            markings,
            mark=at position \pgfdecoratedpathlength-0.5pt with {\arrow[blue,line width=#1] {}; },
            mark=between positions 0 and \pgfdecoratedpathlength-8pt step 0.5pt with {
                \pgfmathsetmacro\myval{multiply(divide(
                    \pgfkeysvalueof{/pgf/decoration/mark info/distance from start}, \pgfdecoratedpathlength),100)};
                \pgfsetfillcolor{blue!\myval!orange};
                \pgfpathcircle{\pgfpointorigin}{#1};
                \pgfusepath{fill};}
}}}}

\title{Nonlinear Nonlocal Diffusion Equations for the Analysis of Coordination}
\author[1,2]{John S. McAlister}
\author[2,3]{Nina H. Fefferman}
\author[4]{Tadele Mengesha}
\affil[1]{Princeton University, Program in Applied and Computational Mathematics}
\affil[2]{NSF Center for Analysis and Prediction of Pandemic Expansion (APPEX)}
\affil[3]{Arizona State University, School of Mathematical and Natural Sciences}
\affil[4]{University of Tennessee - Knoxville, Department of Mathematics}

\date{ }

\begin{document}

\maketitle
\begin{abstract}
Coordination games with explicit spatial or relational structure are of interest to economists, ecologists, sociologists, and others studying emergent global properties in collective behavior. When assemblies of individuals seek to coordinate action with one another through myopic best response or other replicator dynamics, the resulting dynamical system can exhibit many rich behaviors. Here we seek to understand these coordination dynamics in the non-atomic setting by use of a nonlinear, nonlocal diffusion equation framework. For this class of nonlocal equations, we prove existence and uniqueness for the initial value problem in the free boundary case and, with certain conditions on the underlying game theory, we go further and prove a maximum principle, weak regularity results, and a characterization of asymptotic behavior. Additionally, with the help of several numerical results,  we present several modeling results, characterizing stationary solutions both rigorously and through numerical experiments.
\end{abstract}



\section{Introduction}\label{SEC:Introduction}
Coordinating systems are common examples of complex systems that organize themselves in time. Examples of coordinating systems include linguistic systems, in which players modulate their language to be more similar to their associates \cite{Niederhoffer2002,Gonzales}, or cultural systems, where individuals take on the same habits and customs as the others in their communities \cite{CONLEY2013,Duley2025}. Even non-human decision makers act in a coordinating manner. Flocking behavior can be thought of as a type of positional coordination\cite{Shvydkoy2021,tadmore2021}, but more classically, bacteria engaging in quorum sensing are demonstrating an example of coordination\cite{Czaran2009,Wang2020}. The game-theoretic foundation for all of these kinds of interactions is the coordination game, which, in brief, is a broad class of interactions in which two players receive a higher payoff for using the same strategy than they would if they use different strategies (\cref{fig:CoordinatingSystems}). These kinds of complex systems are becoming increasingly relevant in the age of AI agents. When AI agents are making decisions in a dispersed manner, our ability to predict and understand the outcomes of these dispersed decisions relies on our ability to think about the fundamental game-theoretic interactions that are taking place. This was known as early as 1997 by Kraus \cite{KRAUS1997} and has only become increasingly important \cite{zhang2020}. 

\begin{figure}
    \centering
    \begin{tikzpicture}
        \draw(0,0)--(0,6)--(12,6)--(12,0)--(0,0);
        \draw(0,3.5)--(12,3.5);
        \draw(4,3.5)--(4,6);
        \node(lbl) at (0.2,5.8){a)};
        \node(lbl) at (4.2,5.8){b)};
        \node(lbl) at (0.2,3.2){c)};

        \node(a)[circle, fill = blue, inner sep = 3pt] at (0.5,5){};
        \node(b)[circle, fill = blue, inner sep = 3pt] at (1.5,5){};
        \draw (a)--(b);
        \node(lbl) at (1,4){\large$w({\color{blue}{\bullet}},{\color{blue}{\bullet}})$};

        \node(lbl) at(2,4){\large$ >$};
        
        \node(c)[circle, fill = blue, inner sep = 3pt] at (2.5,5){};
        \node(d)[circle, fill = orange, inner sep = 3pt] at (3.5,5){};
        \draw (c)--(d);
        \node(lbl) at (3,4){\large$w({\color{blue}{\bullet}},{\color{orange}{\bullet}})$};

        \node(a)[circle, fill = blue, inner sep = 3pt] at (8,5.2){};
        \node(b)[circle, fill = orange, inner sep = 3pt] at (6,5.7){};
        \node(c)[circle, fill = blue, inner sep = 3pt] at (7,4.7){};
        \node(d)[circle, fill = blue, inner sep = 3pt] at (9,5.7){};
        \node(e)[circle, fill = gray, inner sep = 3pt] at (10,4.7){};
        \node(f)[circle, fill = gray, inner sep = 3pt] at (5,5){};
        \node(g)[circle, fill = blue, inner sep = 3pt] at (11,5.4){};
        \draw(b)--(a)--(c);
        \draw(d)--(a)--(e);
        \draw(g)--(a)--(f);
        
        \node(lbl) at (8,5.5){$v$};
        \node(lbl) at (8,4){\large $w_v({\color{blue}\bullet})=w({\color{blue}\bullet},{\color{orange}\bullet})+2w({\color{blue}\bullet},{\color{gray}\bullet})+3w({\color{blue}\bullet},{\color{blue}\bullet})$};

        \node(a)[circle,fill = orange, inner sep = 2pt]at(2+6,0+1.75){};
        \node(b)[circle,fill = orange, inner sep = 2pt]at(-1+6,-1.5+1.75){};
        \node(c)[circle,fill = orange, inner sep = 2pt]at(-2+6,0.5+1.75){};
        \node(d)[circle,fill = orange, inner sep = 2pt]at(-2.8+6,1.2+1.75){};
        \node(e)[circle,fill = orange, inner sep = 2pt]at(0.1+6,1.1+1.75){};
        \node(f)[circle,fill = orange, inner sep = 2pt]at(0.9+6,-0.6+1.75){};
        \node(g)[circle,fill = orange, inner sep = 2pt]at(3+6,-0.6+1.75){};

        \node(h)[circle,fill = blue, inner sep = 2pt]at(-3+6,-1+1.75){};
        \node(i)[circle,fill = blue, inner sep = 2pt]at(-1+6,0+1.75){};
        \node(j)[circle,fill = blue, inner sep = 2pt]at(-2+6,-1.2+1.75){};   \node(k)[circle,fill = blue, inner sep = 2pt]at(-2.8+6,0+1.75){};
        \node(l)[circle,fill = blue, inner sep = 2pt]at(0.1+6,-1.2+1.75){};
        \node(m)[circle,fill = blue, inner sep = 2pt]at(0.9+6,0.6+1.75){};
        \node(n)[circle,fill = blue, inner sep = 2pt]at(2+6,1+1.75){};

        \node(o)[circle, fill = orange , inner sep = 2pt]at(0+6,0+1.75){};

        \node(p)[circle,fill = blue, inner sep = 2pt]at(1.5+6,-1.2+1.75){};
        \node(q)[circle,fill = blue, inner sep = 2pt]at(3+6,0+1.75){};
        \node(r)[circle,fill = blue, inner sep = 2pt]at(3.6+6,0.6+1.75){};
        \node(s)[circle,fill = blue, inner sep = 2pt]at(4+6,-1+1.75){};
        \node(t)[circle,fill = blue, inner sep = 2pt]at(-1.6+6,0.9+1.75){};

        \draw(d)--(c)--(k)--(h)--(j)--(k)--(i)--(t)--(c);
        \draw(d)--(t)--(e)--(o)--(i)--(j)--(b)--(i)--(l)--(o)--(f)--(m)--(e);
        \draw(b)--(l)--(g)--(a)--(m)--(n)--(q)--(r)--(n)--(e);
        \draw(l)--(p)--(g)--(s)--(q)--(a)--(n);
        \draw(d)--(2.5,3);
        \draw(d)--(2,2.1);
        \draw(k)--(2.3,1.6);
        \draw(h)--(2.1,0.5);
        \draw(t)--(5,3.2);
        \draw(n)--(9,3.1);
        \draw(r)--(10.5,2.9);
        \draw(r)--(10,2);
        \draw(q)--(10,1.8);
        \draw(s)--(10.5,1.4);
        \draw(s)--(9,0.2);
        \draw(p)--(8.5,0.2);

    \end{tikzpicture}
    \caption{\textbf{a)} In a coordination system, a player gets a better payoff ($w$) by playing the same strategy (indicated by color) as their coplayer compared to if they played a different strategy. \textbf{b)} In a multiplayer assembly, the total payoff a player $v$ earns ($w_v$) is the sum of all the pairwise payoffs they receive from interacting with all of their coplayers. \textbf{c)} When many players interact with an explicit relational structure, each player has a payoff which is a function of the strategy they choose,  the relational structure, and the strategies that the other players in the environment are playing. }
    \label{fig:CoordinatingSystems}
\end{figure}

Among all of these examples, the ways individuals are organized relative to one another, i.e., their relational structure, can be a determining factor in the ways the coordinated organization takes place. It is a central question to understand how that relational structure determines the types of game-theoretic outcomes. Although this question looks different depending on the understanding of relational structure, an easy example is to ask: through coordination alone, is it possible for a ``strategic gradient" to occur at equilibrium? By this we mean, when players are coordinating with an explicit relational structure, can there ever be a region in the player space across which strategies gradually change from one to another? We see this type of strategic gradient as shifts in dialect and accent across regions \cite{Heeringa2001} or even in sympatric speciation in signaling behavior in certain animals \cite{endler1977geographic}. By examining the coordination dynamics on their own, we can assess whether these observed gradients are the result of simple coordination or if other forces are acting on these systems.

In this manuscript, we show that this broad class of games, which includes certain coordination and anti-coordination games, when considered dynamically through a modified Myopic Best Response strategy revision protocol, can be modeled by a nonlocal nonlinear diffusion equation. These equations are of the type studied by Caffarelli \cite{Caffarelli2011}, Medvedev \cite{Medvedev2014}, and others. This continuous type of model allows us to use tools from analysis and PDEs to study how strategy profiles evolve and how they look at equilibrium. Our work mirrors the framework of Friedman and Ostrov \cite{FRIEDMAN2010,FRIEDMAN2013}, who consider Myopic Best Response dynamics in an unstructured population. The analytical understanding we can achieve in this structured setting provides novel insights about how coordinating systems can fracture and form non-consensus equilibria. It also provides us with an analytical way to study the interaction between the relational structure and the behavior of the system. Additionally, we characterize stationary solutions by their formation of strategic bands, which may be of interest for those studying nonlinear nonlocal diffusion systems more generally.  

This paper is organized as follows.  The description of the continuous game is given in \cref{SEC:Extension} with the rigorous description of the modified Myopic Best Response strategy revision protocol and its convergence to a nonlocal parabolic evolution equation. Once we attain this equation, we state and prove some standard results like existence, uniqueness, regularity, and asymptotic behavior in \cref{SEC:ExistenceUniquenessAndMainResults}. These results allow us to give some analytical examples in \cref{SEC:AnalyticalExamples} and formulate well-posed numerical methods to produce numerical results in  \cref{SEC:NumericalResults}, which we use to numerically investigate solutions in \cref{SEC:NumericalExperiments}. In \cref{SEC:StationarySolutions}, we use this model to characterize stationary solutions of the continuous coordination game. Our discussion in \cref{SEC:conclusion} describes the way that this type of model can be used to improve our analytical understanding of many types of complex systems, including coordinating systems.

\section{Model setup}\label{SEC:Extension}

\subsection{Game-theoretic background}\label{SUBSEC:GameTheoreticBackground}

The study of coordinating systems has originated several different times from different areas of complex systems and applied dynamical systems and thus has several different names. For the economists and the game theorists who call this type of interaction ``coordination," the early work was done by Kandori, Milath, and Rob \cite{Kandori1993} and Gilboa \cite{Gilboa1991}, who showed that, if a coordinating system is well mixed, under Myopic Best Response, the only long-term stable equilibria are the consensus equilibria in which every player agrees on a strategy (regardless of what that strategy is). From there, the question quickly shifted to understanding these systems in the case where every pairwise interaction is not equally likely. Much of the work on this ``structured" coordination game started with Ellison \cite{Ellison1993}, who considered how a coordinating system would behave if the players were arranged in a line, a circle, or a grid, and were interacting only with the subset of other players nearest them. The findings, which were confirmed by many other similar models \cite{Robson1995,Oechssler1997,Oechssler1999,Ely2002}, showed that with these simple spatial structures, coordinating systems with stochastic perturbations always converged to a consensus. However, that consensus is not necessarily the consensus that globally maximizes total payoff. Moreover, the consensus to which the system converges and the speed at which it converges are predictable\cite{Ellison2000}.

Sociologists and statistical physicists were studying a similar problem by a different name. Based on the work of DeGroot \cite{Degroot1974}, who proposed the opinion dynamics models, Friedkin and Johnsen \cite{friedkin1990social} formulated the same question in the context of opinions: ``What kinds of relational structures promote or suppress consensus?"  These kinds of questions are still being actively considered \cite{Zha2020}. 

The study of these systems in the structured case has continued in several different forms. Much of the work has been done through simulation, as in \cite{Buskens2016,Tomassini2010, McAlister2024simulation,Raducha2022}, which have been helpful in observing the ways that structure and game-theoretic outcomes interact, but they do not provide much analytical power in describing these interactions. Much of the analytical work recently has considered the stability of different consensus equilibria, as in \cite{Paarporn2021} and in \cite{Arditti2024}. It is still a question of great interest to consider how the relational structure of players (typically a graph in the studies cited above) can determine the outcome of coordination processes. This is of interest not only to the people studying coordination, but also to the network scientists thinking about modularity in graphs. The pure coordination game represents a locally driven form of community detection in which strategic communities (i.e., groups of players playing the same strategy) are more highly connected internally than they are externally. Comparisons between coordinated equilibria and the community detection methods studied by Clauset, Newman, and others \cite{Clauset2004, Blondel2008,Newman2006} form an interesting area of further study.

To describe our nonlinear nonlocal diffusion model, we must give some game-theoretic background. In a strategic form game, there is a set of players, $V$, which each have a set of strategies, $S_v$. That set of strategies may be finite or infinite, discrete or continuous \cite{Basar1998}. A strategy profile is a collection of players' strategies from the Cartesian product $S:=\prod_{v\in V}S_v$. Each player also has a payoff function $w_v:S\rightarrow \mathbb{R}$. For our investigation, we are interested in games where every player has access to the same set of strategies $S$, and each player has the same payoff function $w:S^{|V|}\rightarrow \mathbb{R}$. 

Crucial to the study of strategic form games is the concept of\textit{ best response}. For a focal individual $v$ and a strategy profile $s\in S$, we break the strategy profile down as $(s_v,s_{-v})$, which is the ordered pair of player $v$'s strategy and the strategies of the players other than $v$, respectively. We say that $v$'s \textit{best response} to $s$ is whatever strategy (or strategies) in $S_v$ maximizes $v$'s payoff supposing all the other players play according to $s_{-v}$. That is
\[ br_v({s})=\argmax_{c\in S_v}\{w_v(c,s_{-v})\}.\]
$s\in S$ is called a \textit{Nash equilibrium} if and only if all players are playing a best response to $s$. That is $s$ is a Nash equilibrium if and only if $ s=(s_1,s_2,...,s_n)$ and $s_v\in br_v(s)$ for all $v$.

We are interested in using a class of two-player games which we call Toeplitz games. A finite and discrete Toeplitz game has a Toeplitz payoff matrix. In this setting, each player has a set of pure strategies which can be ordered as $S_1=S_2=\{s^{(1)},s^{(2)},...,s^{(m)}\}$ so that the payoff matrix $A$ (given by $a_{i,j} =w(s^{(i)},s^{(j)})$) is diagonal-constant (AKA Toeplitz). Three examples below, from left to right, are the pure coordination game, an anti-coordination game, and a game that is neither coordination nor anti-coordination. 

\[\begin{bmatrix}
    1&0&0&0\\0&1&0&0\\0&0&1&0\\0&0&0&1
\end{bmatrix}
\quad\quad\quad
\begin{bmatrix}
    0&1&2&4\\
    1&0&1&2\\
    2&1&0&1\\
    4&2&1&0
\end{bmatrix}
\quad\quad\quad 
\begin{bmatrix}
    0&1&2&3\\
    -1&0&1&2\\
    -2&-1&0&1\\
    -3&-2&-1&0
\end{bmatrix}\]

The ordering required to write the payoff matrix as a Toeplitz matrix, $(s^{(1)},s^{(2)},...,s^{(m)})$, induces a metric on the strategy space $d(s^{(i)},s^{(j)})=|i-j|$. With a metric and a complete ordering, this kind of game can be reduced to a two-player game in which each player $v$ selects a number $n_v\in\{1,...,m\}$ and their payoff is determined by $w_v(n_v,n_u)=\rho(n_v-n_u)$. This ordering is not unique. For instance, in the pure coordination game, any ordering of the strategies will satisfy this condition. It is clear from this formulation with $\rho$ that the same type of game can be defined with any type of strategy space, discrete or continuous, finite or infinite. In general, we say that a Toeplitz game is a game where the payoff of playing strategy $z_1$ against strategy $z_2$ is a function of the signed distance $d(z_1-z_2)$ whenever such a signed distance can be constructed.

Those Toeplitz games that are coordination games are of particular interest to us. A two-player coordination game is a strategic form game that satisfies the Bandwagon Property \cite{Kandori1998}. This property states that when both players use the same strategy set, $S$, then, for each player, and for any (possibly mixed) strategy $s\in S$, $BR_v(s)\subseteq C(s)$, where $C(s)$ is the support of $s$ (i.e., all those strategies which may be played with a strictly positive probability by player $v$). In the present case, we study only pure strategies, so the support is just the strategy itself: $C(s)=\{s\}$. A game with payoff function $w$ is an anti-coordination game if and only if the game with payoff function $-w$ satisfies the bandwagon property. 

These coordination and anti-coordination games, which can be expressed as Toeplitz games, are of particular interest because they model the homogeneous impact of cooperative or uncooperative behavior. In a Toeplitz game, a particular strategy does not give any fitness benefit independent of its proximity to a co-player's strategy. Any fitness benefit is only the result of the interaction between players (whether the interactions are sympathetic or antagonistic). Because classical approaches to coordination games are often interested in selection between conventions, they typically assume strategies also have some intrinsic value, meaning there may be a separate Pareto-dominant Nash equilibrium and risk-dominant Nash equilibrium. The Toeplitz approach restricts the class of games to only those for which every consensus equilibrium is a global maximizer of total payoff, and so it cannot tell us anything about equilibrium selection in the classic sense. It can provide us insight into the selection of consensus or non-consensus equilibria mediated by spatial interactions alone. As an example, if several methods of communication are equally efficient but the benefit of communication is only achieved when two players are using the same method, this can be described as a Toeplitz coordination game. Building a foundational theory of this homogeneous problem will then allow for better treatment of inhomogeneous coordination and anti-coordination processes.

\subsection{Continuous player space}\label{SUBSEC:ContinuousExtensionsInSpace}
As described above, dyadic interactions of this type are easy to understand. Therefore, when we seek to understand multiplayer Toeplitz games, we easily generalize the two-player interactions and say that a player's payoff is the sum (or equivalently the arithmetic mean) of the payoffs from each dyadic interaction. There are many approaches to describing games with continuous player spaces, but key to our approach is that we describe our player space as non-atomic or atomless. Non-atomic games were introduced by Aumann \cite{Aumann1964} and later discussed by Schmeidler \cite{Schmeidler1973} as games in which there is a continuum of players and no player can impact the payoff of another player on its own (the measure of a single player in the player domain is 0 and so the set of players has no atoms with respect to Lebesgue measure). A popular approach to studying non-atomic games is the Mean Field Game (MFG) approach pioneered by Lasry and Lions \cite{Lasry2007}, which considers each individual as playing against the distribution of the remainder of the population identically. This approach has been successfully used with explicit spatial models when space is an element of a player's strategy \cite{Carmona2019}. A modification of this approach is the Graphon MGG approach \cite{Cains2018, Cains2021}. In such a model, every individual interacts with the strategic distribution of the population weighted by their perception of the distribution, which is determined by their position in an infinitely dense graph called a graphon. Key to understanding this approach is that individuals are not anonymous and not exchangeable \cite{Yoshioka2025}. Our approach is similar to the Graphon MFG setup, although our analysis will be very different from the typical MFG analysis. Below we define our continuous spatial model through an analogy with a discrete graph. Similar types of setups can be found in \cite{Nanda2017},\cite{AYDOGMUS2017},\cite{Hwang2013}\cite{Carmona2022}.

 As classically described by Blume in \cite{BLUME1993}, the discrete setting considers the players as vertices in a graph where (possibly weighted) edges describe the strength of interaction between individuals. Suppose the game is played on the graph $G(V,E)$ with weighted adjacency matrix $W$ and that each player has pure strategies $B=\{\hat{e}_1,\hat{e}_2,...,\hat{e}_m\}$, which is the standard basis for $\mathbb{R}^m$. Then, if we consider a strategy profile as a function $u:V\rightarrow B$ rather than an element of the Cartesian product $B^{|v|}$ (although the two spaces are clearly identical), we can write our payoff function as 
\begin{equation}\label{disPayoff}
    w(v|u)=\sum_{i\in V}W_{i,v}u(v)^\intercal Au(i)    
\end{equation}
where $A$ is the payoff matrix as described above. This formulation holds true for any symmetric game with a payoff matrix $A$; it need not be Toeplitz. 

If, instead of a collection of discrete players, we considered a continuum of players in$\Omega\subset\mathbb{R}^n$, and instead of summing we would integrate over the entire domain. The weighted adjacency matrix is replaced by a nonnegative integrable kernel $K\in C_b^0(\Omega; L^1(\mathbb{R}^n))$ (That is, a kernel $K(x,y)$ which is continuous in $x$, in the $L^1$ sense, and integrable in $y$, where $\sup_{x\in\Omega}\|K(x,\cdot)\|_{L^1(\mathbb{R}^n)}<\infty$). Replacing the sum in \eqref{disPayoff} with the integral, we get that a strategy profile $u:\Omega\rightarrow B$ gives the payoff
\begin{equation}\label{ctsSpacePayoff}
w(x|u)=\int_\Omega K(x,y)u(x)^\intercal Au(y)dy.
\end{equation} This analogy is quite natural and can be pictured in  \cref{fig:ContinuousPlayerSpace}.

\begin{figure}
    \centering
    \begin{tikzpicture}
        \node[circle, fill=blue, inner sep = 2pt] (a) at (0,0){};
            \node[circle, fill=blue, inner sep = 2pt] (b) at (1,0){};
            \node[circle, fill=blue, inner sep = 2pt] (c) at (2,0){};
            \node[circle, fill=orange, inner sep = 2pt] (d) at (3,0){};
            \node[circle, fill=orange, inner sep = 2pt] (e) at (4,0){};

            \node(lbl) at(2,-0.4){$v$};
            \draw(a)--(b)--(c)--(d)--(e);
            \draw(c) to [out = 145, in = 35] (a);
            \draw (c) to [out = 35, in = 145] (e);
            \node (lbl) at (3.5,0.5){$W_{v,w}$};
            \node (lbl) at (2,-1.2){$w_v({u_v}|u)=\sum\limits_{i\in V}W_{v,i}{u_v}^\intercal Au_i$};

            \draw[line width = 3pt,draw = blue](5+1,0)--(7.5+1,0);
            \draw[line width =3pt,draw=orange](7.5+1,0)--(9+1,0);
            \node(lbl) at (7+1,-0.4){$x$};
            \draw(8,0)--(8,-0.15);
            \draw plot [smooth] coordinates { (5+1,0.05)(6+1,0.2)(6.6+1,1)(7+1.4,1)(8+1,0.2)(9+1,0.05)};
            \node(lbl) at ( 8.7+1,0.5){$K(x,y)$};
            \node(lbl) at (7+1,-1.05){$w(x|u)=\int_{\Omega}K(x,y)({u(x)}^\intercal Au(y))dy$};

            \draw[->] (4,1) to [out = 35, in =145] (6,1);
            \node(lbl) at (2,1.5){Discrete Player Space};
            \node (lbl) at (8,1.5){Continuous Player Space};
    \end{tikzpicture}

     \caption{In a discrete coordinating system, each player's payoff is determined by the sum or average of the pairwise payoffs that player receives from each of its co-players. In the continuous setting, that sum becomes an integral and the adjacency matrix becomes an integrable kernel.}
    \label{fig:ContinuousPlayerSpace}
\end{figure}

Notice that this is not a limiting approach from graph to continuum. We simply show that the two are analogous and do not study the asymptotic properties of graphs approaching the infinitely dense graphon limit. We refer the reader to the work of Carmona et al. \cite{Carmona2022} who prove that Nash equilibria in the graphon games are good approximations of Nash equilibria in the very large graph limit.  Our assumption that $K(x,\cdot)$ is continuous in the $L^1$ sense is more restrictive than the assumptions typically used to describe graphons; however, we can certainly construct a sequence of graphs of increasing size so that the limiting adjacency matrix behaves like $K(x,y)$. For this reason, the results here cannot at present be generalized to all graphon games, but the intuition about how these games behave is similar.  

\subsection{Continuous strategy space}\label{SUBSEC:ContinuousExtensionsInStrategy}
We must also discuss continuous strategies. Again, there are many approaches to games with continuous strategy spaces \cite{Basar1998,FRIEDMAN2013,CHEUNG2018}.
Here we consider a 1-dimensional continuum of pure strategies. To illustrate the analogy between the discrete and continuous strategy space, consider a Toeplitz payoff matrix. We can express the pairwise payoff simply as some function $\rho(dd(u(x),u(y))$, where $dd$ is the signed distance, which is natural from the ordering induced by the Toeplitz matrix. The continuous analogy from this point is clear. The ordering of the strategies means that they can be considered as elements of $\mathbb{R}$, and we can allow $\rho:\mathbb{R}\to \mathbb{R}$ so that our payoff can be expressed as $\rho(u(x)-u(y))$. We call this $\rho$ a ``recognition function" as it describes how well one player may recognize another player's strategy as a function of the distance between their strategies. With this recognition function, we can write our payoff function as 
\begin{equation}\label{ctsStratPayoff}
w(x|u)=\int_{\Omega}K(x,y)\rho(u(x)-u(y))dy.
\end{equation}
It is very common in the examples given in \cref{SEC:Introduction} that strategies can vary continuously and players do not judge strategies as ``all or nothing." Rather, two strategies may be ``close enough" to still confer a high payoff (\cref{fig:ContinuousStrategySpace}). As an example, if two people are communicating, one is using American English and the other British English, these two strategies are ``close enough" to still confer the same payoff as if the two were speaking exactly the same dialect. This idea of mutual intelligibility is a feature of many types of coordinating systems, not just this single language example. 

\begin{figure}
\definecolor{b2o}{rgb}{0.5,0.25,0.5}
\definecolor{b3o}{rgb}{0.25,0.125,0.75}
\definecolor{b4o}{rgb}{0.75,0.375,0.25}
    \centering
    \begin{tikzpicture}
        \node (lbl) at (2.2,-2.05){\large$\begin{bmatrix}
            0.95&0.05&0\\
            0.05&0.95&0.05\\
            0&0.05&0.95
        \end{bmatrix}$};
        \node[circle, fill=blue, inner sep = 2pt] (b) at (0.5,-1.5){};
        \node[circle, fill=b2o, inner sep = 2pt] (b) at (0.5,-2){};
        \node[circle, fill=orange, inner sep = 2pt] (b) at (0.5,-2.5){};
        \node[circle, fill=blue, inner sep = 2pt] (b) at (1.3,-1){};
        \node[circle, fill=b2o, inner sep = 2pt] (b) at (2.2,-1){};
        \node[circle, fill=orange, inner sep = 2pt] (b) at (3.1,-1){};

        \node (lbl) at (2,-3.7){$w(a,b)=a^\intercal Ab$};

        \draw[->](4,-2)to [out = 35, in = 145](7.5,-2);
        \node(lbl) at (5.75,-2.5){\parbox{3cm}{\centering\footnotesize Continuous completion of metric induced by Toeplitz ordering}};


        \path [test=2pt] (11.5,-2.75) .. controls (10,-3+0.25) and (9,-3+0.25) .. (8,-3+0.25);
        \draw plot [smooth, tension =0.5] coordinates { (8.5+0.75,-2.8+0.25)(9+0.75,-2.7+0.25)(9.6+0.75,-2+0.25)(10+0.75,-1.6+0.25)(10.4+0.75,-2+0.25)(11+0.75,-2.7+0.25)(11.5+0.75,-2.8+0.25)};
        \node(lbl) at (11+0.75,-1.5+0.25){$\rho({\color{b4o}\bullet}-b)$};

        \draw[dashed] plot [smooth, tension =0.5] coordinates { (8.5-0.75,-2.8+0.25)(9-0.75,-2.7+0.25)(9.6-0.75,-2+0.25)(10-0.75,-1.6+0.25)(10.4-0.75,-2+0.25)(11-0.75,-2.7+0.25)(11.5-0.75,-2.8+0.25)};
        \node(lbl) at (9-0.75,-1.5+0.25){$\rho({\color{b3o}\bullet}-b)$};

        \node(lbl) at (10,-3.7){$w(a,b)=\rho(a-b)$};
        \node(lbl) at (2,-0.25){Discrete Strategy Space};
        \node(lbl) at (10,-0.25){Continuous Strategy Space};

    \end{tikzpicture}
     \caption{In the discrete strategy space, the payoff is given by a simple bilinear form with the payoff matrix $A$. If the focal individual plays strategy $a$ against strategy $b$, the payoff is $ w(a,b)=A_{a,b}=a^\intercal Ab$. If $A$ is Toeplitz, then the complete ordering of the strategies induces a metric on the strategies which can be interpolated continuously. This means that, in the continuous setting, the bilinear form is replaced with a function that continuously relates the distance between two strategies to the payoff that results from those two strategies being used against each other. }
    \label{fig:ContinuousStrategySpace}
\end{figure}

To complete the analogy, all we have done is replace the bilinear form in \eqref{ctsSpacePayoff} with our payoff function induced by the Toeplitz matrix. We are interested in this type of game, both because of its implications in the application areas and because, in the continuous form, it has great similarities to nonlocal diffusion equations. Indeed, the time-dependent model will be a nonlinear, nonlocal diffusion equation, and understanding this model will contribute to the understanding of nonlocal diffusion problems in general.  

With this fitness function, we have now described a strategic form game completely, as we have a set of players $\Omega$, a set of strategies $\mathbb{R}$, and a payoff function \eqref{ctsStratPayoff}. It is important to note that this setting complicates our idea of a solution concept. As written by Schmeidler \cite{Schmeidler1973}, the pointwise Nash equilibrium concept cannot apply in a non-atomic game because for almost any pointwise Nash equilibrium there is a strategy profile which is equivalent in $L^p$ which is not a Nash equilibrium. At best we can say that a strategy profile is a Nash equilibrium almost everywhere (i.e. the Nash condition is satisfied for almost every player in the measure-theoretic sense).  
 
There are other refinements of the almost everywhere Nash equilibrium which may be useful (e.g., \cite{BARLO2015,CARRILLO2023} etc.), but we will eventually study the properties of the almost everywhere Nash equilibria, Local Nash equilibria (almost everywhere), and Local $\epsilon-$Nash equilibria. On our way to studying these solution concepts and consider the game dynamically through myopic best response (a la  Ellison \cite{Ellison1993}, Raducha and San Miguel \cite{Raducha2022} and others). Myopic best response is a strategy revision protocol for a dynamic game in which a set of players, who are chosen to update their strategies, take on their best response (often pure strategy best response) to the current strategy profile.  By repeating this process, we can study a time series of strategy profiles which, if it terminates, will result in a strategy profile in which every player is playing a best response. Under myopic best response, we can see that the evolution of strategy profiles for the game with payoff function \eqref{ctsStratPayoff} will evolve according to a particular nonlocal equation.

A general framework for understanding the evolution of strategy profiles through a myopic best response in continuous time in an unstructured population is given by Friedman and Ostrov \cite{FRIEDMAN2010}. They show that under the assumption that changing strategies is quadratically costly, the distribution of strategies in a well-mixed population evolves according to a nonlinear PDE. The same intuition for modeling the evolution of strategic distributions without space can be used to model spatial distributions of pure strategies. The general framework that we prove for the structured game coincides with that described for the unstructured game in \cite{FRIEDMAN2010} and \cite{FRIEDMAN2013}.
    
There are several assumptions necessary to show our equivalence. First, we assume that the time step is arbitrarily small. Secondly, as in \cite{FRIEDMAN2010}, we assume that the cost of changing your strategy by $h$ in $\Delta t$ is $\frac{h^2}{\Delta t}$. This assumption means that there are increasing marginal costs for changing strategy in a single time step. The cost of a strategic change from $0$ to $h$ can be thought of as the product of the effort involved in changing a strategy in a unit time, $\frac{h}{\Delta t}$, which grows linearly with the size of the change, and the effort involved in learning the new strategy, $h$, which also grows linearly with the size of the change. The total cost of changing strategy, therefore, is $\frac{h^2}{\Delta t}$.  The final assumption on $\rho$ is required so that the best response function is defined (although possibly set-valued). Otherwise, in extreme cases like quartic anti-coordination $\rho(z) = -z^4$, the payoff for changing a strategy may increase without bound, so there is no maximizer of payoff.  With these assumptions, we can show that this game can be described in the continuous-time limit by a nonlinear nonlocal diffusion equation.

    \begin{proposition}\label{PROP:ModelWellFounded}
        Under myopic best response, bounded strategy profiles of the game with players $\Omega\subseteq \mathbb{R}^n$ choosing strategies in $\mathbb{R}$, with fitness as in \eqref{ctsStratPayoff}, will evolve as
       
       \[\frac{\partial}{\partial t}u(x,t)=\int_{\Omega}K(x,y)\rho'(u(x,t)-u(y,t))dy\]
        so long as the following three hypotheses are met
        \begin{itemize}
            \item[(H1)] Players change their strategies in arbitrarily small time steps $\Delta t$
            \item[(H2)] Players incur a quadratic cost $\frac{h^2}{\Delta t}$ for changing their strategy
            \item[(H3)] $\rho\in C^{1,1}(\mathbb{R})$ and $\rho(z)\leq Cz^2+A$ for some nonnegative constants $C,A$
        \end{itemize}
    \end{proposition}

    \begin{proof}
        Consider a bounded strategy profile $u(\cdot, t):\Omega\to \mathbb{R}$. Every player will seek to update their strategy by some amount $h$ in order to take on their best response to $u(\cdot, t)$ after a time step of $\Delta t$, and in doing so they will incur a cost of $\frac{h^2}{\Delta t}$. Let 
        
       \[S_x(h):=\int_\Omega K(x,y)\rho(u(x,t)+h-u(y,t))dy\] be the payoff that player $x$ will receive after updating their strategy by $h$. Because $\rho$ is subquadratic (H3), we can see that 

        \begin{equation*}
        \begin{split}
            S_x(h)&\leq C\int_\Omega K(x,y)(h+(u(x,t)-u(y,t))^2dy+\underbrace{A\sup_{x\in\Omega}\|K(x,\cdot)\|_{L^1(\Omega)}}_{A_1}\\
            &\leq \underbrace{C\int_\Omega K(x,y)dy}_{C_1}h^2 + 2hC\int_\Omega K(x,y)(u(x,t)-u(y,t))dy\\
            &\quad +C\int_\Omega K(x,y)(u(x,t)-u(y,t))^2dy+A_1
        \end{split}
        \end{equation*}
        Since $K\in C_b^0(\Omega; L^1(\mathbb{R}^n))$, $C_1$ and $A_1$ can be bounded by constant independent of $x$. Now because $u$ is bounded we know that $|u(x,t)-u(y,t)|$ is bounded by some $M$ and thus we can write 
        \begin{equation*}
            \begin{split}
                S_x(h)&\leq C_1h^2+\underbrace{2CM\sup_{x\in\Omega}\|K(x,\cdot)\|_{L^1(\Omega)}}_{B_1}h+ \underbrace{CM^2\sup_{x\in\Omega}\|K(x,\cdot )\|_{L^1(\Omega)} + A_1}_{A_2}\\
                &\leq C_1h^2+B_1h+A_2
            \end{split}
        \end{equation*}
        Having shown that $S_x(h)$ is uniformly subquadratic in $h$, we also note that when $\rho\in C^{1,1}$ (H3), $S_x(h)$ is also in $C^{1,1}$. Observe that 
        \begin{equation*}
            \frac{d}{dh}S_x(h)=\int_\Omega K(x,y)\rho'(u(x,t)+h-u(y,t))dy
        \end{equation*}
        And so clearly for any compact subdomain $I\subset \mathbb{R}$ we can write 
        \begin{equation}\label{EQ:SinC11}
        \begin{split}
            \bigg|&\frac{d}{dh}S_x(h_1)-\frac{d}{dh}S_x(h_2)\bigg|\\
            &\leq \int_\Omega |K(x,y)| |\rho'(u(x,t)+h_1-u(y,t))-\rho'(u(x,t)+h_2-u(y,t))|dy\\
            &\leq \int_\Omega K(x,y)L_{\rho}|h_1-h_2|dy\\
            &\leq \sup_{x\in\Omega}\|K(x,\cdot)\|_{L^1(\Omega)}L_{\rho}|h_1-h_2|
            \end{split} 
        \end{equation}
        Because $u$ is bounded, the choice of a compact subdomain $I$ gives us a Lipschitz constant for $\rho'$ on the subdomain $[-2\sup |u|-\sup I,2\sup|u|+\sup I]$ which is called $L_\rho$ in \eqref{EQ:SinC11}. Therefore $\frac{d}{dh}S_x(h)$ is locally Lipschitz. 

        This is important because it means that, for any $x$, when $\Delta t$ is small enough, $S_x(h)-\frac{h^2}{\Delta t}\leq C_2h^2+A_3$ for some negative $C_2$ and some $A_3$. Thus $S_x(h)-\frac{h^2}{\Delta t}$ must have at least one global maximizer, $h^*$ and that global maximizer will satisfy 
        
       \[\frac{d}{dh}S_x(h^*)=2\frac{h^*}{\Delta t}\]
        because the set of strategies has no boundary and because everything is continuously differentiable. 

        Having observed this, we also note that there must be a negative $h^-$ so that, when $h<h^-$, then $S_x(h)-\frac{(h)^2}{\Delta t}<S_x(0)$. Likewise there must be an $h^+$ so that $h>h^+\implies S_x(h)-\frac{(h)^2}{\Delta t}<S_x(0)$.
        These $h^-$ and $h^+$ will form a compact interval which will surely contain $h^*$. Moreover, from the chosen interval we have a Lipschitz constant for $\frac{d}{dh}S_x$ which we call $L_S$.
        The next step is to put bounds on $h^*$. In order to do this, we will consider different cases: when $h^*>0, h^*<0$, and $h^*=0$. 

        If $h^*>0$ we write that 
        \begin{equation*}
            \begin{split}
                -h^*L_S&\leq \frac{d}{dh}S_x(h^*)-\frac{d}{dh}S_x(0)\leq h^*L_S\\
                -h^*L_S&\leq \frac{h^*}{2}-\frac{d}{dh}S_x(0)\leq h^*L_S\\
                \end{split}
        \end{equation*} 
        Reorganization on each side of the inequality, with the knowledge that $\Delta t$ can be made small enough so that $2-L_S\Delta t>0$, will yield 
        \begin{equation*}
                \frac{d}{dh}S_x(0)\frac{\Delta t}{2+L_S\Delta t}\leq h^*\leq \frac{d}{dh}S_x(0)\frac{\Delta t}{2-L_S\Delta t}
        \end{equation*}

        Likewise we can show that if $h^*<0$ we will get the inequality 
        \begin{equation*}
            \frac{d}{dh}S_x(0)\frac{\Delta t}{2-L_S\Delta t}\leq h^*\leq \frac{d}{dh}S_x(0)\frac{\Delta t}{2+L_S\Delta t}
        \end{equation*}

        We also note here that if $h^*=0$ this requires that $\frac{d}{dh}S_x(0)=0$. 

        Now that we have bounds on $h^*$, recall that $h^*$ is the change in strategy in one time step. That is $h^*=u(x,t+\Delta t)-u(x,t)$. If we make this substitution and divide our inequalities by $\Delta t$, we see that 
        \begin{equation*}
            \begin{split}
                \frac{d}{dh}S_x(0)\frac{1}{2+L_S\Delta t}&\leq \frac{u(x,t+\Delta t)-u(x,t)}{\Delta t}\leq\frac{d}{dh}S_x(0)\frac{1}{2-L_S\Delta t}\quad \quad h^*>0\\
                \frac{d}{dh}S_x(0)\frac{1}{2-L_S\Delta t}&\leq \frac{u(x,t+\Delta t)-u(x,t)}{\Delta t}\leq\frac{d}{dh}S_x(0)\frac{1}{2+L_S\Delta t}\quad \quad h^*<0
            \end{split}
        \end{equation*}
        Both inequalities trivially hold in the case that $h^*=0$. In any case, when we take $\Delta t\to 0$ we see that by the squeeze theorem 
        \begin{equation*}
            \frac{\partial}{\partial t}u(x,t)=\frac{1}{2}\frac{d}{dh}S_x(0)
        \end{equation*}

        Note that this is true for any global maximizer $h^*$, so the best response need not be unique to have a unique limit. To complete the proof, we need only note that 
        
       \[\frac{d}{dh}S_x(0)=\int_\Omega K(x,y)\rho'(u(x,t)-u(y,t))dy\] and do a trivial rescaling of space-time to arrive at the desired nonlocal equation.
        \end{proof}

        Thus we have shown that, for any Toeplitz-type game, bounded strategy profiles will evolve in time according to a nonlocal equation. We will call the nonlocality 
        \begin{equation}\label{nonlocality}
            g[u](x,t):=\int_\Omega K(x,y)\rho'(u(x,t)-u(y,t))dy
        \end{equation}
        and express the nonlocal equation as $u_t=g[u]$. This is similar in form to the nonlocal heat equation. Indeed, if  $\rho= -\frac{1}{2}z^2$, then this Toeplitz coordination game is exactly the nonlocal heat equation \cite{NonlocalDiffusionProblems,GOMEZ2017}. This linear example has been used to study continuous opinion dynamics by Aletti and Naldi \cite{ALETTI2022}. More is said about this particular example in \cref{SEC:AnalyticalExamples}. 

        Notice that $u_t=g[u]$ is a gradient flow of the function $\int_\Omega\int_\Omega K(x,y)\rho(u(x)-u(y))dydx$ and so coincides with the unstructured framework of Friedman and Ostrov. The fact that it is a gradient flow is not at all surprising because, if $\rho$ is convex, this game is a potential game, and thus Nash equilibria are local maxima of the game potential \cite{Monderer1996,Quang2016,SANDHOLM2001}. Moreover, Friedman and Ostrov have described the same result for general potential games in the unstructured setting \cite{FRIEDMAN2010}.
        These gradient flows are studied by Caffarelli \cite{Caffarelli2011} in a more restrictive setting and by Medvedev \cite{Medvedev2014} in the graphical limit. Our present investigation mildly extends the results of Caffarelli by showing similar results in a less restrictive class of nonlinear nonlocal diffusion equations. It also furthers the work of Friedman and Ostrov by adapting their work to structured populations.

\section{Model analysis}\label{SEC:ExistenceUniquenessAndMainResults}
\subsection{General Toeplitz games}\label{SUBSEC:GeneralToeplitzGames}
 We begin with some classical existence and uniqueness results for the initial value problem (IVP) wherein an initial strategy profile $u_0$ is given in $\Omega$, and there are no boundary data prescribed. The assumption is that no information is entering the system from outside the domain, and so it is argued by Chasseigne, Chaves, and Rossi \cite{CHASSEIGNE2006} that this is the analogue to the Neumann boundary condition for nonlocal linear diffusion-type problems. It is important to note that these results for the nonlocal equation do not depend on the assumptions that make the model appropriate for the game-theoretic application ($H1-H3$ from \cref{PROP:ModelWellFounded}). Namely, $\rho$ need not be subquadratic. However, we will require $\rho\in C^{1,1}$. 
 
 
 The main result of this section is a Picard iteration-type proof relying on the contraction mapping principle. In order to achieve this, we will first state two important lemmas. The first will show that the nonlocality, $g$, maps continuous and bounded functions to continuous and bounded functions. Its proof is given in Appendix \ref{APP:Proofs} under \cref{APPLEM:welldefined}. The second will show that $g$ is locally Lipschitz with respect to the sup norm in $C^0_b$. Its proof is given in Appendix \ref{APP:Proofs} under \cref{APPLEM:Lipschitzg}. For the following, let $\Omega\subset \mathbb{R}^n$ be an open domain and let $\Omega_T:=\Omega\times [0,T]$. 

  \begin{lemma}[$g$ is well defined]\label{LEM:welldefined}
        The nonlocal operator $g[u](x,t)=\int_\Omega K(x,y)\rho'(u(x)-u(y))dy$ is well defined from $C^0_b(\Omega_T,\mathbb{R})$ into $C^0_b(\Omega_T,\mathbb{R})$.
    \end{lemma}    
    
\begin{lemma}[Lipschitz continuity of $g$] \label{LEM:Lipschitzg}
    For any bounded (in the sup norm sense) subset $X_R:=\{u\in C^0_b(\Omega_T);\|u\|_\infty \leq R\}\subset C^0_b(\Omega_T)$, there exists a $C^g\geq 0 $ such that for every $u,v\in X_R$ 
    \begin{equation*}
        \|g[u](\cdot,t)-g[v](\cdot,t)\|_{\infty} \leq C^g\|u(\cdot,t)-v(\cdot,t)\|_{\infty} 
    \end{equation*}
\end{lemma}

    Having stated these lemmas, we can now prove short-time existence and uniqueness of the solution to the initial value problem $u_t=g[u]$ in $\Omega_T$ with $u(0,t)=u_0\in C^0_b(\Omega)$ through a contraction mapping principle shown in Appendix \ref{APP:Proofs} as \cref{APPTHM:ShortTimeExistenceAndUniqueness}.

    \begin{theorem}[Short time existence and uniqueness] \label{THM:ShortTimeExistenceAndUniqueness}
        The initial value problem $u_t=g[u]$ in $\Omega_\tau$ has a unique continuous and bounded solution in $\Omega_\tau$ for some $\tau$, when $u(x,0)=u_0\in C_b^0(\Omega)$, $\rho\in C^{1,1}(\mathbb{R})$, and $K\in C^0_b(\Omega;L^1(\Omega))$.
    \end{theorem}

    We have shown that for any continuous Toeplitz game with certain hypotheses on $K$ and $\rho$, and with a bounded and continuous initial strategy profile, there is exactly one way the strategy profile will evolve for some time. Next, we note that if there is a finite maximal existence time, $T$ (i.e., $u(x,t)$ cannot be extended beyond time $T$ as a solution of the IVP), then there is necessarily finite-time blowup. The proof is given in Appendix \ref{APP:Proofs} under \cref{APPLEM:FiniteTimeBlowup} and is very similar to the proof of the same result in the ODE setting. 

    \begin{lemma}[Finite time blow up]\label{LEM:FiniteTimeBlowup}
    If  $T<\infty$ is the maximal time of existence for a solution $u$ to the IVP $u_t=g[u]$ with $u(\cdot,0)=u_0\in C^0_b(\Omega)$ then $\|u(\cdot, t)\|_\infty \rightarrow \infty $ as $t\to T$.
    \end{lemma}

    These results are for general continuous Toeplitz games with appropriate $\rho$ and have also been proved by Cafarelli \cite{Caffarelli2011} and Medvedev \cite{Medvedev2014}. These general games are interesting, but we can achieve more specific results when we constrain our study to Toeplitz coordination games.

\subsection{Coordination toeplitz games}\label{SUBSEC:CoordinationToeplitzGames}
In the discrete case, if a Toeplitz game has a positive Strictly Diagonally Dominant (SDD) payoff matrix, then it is necessarily a coordination game. Although the translation between an SDD matrix and the function $\rho$ used to describe pairwise payoff is not clear, we consider the continuous coordination Toeplitz game as games where
\begin{equation}\label{coordinationcondition}
\rho'(z)\begin{cases}\leq0&z>0\\
=0&z=0\\
\geq 0 & z<0\end{cases}
\end{equation} 
This implies that $\rho$ achieves its global maximum at $z=0$, and so the pairwise interaction with pure strategies governed by such a $\rho$ will satisfy the bandwagon property (and thus, this is indeed a coordination game). This restriction on $\rho$ allows us to think more specifically about how coordination behavior evolves in time. We should note that there are games, even in this continuous setting, which are coordination games by the bandwagon property but do not satisfy this condition. However, every game which satisfies \eqref{coordinationcondition} is a coordination game.  

It is appropriate to make the comparison between the continuous Toeplitz coordination games and nonlocal diffusion equations, which are well covered in \cite{NonlocalDiffusionProblems}. Indeed, if we choose $\rho(z)=\frac{-1}{2}z^2$, our coordination equation becomes exactly the linear nonlocal diffusion equation $u_t=\int_\Omega K(x,y)(u(y)-u(x))dy$. Moreover, regardless of our choice of $\rho$, if we assume it is even and $C^{1,1}$, near zero $\rho'(z)\approx-z.$ Because of this similarity, we will expect some of the same behavior as the nonlocal diffusion equation. In particular, we will see that there is a weak maximum principle, and thus solutions to the IVP will exist globally in forward time. In the same way, it is appropriate to compare a continuous Toeplitz anti-coordination game to the backward nonlocal diffusion equation. The proof of this weak maximum principle is short and included in Appendix \ref{APP:Proofs} as \cref{APPLEM:MaximumPrinciple}

\begin{lemma}[Weak maximum principle]\label{LEM:MaximumPrinciple}
        If $u$ solves $u_t = g[u]$ in $\Omega_T$ with $u(\cdot, 0)=u_0\in C^0_b(\Omega)$ and if $\rho$ satisfies \eqref{coordinationcondition}, then 
       \[\|u(\cdot, t_2)\|_\infty\leq\| u(\cdot , t_1)\|_\infty \] whenever $t_1\leq t_2$. 
    \end{lemma}

This is tremendously helpful because we can use it in conjunction with \cref{LEM:FiniteTimeBlowup} to show that solutions exist for all time when $\rho$ satisfies \eqref{coordinationcondition}. 
    

    \begin{theorem}[Global existence and uniqueness with particular $\rho\in C^{1,1}$]\label{THM:GlobalExistenceAndUniqueness}
    Let $\rho\in C^{1,1}(\mathbb{R})$ satisfy \eqref{coordinationcondition}. Under this strengthened hypothesis, the Initial Value Problem $u_t=g[u]$ with $u(x,0)=u_0\in C_b^0(\Omega)$ has a unique continuous and bounded solution for all finite time.
    \end{theorem}

    \begin{proof}
        Observe that for any time $t$, the solution $u(\cdot t)$ must be in the {closed and bounded} subset 
        \[E_0:=\{u\in C_b^0(\Omega);\|u\|_\infty\leq \|u_0\|_\infty\}\subset C^0_b(\Omega)\]
        because of \cref{LEM:MaximumPrinciple}. If the maximal time of existence is $T<\infty$, then it must leave this {closed and bounded} subset before time $T$ by \cref{LEM:FiniteTimeBlowup}. This is a contradiction, so $u$ cannot have a finite maximal time of existence. 
    \end{proof}

    \Cref{LEM:MaximumPrinciple} and \cref{THM:GlobalExistenceAndUniqueness} are consistent with our understanding of the coordination game in the discrete case. For multiplayer coordination, it is an easy extension of the bandwagon property to say that innovation outside of the support of the current strategy profile is never a best response. Indeed, \cref{LEM:MaximumPrinciple} is the continuous version of the Weak Bandwagon Property of Cui and Shi \cite{cui2022}, which says that it is never optimal for an individual to take on a strategy not used by any opponent in a multiplayer coordination game. 

    The boundedness of solutions for the coordination game enables us to give a weak regularity result.  This result will be especially important when we attempt to approximate solutions through numerical methods. It is stated here and proved in Appendix \ref{APP:Proofs} as \cref{APPTHM:Regularity}

    \begin{theorem}[Regularity for the coordination game without boundary conditions]\label{THM:Regularity}
        Let $\Omega\subseteq \mathbb{R}^n$. Suppose $u$ solves the Initial Value Problem $u_t=g[u]$ with the coordination assumption \eqref{coordinationcondition} and with $u_0\in C^{0,1}_b(\Omega)$ with a uniform Lipschitz constant $L_0$. If $K$ is uniformly Lipschitz in the first variable with respect to the $L^1$ norm, then for any finite time $t$, $u(\cdot, t)\in C^{0,1}(\Omega)$. Moreover, if $\Omega$ is open and bounded with a $C^1$ boundary,  $u\in W^{1,\infty}(\Omega)$. Furthermore, for some positive $c$ and $C$, 
        
       \[\|D_xu(\cdot, t)\|_{L^\infty(\Omega)}\leq (L_0+Ct)e^{ct} \]
        Where $C$ depends only on $K$ and $u_0$ and $c$ depends only on $K$ and $\rho$.
    \end{theorem}

    Note that the Lipschitz continuity of $u$ holds without restrictions on the domain. Indeed, if $\Omega$ is not bounded, we have a global Lipschitz constant. Also note that in any open domain $\Omega$ and for each $t$, $u(\cdot, t)\in W^{1,\infty}_{loc}(\Omega)$. The following corollary is the final ingredient that allows us to be confident in our numerical simulations. The proof is included in Appendix \ref{APP:Proofs} under \cref{APPCOR:GlobalLip}

    \begin{corollary}\label{COR:GlobalLip}
        For $\Omega \subseteq \mathbb{R}^n$, suppose that $u$ solves the IVP $u_t=g[u]$ under the assumptions in \cref{THM:Regularity} on a finite time domain $[0,T]$. In this case, $u$ is globally Lipschitz in $\Omega\times[0,T]$
    \end{corollary}

    Before we proceed to the numerical treatment of solutions, we would like to characterize the asymptotic behavior of solutions. Unfortunately, we cannot prove that solutions will converge to a limit at present. Instead, we show that solutions are dissipative but describe how this is not sufficient for convergence to a limit. The proof of the following lemma and its corollary are included in Appendix \ref{APP:Proofs} as \cref{applem:dis} and \cref{Appcor:intColapse} respectively. 
    
    \begin{lemma}[Solutions are dissipative]\label{lem:dis}
        Suppose that $K$ is symmetric, $\rho$ is even, satisfying \eqref{coordinationcondition}, and $\rho'$ has only a single zero in the interior of its support. If $u$ is a solution the IVP $u_t=g[u]$ and $u(\cdot, 0)=u_0\in L^2(\Omega)$. Then $\|u(\cdot,T)\|_{L^2(\Omega)}\leq \|u_0\|_{L^2(\Omega)}$ for all $T$.
    \end{lemma}

    \begin{corollary}[Collapse of the intermediate range]\label{cor:intColapse}
        Suppose $\rho$ is as in Lemma  \ref{lem:dis} with support in $[-a,a]$. Let $I$ be any closed subinterval in $(0,a)$ and let $E_I(t):=\{(x,y)\in \Omega\times\Omega; |u(x,t)-u(y,t)|\in I\}$. For any $\epsilon>0$ we can pick a $T$ such that 
        \[\int_{T}^\infty\iint_{E_I(t)}K(x,y)dxdy<\epsilon\]
    \end{corollary}
    
    Frustratingly, this is not sufficient to show that solutions converge to a limit. If the corollary \ref{cor:intColapse} could be strengthened to say that $\iint_{E_I(t)}K(x,y)dxdy\to 0$, we could get closer to a convergence proof, but without uniform continuity of this quantity in time, we can not make progress in this direction. Even noting that solutions are gradient flows of the total payoff $P[u]=\int_\Omega\int_\Omega K(x,y)\rho(u(x)-u(y))dydx$, and that our game potential gives us a Lyapunov function, without the knowledge that solutions are compact, we cannot even characterize the $\omega$-limit sets of solutions.  Although we cannot give a general convergence result, we can compare what we know about the asymptotics of solutions in the general case to results we can derive separately about stationary solutions in \ref{SEC:StationarySolutions}. 

    We can achieve a convergence result if we further restrict the class of admissible $K$ and $\rho$. Although the following results are not general, we still present them here because they represent a strengthening of a theorem of Andreu-Vaillo et al. \cite{NonlocalDiffusionProblems}, on which the argument is based. We first present a lemma, which is proved as Lemma \cref{applem:PositiveEigenvalues} in Appendix \ref{APP:Proofs}. The proof relies on a result which appears later in section \ref{SEC:StationarySolutions}, but that result does not require any result from the present section. 
    \begin{lemma}[Nonlocal Poincar\'e inequality]\label{lem:PositiveEigenvalues}
        Consider a symmetric kernel $K$ with $$\inf_{x\in\Omega}\|K(x,\cdot)\|_{L^1(\Omega)}=\xi>0.$$ Suppose also that there exists a  positive $\delta$ so that $B_\delta (x)\subset supp(K(x,\cdot))$ for all $x$. For a bounded domain $\Omega$, and a positive constant $c$, the quantity
        \[\inf_{u\in L^2(\Omega),\fint u_0^2=0}\frac{c\int_\Omega\int_\Omega K(x,y)(u(x)-u(y))^2dxdy}{\|u\|_{L^2}^2}>0. \]
    \end{lemma}
    With this Lemma, we can prove the convergence result with appropriate $K$ and $\rho$. The assumptions on $K$ are explained further in \cref{SEC:StationarySolutions}. The proof of this Theorem can be found in \cref{APP:Proofs} as Theorem \ref{appthm:conv}.
    \begin{theorem}[Convergence of solutions in a special case]\label{thm:conv}
        Let $K(x,y)$ be a symmetric kernel with $\inf_{x\in \Omega}\|K(x,\cdot)\|_{L^1(\Omega)}=\xi>0$ and $(B_\delta(x)\cap\Omega)\subset supp(K(x,\cdot))$ for some positive $\delta$ for all $x$. If $\rho$ is even, and $\rho'(z)z<0$ away from $0$, then the solution $u$ to the IVP $u_t=g[u]$ with $u(\cdot, 0)=u_0\in C^0_b(\Omega)$ satisfies 
        \[\lim_{t\to \infty}\left\|u(\cdot,t)-\fint_\Omega u_0\right\|_{L^2(\Omega)}= 0. \]
    \end{theorem}

    In this section, we have shown that this model has unique solutions and, for particular $\rho$, those solutions exist globally and are as regular as we can expect. Unlike local diffusion models, nonlocal diffusion problems, {corresponding to integrable kernels,} do not exhibit a smoothening of initial data in general, so we suspect that, to gain higher regularity, the regularity of the initial data would have to be increased. Moreover, these solutions are dissipative and can be shown, in a particular special case, to converge to a limit which is a constant in space.  

    This analysis is important for two main reasons. The first reason is that the regularity results make it possible for us to develop numerical methods which are sure to converge. These numerical methods are described in \cref{SEC:NumericalResults} and make it possible for us to simulate solutions and conduct the numerical experiments found in  \cref{SEC:NumericalExperiments}.    
    
    The second reason relates specifically to the use of this type of dynamical system for the understanding of strategic gradients under coordination dynamics. The result of \cref{THM:Regularity} tells us that the rate of strategic change across space is bounded by $(L_0+CT)e^{cT}$ on the time interval $[0, T]$. As the time interval gets longer, our bound on this rate of change gets worse. However, we know that this bound is far from tight because \cref{thm:conv} tells us that particular solutions act in a manner similar to linear nonlocal diffusion equations, which have bounded total variation and converge to limits.
    
    From a nonlocal equations perspective, these results are interesting because they give us more insight into the nonlinear class of nonlocal diffusion problems. The results of this section mildly extend results from Cafarelli \cite{Caffarelli2011}, Medvedev \cite{Medvedev2014}, and Andreu-Vaillo \cite{NonlocalDiffusionProblems}, and they give us insight into the ways that nonlinear nonlocal diffusion compares to its linear counterpart. From the game-theoretic perspective, these results are of interest because they give us insights into coordination and other similar games. Coordination is a centralizing process, meaning that strategies close to the mean should become more popular. This understanding of coordination can give the impression that strategic gradients should always collapse as the population tends towards a strategic mean, as they do in the case of \cref{thm:conv}. However, we will see later in \cref{SEC:NumericalExperiments} that when the conditions of this theorem are not met, singularities start to form in the weak derivative $D_xu(x,t)$ at the rate described in \cref{THM:Regularity}.
    This means that not only are strategic gradients not necessarily suppressed by coordination, but they can even be promoted and grow at the rate described in \cref{THM:Regularity}. 

    \section{Analytical examples}\label{SEC:AnalyticalExamples}
    Having shown that solutions to the IVP without boundary conditions exist and are unique, we now turn our attention to the behavior and properties of solutions. We begin by showing several examples wherein we can write down solutions easily in \cref{tab:analytical examples}. These problems can be solved easily by observing how the mean of $u(x,t)$ changes (or does not change), then reducing the problem to a pointwise ODE comparing $u(x,t)$ to the mean of $u(x,t)$. The full descriptions of these analytical examples can be found in the supplement. 

    \begin{sidewaystable}
        \centering
        \caption{There are four easy analytical solutions posed here. In the unstructured coordination model, we see clearly that the only rest point of the system is consensus and that it is stable. In the unstructured anti-coordination case, we see that again consensus is the only rest point, but it is unstable. In the case of directional selection, we see that there are no stable rest points. Finally, under unstructured coordination with directional effects, we see again that there are no rest points. Although every individual will, in the limit, play the same strategy, that strategy is moving towards infinity linearly.}\label{tab:analytical examples}
        \begin{tabular}{m{2.3cm}|m{2cm}|m{1.5cm}|m{5.7cm}|m{6.3cm}}
        Example& $K(x,y)$& $\rho(z)$&Evolution Equation and Solution&Notes\\
        \hline
            Unstructured Coordination&$\frac{1}{vol(\Omega)}\chi_{\Omega}(y)$&$-z^2$  &$\frac{\partial}{\partial t}u(x,t)=\fint_\Omega u(y,t)dy-u(x,t)$\newline\newline$u(x,t)=e^{-t}\left(u_0(x)- \bar{u}\right)+\bar{u}$& Under these dynamics, $\bar u=\fint_\Omega u(x,t)dx$ (the average of $u(x,t)$ over the domain $\Omega$) is constant. This is pure centralizing selection. Every individual conforms to the mean at an exponential rate. \\
            \hline
            
            Unstructured Anti-coordination &$ \frac{1}{vol(\Omega)}\chi_{\Omega}(y)$&$ z^2$  & $\frac{\partial}{\partial t} u(x,t)=u(x,t)-\fint_\Omega u(y,t)dy$\newline\newline$u(x,t)=e^{t}\left(u_0(x)- \bar{u}\right)+\bar{u}$& The solution to the anti-coordination game is the backwards time solution to the forward coordination game. This is a selection for the extremes of the distribution, and every individual will flee the mean at an exponential rate. \\
            \hline
            Structured asymmetric Toeplitz game&$K(x,y)$&$cz$&$\frac{\partial}{\partial t}u(x,t)=\int_\Omega cK(x,y)dy$\newline\newline$u(x,t)=u_0(x)+tc\|K(x,\cdot)\|_{L^1(\Omega)}$& When $\rho'$ is even, the integral operator acts as an advection term rather than a diffusion term. In this way, we get a directed selection model instead of a centralizing selection model.\\
            \hline
            Directed Centralizing model & $\frac{1}{Vol(\Omega)}\chi_{\Omega}(y)$& $\frac{-z^2}{2}+cz$&$\frac{\partial}{\partial t}u(x,t)=u(y,t)-u(x,t)+cdy$\newline\newline$u(x,t)=e^{-t}(u_0(x)-\bar u)+\bar u+ct$& The solution appears to be a linear combination of examples 1 and 3. However, this is misleading. This only happens because the integral operator is invariant under vertical translation. 
        \end{tabular}
    \end{sidewaystable}

    \section{Numerical results}\label{SEC:NumericalResults}
    The examples in the previous subsections help us to understand the kinds of behavior we may expect from these kinds of games, but we cannot, in general, solve the IVP analytically. However, because of \cref{THM:Regularity}, we can use numerical methods to find solutions, at least in the coordination case. 

    Let us begin on the unit square $\Omega=\prod_{i=1}^n[0,1]$ in $\mathbb{R}^n$ and let $\Omega_T=\Omega\times(0,T)$. If $u_t=g[u]$ on $\Omega_T$ and $u(x,0)=u_0\in C^{0,1}(\Omega)$ then we will try to approximate $u$ with the grid function $w\in \mathcal{V}(\Omega_T^{(h,\tau)})$. Here $\mathcal{V}$ is the set of grid functions which are defined on a discretization of $\overline{\Omega_T}$, 
    
    \begin{equation}\label{hypercube}
    \Omega_T^{(h,\tau)}=\Omega^{h}\times \{\tau l\}_{l=0}^\frac{T}{\tau}=\prod_{i=1}^n\left(\{hk\}_{k=0}^\frac{1}{h}\right)_i\times \{\tau l\}_{l=0}^{T/\tau}.
    \end{equation}
    In particular, the set of grid functions we are interested in is $\mathcal{V}(\Omega^{(h,t)}_T)=\{v; v:\Omega^{(h,t)}_T\rightarrow \mathbb{R}\}$ and $\mathcal{V}(\Omega^h)=\{v;v:\Omega^h\to \mathbb{R}\}$. These grid function spaces are different ways of imagining $\mathbb{R}^{\frac{n}{h}\cdot \frac{T}{\tau}}$ and $\mathbb{R}^{\frac{n}{h}}$, respectively. We use this reimagining so that the comparisons between $v\in \mathcal{V}(\Omega^{(h,\tau)}_T)$ and $u\in C^{0,1}(\Omega_T)$ are more natural. 
    
    Let $\pi^h:C^{0}(\Omega)\rightarrow \mathcal{V}(\Omega^h)$ be the operator that takes a function on $\Omega$ and returns the grid function that is equal to the input function at every point of the grid $\Omega^h$. In the present study, we will deal only with the case that $K\in C^0_b(\Omega;C^{0,1}_0)$ so $\pi^hK(x,\cdot)\in \mathcal{V}(\Omega^h)$ for every $x\in \Omega^h$. A generalization is possible but is not immediately necessary for the main results of the paper. There are some cases of $K\in C^0_b(\Omega;L^1(\Omega))$ for which this method is not appropriate. (For example $K(x,y)=d(x,\partial \Omega)^{-n}J(\frac{x-y}{d(x,\partial\Omega)})$ where $J\in L^1(B(0,1))$).

    Let $w(\cdot, 0)=\pi^h u_0$ on $\Omega^h$, and compute for an $\mathbf{x}\in \Omega^h$,
    \begin{equation}\label{ForwardEuler}
    w(\mathbf{x},t_{i+1})=w(\mathbf{x},t_i)+\tau \sum_{\mathbf{y}\in ^-\Omega^h}K(\mathbf{x},\mathbf{y})\rho'(w(\mathbf{x},t_i)-w(\mathbf{y},t_i))h^n
    \end{equation}
    where $^-\Omega^h=\prod_{i=1}^n(\{hk\}_{k=0}^{\frac{1}{h}-1})_i$. We can show for this particular domain that the method \eqref{ForwardEuler} is consistent and convergent. The consistency result is stated below, and the proof is shown in Appendix \ref{APP:Proofs} under \cref{APPLEM:Consistency}.

    \begin{lemma}[Consistency of forward Euler scheme]\label{LEM:Consistency}
        The numerical scheme \eqref{ForwardEuler} is consistent to order $\tau+h$ for the nonlocality \eqref{nonlocality} with a bounded $w\in C^{1,1}$ with no boundary data when $K\in C^0_b(\Omega;C^{0,1}_0)$ with a uniform bounds on $\|K(x,\cdot)\|_\infty$ and on the Lipschitz constant for $K(x,\cdot)$.
    \end{lemma}
    
    For a linear problem, the consistency from \cref{LEM:Consistency} and stability from a discrete maximum principle would complete the proof of convergence. The discrete maximum principle holds for this method and is included in the supplement. 
    However, it only provides a sanity check that this method makes sense when $\tau$ and $h$ are sufficiently small. Because the equation is nonlinear, we cannot use the Lax principle and the discrete maximum principle 
    to prove convergence. Instead, we will recapitulate in the discrete case the argument about the Lipschitz continuity of $\mathcal{G}^h$ in $\Omega^{(h,\tau)}_T$. The argument itself is found in Appendix \ref{APP:Proofs} under \cref{APPLEM:DiscreteLipschitz}. 

    \begin{lemma}[Lipschitz continuity of $\mathcal{G}^h$]\label{LEM:DiscreteLipschitz}
        For any  bounded (in the sup norm sense) subset of $X_R:=\{w\in \mathcal{V}(\Omega^{(h,\tau)}_T);\|w\|_{\ell^\infty}\leq R\}$, there is a $C^\mathcal{G}\geq 0$ such that for every $w_1, w_2\in X_{R}$
        
       \[\|\mathcal{G}[w_1](\cdot, t_i)-\mathcal{G}[w_2](\cdot, t_i)\|_{\ell^\infty\left(\Omega^h\right)}\leq C^\mathcal{G}\|w_1(\cdot, t_i)-w_2(\cdot, t_i)\|_{\ell^\infty\left(\Omega^h,\right)}\]
        so long as $K\in C^0_b(\Omega;C^{0,1}_0)$ with a uniformly bounded $\|K(x,\cdot)\|_\infty$ and $\mathcal{G}$ is defined as in \eqref{quadrature}.
    \end{lemma}

    With these two results, we can show that the numerical scheme \eqref{ForwardEuler} is convergent in the case that $K\in C^0_b(\Omega;C^{0,1}_0)$ with uniformly bounded $\|K(x,\cdot)\|_\infty$ and with a global Lipschitz constant. The general case wherein $K\in C^0_b(\Omega;L^1(\Omega))$ requires a choice of an approximate $K$ that is bounded and defined everywhere to manage the case in which $K(x,\cdot)$ has a singularity for some $x\in \Omega$, which makes this method inappropriate in some cases. There are more sophisticated numerical methods for nonlocal equations that are equipped to handle the more general case \cite{D’Elie2020}, but we do not use them in the analysis in \cref{SEC:NumericalExperiments}. For this reason, we only consider the finite difference method in this nonlinear setting, which mildly extends the results from Du et al. \cite{Du2019} for the linear setting.
    
    \begin{theorem}[Convergence of the forward Euler scheme]\label{THM:Convergence}
    In a domain with discretization \eqref{hypercube}, the IVP $u_t=g[u]$ on $\Omega$, with $u(x,0)=u_0\in C^{0,1}_b(\Omega)$, and with $\rho$ satisfying \eqref{coordinationcondition}, the numerical scheme \eqref{ForwardEuler} is convergent in the case where $K\in C^0_b(\Omega;C^{0,1}_0)$ with uniformly bounded $\|K(x,\cdot)\|_\infty$ and a uniform global Lipschitz constant.
    \end{theorem}

    This proof, which is found in the appendix \ref{APP:Proofs} under \cref{APPTHM:Convergence}, was done for the unit cube, but each part is easily generalizable to any bounded domain with constants which may depend on the volume of the domain itself.  Now that we are certain that this numerical scheme appropriately approximates solutions to the IVP, we can show several examples.
    \begin{figure}
        \centering
        \includegraphics[width=0.8\linewidth]{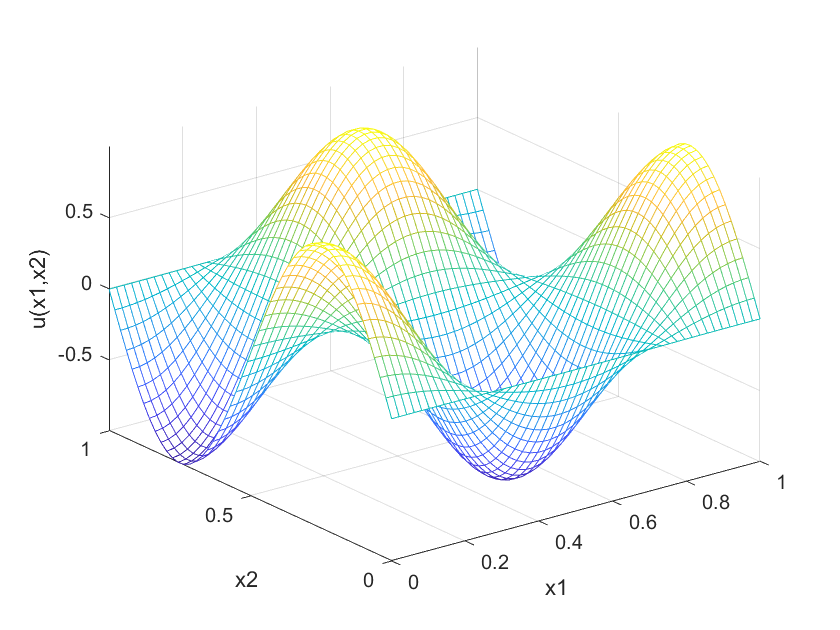}
        \includegraphics[width = 0.8\linewidth]{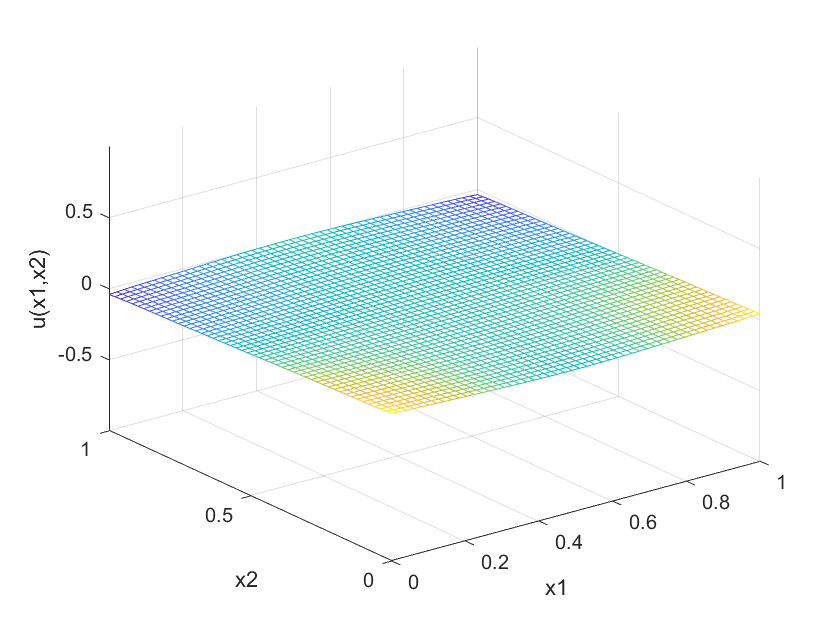}
        \caption{The Initial condition (\textbf{above}) and solution after $T=10$ (\textbf{below}) to the IVP $u_t=g[u]$ approximated by the numerical scheme \eqref{ForwardEuler}. Here the recognition function, $\rho=e^{-z^2}$, has non-zero derivative on all of $\mathbb{R}\setminus\{0\}$}
        \label{fig:NCPT}
    \end{figure}
    \begin{figure}
        \centering
        \includegraphics[width=0.8\linewidth]{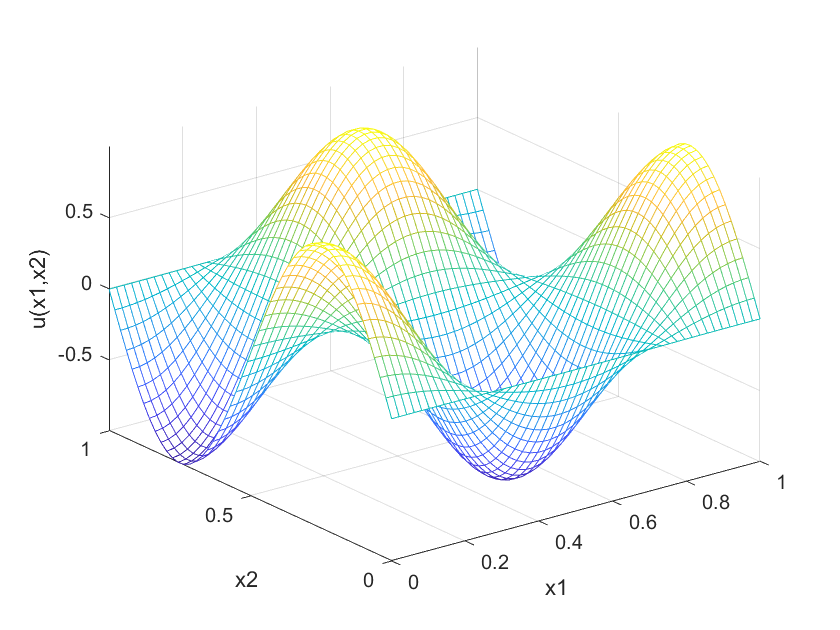}
        \includegraphics[width = 0.8\linewidth]{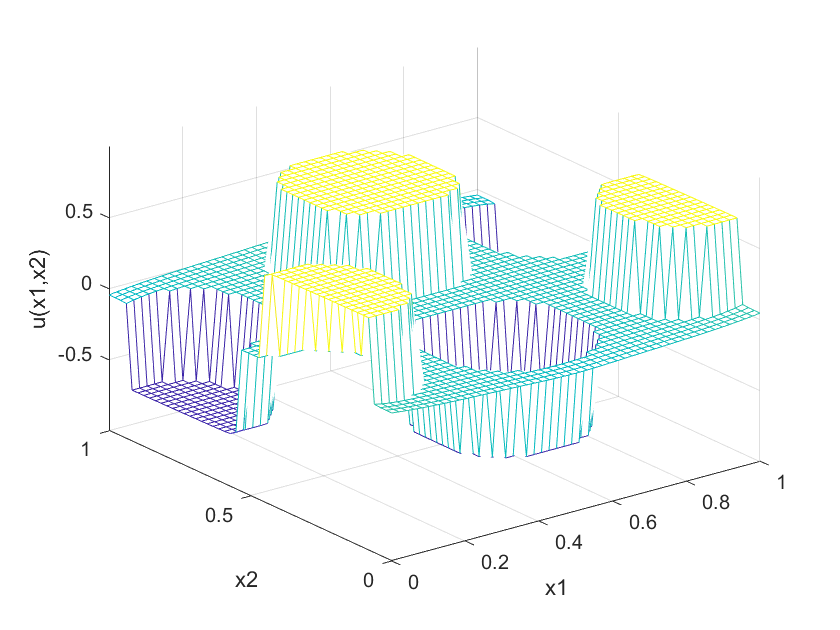}
        \caption{The Initial condition (\textbf{above}) and solution after $T=10$ (\textbf{below}) to the IVP $u_t=g[u]$ approximated by the numerical scheme \eqref{ForwardEuler}. Here the recognition function is the bump function, $\rho = e^{-1/(1-(2z)^2)}$ in $B_{1/4}(0)$ and 0 outside $B_{1/4}(0)$.}
        \label{fig:CPT}
    \end{figure}

    Interrogating the system with numerical methods allows us to observe some curious properties of solutions to the coordination equation. One such interesting observation is the qualitatively different behavior seen for solutions that depend on the support of $\rho$, as in \cref{fig:NCPT} and \cref{fig:CPT}. This will be discussed further in the modeling results.

    \section{Numerical experiments}\label{SEC:NumericalExperiments}

    The most striking element of the numerical methods, as shown in \cref{fig:NCPT} and \cref{fig:CPT}, is the formation of discontinuities in the limit as time tends toward $\infty$. We begin examining this and the possible existence of strategic gradients through the use of simulation. We conducted several numerical experiments employing the numerical methods described in \cref{SEC:NumericalResults} to test the relationship between $\rho, K$, and the formation of discontinuities. The code used to run the experiments can be found under the citation \cite{McAlister2025}. The first experiment considered the interval $I=[-1/2,1/2]$ and the initial data $u_0(x)=lx$ where $l$ was varied from $0$ to $4$. With a Gaussian kernel $K(x,y)=\frac{1}{s\sqrt{2\pi}}e^\frac{{(x-y)}^2}{-2s^2}$ (where the kernel concentrates around $x$ as $s$ increases) and a compactly supported recognition function $\rho(z) = e^{\frac{-1}{1-(z/r)^2}}$ when $|z|<r$ and $0$ otherwise, the experiment examines the profile of the solution after the IVP has been solved until $t=20$. The system parameters were chosen as $s=0.5$ and $r=0.2$. The results are shown in \cref{fig:NM1}
    \begin{figure}
        \centering
        \includegraphics[width=\linewidth]{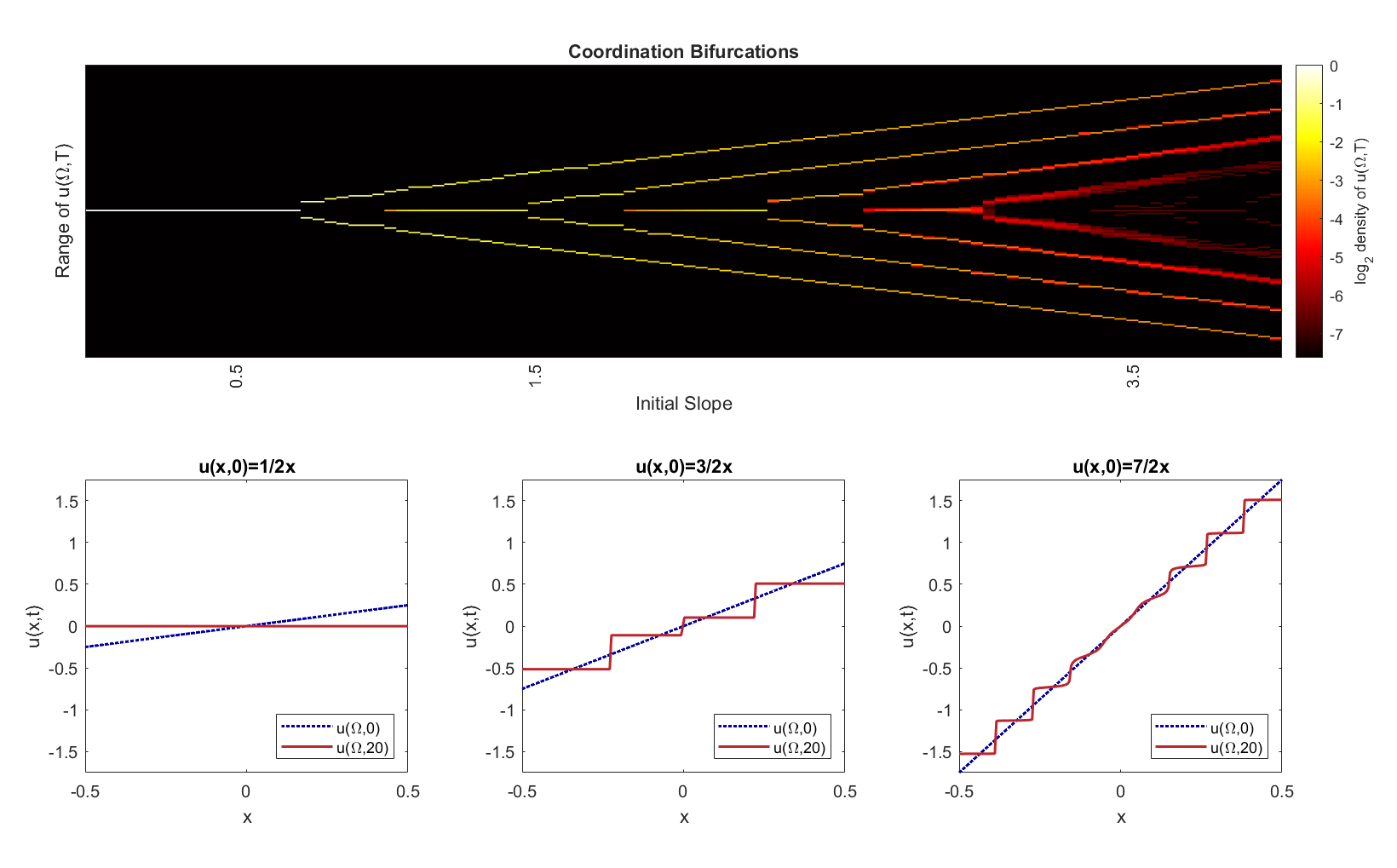}
        \caption{\textbf{Top} A bifurcation-type diagram showing that as the slope of the initial data increases, the distribution of the solution $u(x,20)$ changes. On the far left, when the slope is low, the image is entirely distributed at $u(x) = 0$. As the initial slope increases, there are more bands. The colors on the heat map represent the $log_2$ of the density of $u(x,20)$.
        \textbf{Bottom} Three examples of initial conditions (dotted in blue) and solutions at time $t=20$ (solid in red). The corresponding slices in the bifurcation diagram are labeled in order on the $x$ axis of the heatmap above.}
        \label{fig:NM1}
    \end{figure}

    An interesting observation about the number of bands present at time $t=20$ is the approximately equal spacing between the non-central bands. All of the non-central bands are separated by $0.4$.  When there are two central bands, they appear to be at least a distance of $0.2$ from one another (the radius of the support of $\rho$). Moreover,  it is exactly when the distance between the bands exceeds $0.4$ that the single central band between them emerges. Following this pattern, we can predict, for a linear initial condition, how many bands there will be. No two bands can be closer than 0.2 (the radius of the support $\rho'$) and no two bands can be further than 0.4. For this setting in particular, at most 2 spaces between bands will be less than 0.4. Therefore, we can hypothesize that the number of bands $b$ is bounded by 
    \[b\leq \frac{R}{2r}+1\]
    where $R$ is the length of the range of $u_0$ (i.e., the oscillation of $u_0$ over $\Omega$) and $r$ is the radius of the support of $\rho'$. This bound is only an observation. A real bound can be found later in \cref{SEC:StationarySolutions}.

    Another interesting observation that is not explored further in the present study is the fact that the bands seem to start to form from the outside, then proceed inward. In the case where $l=3.5$ in \cref{fig:NM1}, we can see that the most central bands have not formed by the time $t=20$. This points to a limited speed of propagation. When the solution $u(x,t)$ is odd about a point $x_0$ and the kernel is even, then $u_t(x_0,t)=0$. In the exact center, the solution is odd for all time in each of these experiments. Close to the center, the solution ``almost" has odd symmetry, and near the edge of the range, the solution has no odd symmetry. For points near the center, the lack of odd symmetry far away is hidden by the fact that the kernel $K$ is small at that distance. For points close to the edge of the range, the immediate lack of odd symmetry results in an immediate collapse into discrete strategies. This collapse ruins the odd symmetry closer to the center, and so the collapse propagates to the center (\cref{fig:NM4}).
    \begin{figure}
        \centering
        \includegraphics[width=\linewidth]{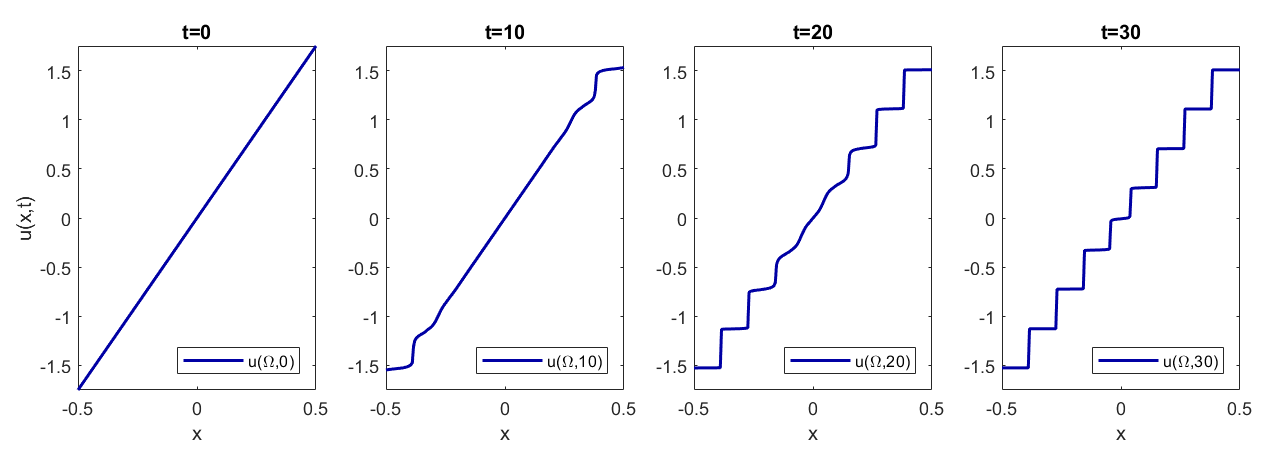}
        \caption{Four time slices of the solution to $u_t=g[u]$ when $u_0=7/2x$ on the domain $\Omega =[-2,2]$. The collapse into discrete strategies starts at the edge of the domain and propagates towards the center.}
        \label{fig:NM4}
    \end{figure}

    We use the next two numerical experiments to test the hypothesis about the minimum band separation. Using a sigmoid function instead of a linear function for the initial data, we can investigate the number of bands that emerge while the range of the initial data is constrained. For the same kernel and recognition function as before, we run the same experiment with $u_0=(1+\exp(-lx))$ for $l\in [0,15]$. The results are shown in \cref{fig:NM2} 

    \begin{figure}
        \centering
        \includegraphics[width=\linewidth]{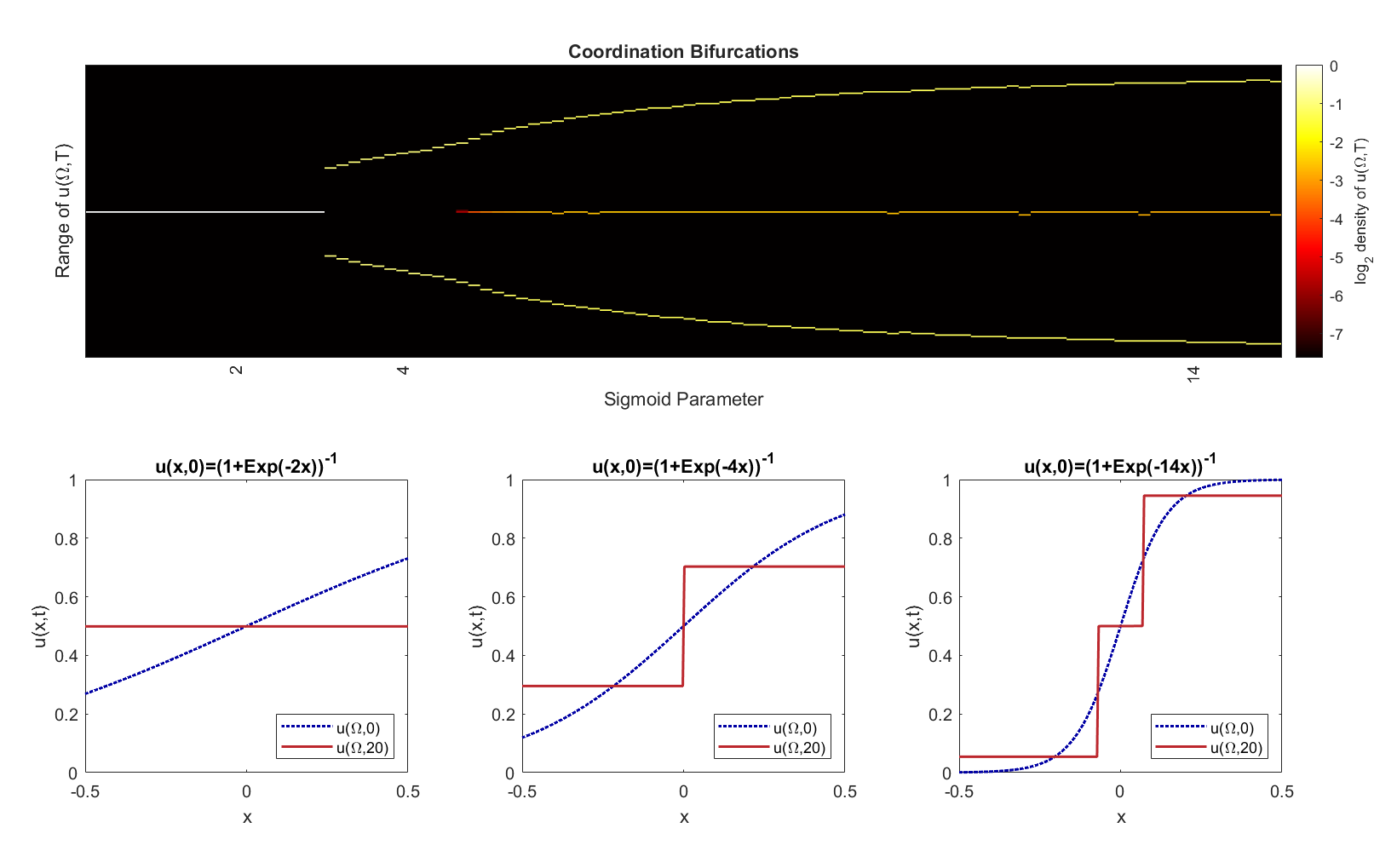}
        \caption{\textbf{Top} A bifurcation-type diagram showing that as the sigmoid parameter of the initial data increases, the distribution of the solution $u(x,20)$ changes. On the far left, when the sigmoid parameter is low, the image is entirely distributed at $u(x)=0$. As the sigmoid parameter for the initial data increases, there are two, then later three bands, but as the sigmoid parameter tends towards its maximum, the bands level out. The colors on the heat map represent the $log_2$ of the density of $u(x,20)$.
        \textbf{Bottom} Three examples of initial conditions (dotted in blue) and solutions at time $t=20$ (solid in red). The corresponding slices in the bifurcation diagram are labeled in order on the $x$ axis of the heatmap above.}
        \label{fig:NM2}
    \end{figure}
    Again, we notice that the number of bands increases with the range of the initial data. The relationship from the first numerical experiment holds in this case too. Because the range of the initial data is constrained to $(0,1)$, we never see more than three bands emerge. This is precisely what is predicted from the hypothesis previously stated. 

    The sigmoid example also shows us that it is not a critical gradient threshold that results in discontinuities in the limit. The derivative is greatest at $x=0$ initially in every example. However, the discontinuities do not always emerge at the center. When the sigmoid parameter is greater than $5$ and there are three bands present, the central band is centered around $x=0$ and the discontinuities are present on either side of this band. Therefore, trying to predict the location of bands present at equilibrium (assuming a solution converges to a single stationary solution) is not as simple as finding where the gradient might surpass a certain threshold. 

    To further test this hypothesis, we vary the support of the recognition function rather than the initial data. In the third numerical experiment, the parameter $r$ in the recognition function 
    \[\rho(z)=\begin{cases}e^\frac{-1}{1-(z/r)^2} & |z|<r\\
    0&|z|\geq r\end{cases}\]
    is varied from $0$ (the limit in which $\rho(z)=\chi_{\{0\}}(z)$) to $0.8$. The results are shown in \cref{fig:NM3}. 
    \begin{figure}
        \centering
        \includegraphics[width=\linewidth]{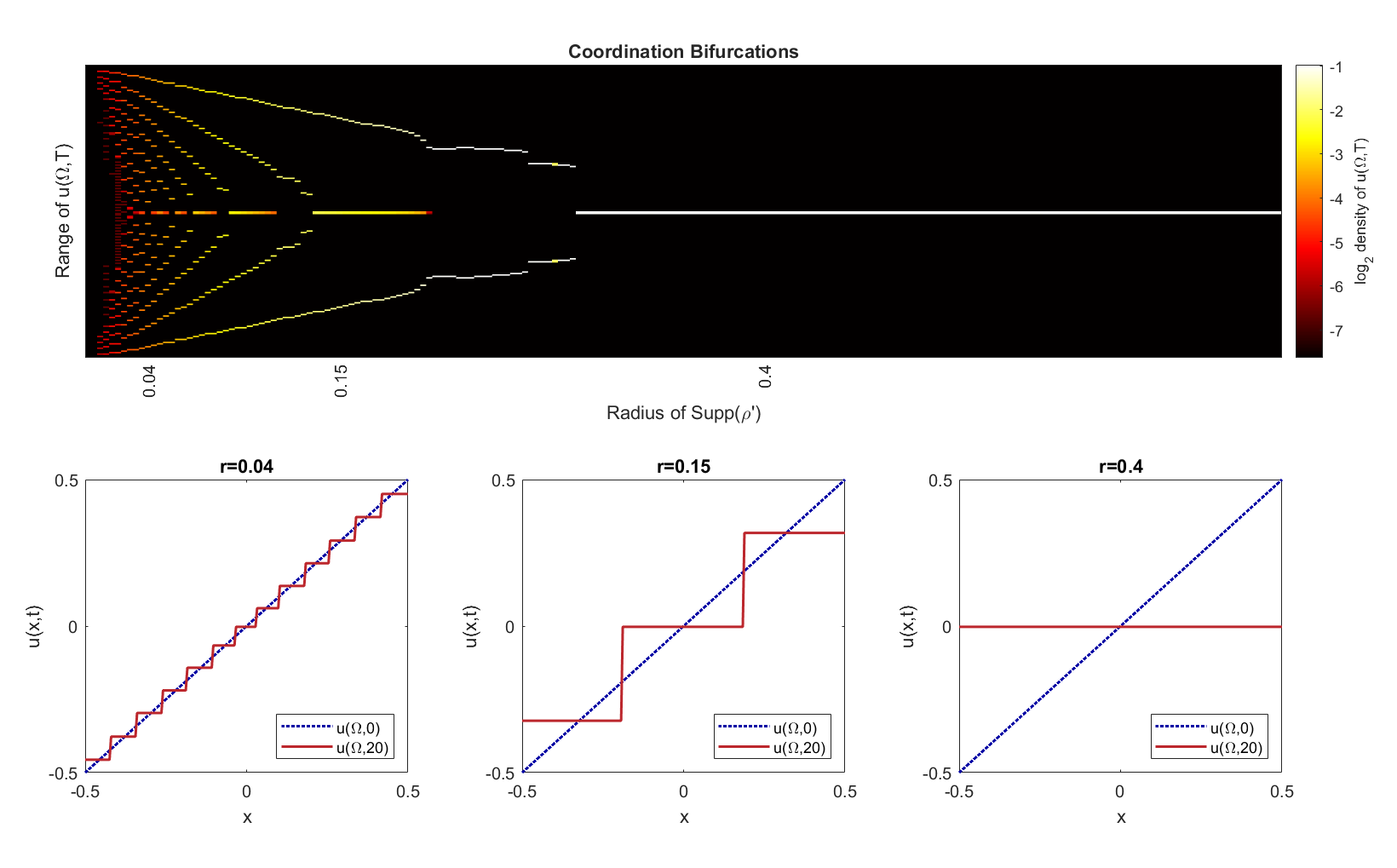}
        \caption{\textbf{Top} A bifurcation-type diagram showing that as the support of the recognition function changes, the distribution of the solution $u(x,20)$ changes. On the far left, when the support of $\rho$ is near $0$, there are many bands. As the support of the recognition function increases, there are fewer and fewer bands until the solution collapses to $u(x)= 0$. The colors on the heat map represent the $log_2$ of the density of $u(x,20)$.
        \textbf{Bottom} Three examples of initial conditions (dotted in blue) and solutions at time $t=20$ (solid in red). The corresponding slices in the bifurcation diagram are labeled in order on the $x$ axis of the heatmap above.}
        \label{fig:NM3}
    \end{figure}
The results of the experiment are predicted exactly by our hypothesis. Again, in every case, the number of bands is bounded above by $(R/2r)+1$.

    These numerical experiments obviously fall far short of a real proof. A weaker result is proved later in \cref{SEC:StationarySolutions}. However, we will need to be careful how we fully discuss these results because we must avoid the assumption that solutions converge to a single stationary solution. 
    
    It has not been proven that solutions to the IVP will converge, even pointwise to a limit, in general, but we do know that, if they do converge, they will clearly converge to a solution to the problem $g[u]=0$ in $\Omega$. Without imposing boundary data, solutions to this stationary problem obviously exist in $ L^\infty(\Omega)$. (e.g. $u\equiv 0$). The existence of non-trivial solutions and solutions with boundary data in this elliptic case are not discussed in the present study. We will, however, discuss several results about stationary solutions and show how these stationary results support the findings of the numerical experiments. 

    \section{Stationary solutions}\label{SEC:StationarySolutions}
    \subsection{Characterizing stationary solutions}\label{subsec:characterization}
    Recall that, as the game was introduced in \cref{SEC:Extension}, we are not only interested in the dynamic results but also in the classical game with no time component. Results about stationary solutions in the dynamic game can reveal more effective ways of searching for Nash equilibria in the classical game. Likewise, results about the classical game can give us insights into the observations we make about the dynamic game. If $u$ is a Nash equilibrium in the classical game, by definition it will have the property that 
    \begin{equation}\label{NECondition}
    J[u]:=\inf_{(x,s)\in \Omega\times \mathbb{R}}\{w(x|u)-w(x|u+t\chi_{\{x\}})\}=0.\end{equation}
    This is exactly the condition that each player is playing their best response. 
    \begin{proposition}\label{PROP:StationarySolutions}
        If $u$ is a pointwise Nash Equilibrium to the game with players in $\Omega\subset\mathbb{R}^n$, strategies in $\mathbb{R}$ and payoffs as in \eqref{ctsStratPayoff}, then $u$ is necessarily a stationary solution in the system $u_t=g[u]$, so $g[u]=0$ pointwise. If $u$ is a Nash equilibrium almost everywhere, then $g[u]=0$ almost everywhere.
    \end{proposition}
    \begin{proof}
        In the same way as in the proof of \cref{PROP:ModelWellFounded} let $S_x(h):=\int_\Omega K(x,y)\rho(u(x)+h-u(y))dy$ and note that, so long as $\rho\in C^{1,1}(\mathbb{R})$ then $S_x(h)\in C^{1,1}(\mathbb{R})$ and $\frac{d}{dh}S_x(0)=\int_\Omega K(x,y)\rho'(u(x,t)-u(y,t))$. Because $u(x)$ is a best response to $u$, it is certain that $S_x(h)$ attains a global maximum at $h=0$, and because the strategic domain is unbounded, we know that $\frac{d}{dh}S_x(0)=\int_\Omega K(x,y)\rho'(u(x)-u(y))dy=0$. This is true for every $x$ (or almost every $x$) and thus $g[u]=0$ in $\Omega$ or $g[u]=0$ almost everywhere in $\Omega$.
    \end{proof}

    It should be noted that the opposite direction does not hold. It is easy to construct a stationary solution which is not a Nash equilibrium almost everywhere.

    \begin{example}
        Consider a coordination game where $\rho\in C^{1,1}(\mathbb{R})$ satisfies $\rho(0)>0, \rho'(0)=0,$ and $supp(\rho')\subset[-a,a]$ and where $K(x,y)=1$ if $|x-y|<1$ and $0$ otherwise.
        
       \[u=\begin{cases}
        0&|x|>\delta\\
        2a&|x|\leq \delta
        \end{cases}\]
        for $\delta<\frac{1}{2}$ is a stationary solution because $g[u]\equiv 0$. However, $u$ is not a Nash equilibrium almost everywhere because $w(0|u)=0$ but, for any $|x_0|<\delta$, if $\tilde{u}=u-2a\chi_{\{x_0\}}$, then $w(x_0|u)-w(x_0|\tilde u)=2\rho(0)(\delta)-2\rho(0)(1-\delta)<0$, so there is a set of positive measure on which the Nash condition does not hold.
    \end{example}

    Unlike the result of Friedman and Ostrov \cite{FRIEDMAN2010}, stationary solutions are not even necessarily Local Nash Equilibria almost everywhere (LNE), which are defined by the existence of a $\delta$ so that almost no player can improve their payoff by changing their strategy by less than $\delta$. 

    \begin{example}\label{ex:nonLNE}
    Consider the simple case where $\Omega=(-1,1)$ and $K(x,y)=1$. Construct a $\rho$ so that $\rho$ satisfies \eqref{coordinationcondition},  $supp(\rho')=(-a,a)$, $|\rho'(z)|<\xi$ for $|z|<\sigma a$ and $|\rho'(z)|>(1-\xi)$ for $|z|>(1-\sigma)a$ for some $\sigma\in (0,1)$ and for some $\xi<\frac{1}{3}$.
       \[u=\begin{cases}
           -a&x\leq -\xi\\
           0&x\in (-\xi,\xi)\\
           a &x\geq \xi 
       \end{cases}\]
       is a stationary solution because $g[u]\equiv 0$. However, for any $x\in (-\xi,\xi)$ and any $\delta>0$ if $\tilde{u}=u+\delta\chi_{\{x_0\}}$ then we calculate  $w(x_0|\tilde u)-w(x_0|u)$ in the following way.
       \[\int_{-1}^{-\xi}K(x_0,y)(\rho(\delta+a)-\rho(a))+\int_{-\xi}^\xi K(x_0,y)(\rho(\delta)-\rho(0))+\int_\xi^1 K(x_0,y)(\rho(a-\delta)-\rho(a))\]
       If $supp(\rho')=(-a,a)$ then $\rho(a+\delta)=\rho(a)$. Recalling that $K(x,y)=1$ we know our change in payoff is 
        \[   
           \int_{-\xi}^\xi (\rho(\delta)-\rho(0))+\int_\xi^1(\rho(a-\delta)-\rho(a))
        \]
        By our construction of $\rho$ we know that $-2\xi^2\delta<2\xi (\rho(\delta)-\rho(0))$ when $\delta$ is smaller than $\sigma$ and that $(1-\xi)\xi\delta<(1-\xi)(\rho(a-\delta)-\rho(a))$. Therefore, the change in payoff is bounded below by $\delta(\xi -3\xi^2)>0$ and so for any $\delta$, player $x_0$ will improve its payoff by changing its strategy by less than $\delta$. This means that there is a set of positive measure for which the Nash condition does not hold and thus $u$ is not a Local Nash Equilibrium almost everywhere.  
         
    \end{example}

    The best that we can say is that the set of stationary solutions is dense with LNE in  $L^\infty(\Omega)$, which we will prove after the stationary solutions are characterized. Even without an exact equivalence between stationary solutions and Nash equilibria, understanding stationary solutions will still inform our understanding of the classical game. In the following theorems, we characterize stationary solutions in the case that $\Omega$ is bounded and there is no boundary data. This result will show that if a system is interacting under pure coordination dynamics, where every individual interacts some amount with its immediate neighborhood, then strategic gradients can never appear at equilibrium. The description of a stationary solution is sensitive to the properties of $K$, so we present two theorems with different assumptions on $K$. 
    
    In the more general case, we assume that $K$ satisfies the following conditions: (a) $K\in C^0_b(\overline \Omega;L^1(\Omega))$, (b) $K$ is symmetric, (c) there is a $\xi>0$ so that $\inf_\Omega\|K(x,\cdot)\|_{L^1(\Omega)}\geq \xi$, and (d) there is a $\delta>0$ so that $B_\delta(x)\subset supp(K(x,\cdot))$ for all $x\in \Omega$. Assumption (a) is a strengthened form of the standard assumption from \cref{THM:ShortTimeExistenceAndUniqueness}. We assume the $L^1$ continuity of $K$ extends to the boundary of $\Omega$ as well. (b) avoids the case wherein interactions are unidirectional. (c) and (d) are necessary to avoid the case where individuals don't interact with their immediate neighborhood. The importance of this assumption is discussed later in an example. In this general case, we can show that all stationary solutions are piecewise constant functions (which are also known as simple functions) of the form 
    \[u=\sum_{k=1}^m c_k\chi_{S_k}\]
    where $\cup_{k=1}^m S_k=\Omega$ and $S_k\cap S_j=\emptyset$. This result is stated here and proved in Appendix \ref{APP:Proofs} as \cref{APPTHM:StationarySolution2}.

    In the second version of the theorem, we assume that $K(x,y)>\lambda$ on $\Omega\times\Omega$. Under this strengthened hypothesis, stationary solutions are still simple functions, but we can also say that there is a minimum separation between the constant pieces. This stronger result is stated below and is proved in Appendix \ref{APP:Proofs} as \cref{APPTHM:StationarySolution}. 
    
    \begin{theorem}[Stationary solutions are simple functions]\label{THM:StationarySolutions}
        Let $\Omega$ be a bounded domain in $\mathbb{R}^n$. Let $K(x,y)\in C_b^0(\overline\Omega;L^1(\Omega))$ symmetric with $\inf\|K(x,\cdot)\|_{L^1(\Omega)}=\xi>0$ and there a $\delta$ so that $B_{\delta}(x)\subset supp(K(x,\cdot))$ for all $x$.  Finally, let $\rho'$ satisfy \eqref{coordinationcondition} and have support $(-a,a)\subset \mathbb{R}$, with the assumption that $\rho'$ has only one zero in this interval at $z=0$ (see \cref{fig:rhoprime}) If $u\in L^1(\Omega)\cap L^\infty(\Omega)$ is a solution to $g[u]=0$ in $\Omega$, then it is a simple function of the form 
        \[u=\sum_{k=1}^mc_k\chi_{S_i}\]
        for some finite $m$ where $\cup_{k=1}^m S_k=\Omega$ and $S_k\cap S_j=\emptyset$.
    \end{theorem}

     \begin{theorem}[Stationary solutions when $\Omega\subseteq  supp K(x,\cdot)$]\label{THM:StationarySolution}
        Let $\Omega$ be a bounded domain in $\mathbb{R}^n$. Let $K(x,y)\in C_b^0(\Omega;L^1(\Omega))$ so that $\lambda \leq K(x,y)$. Finally, let $\rho'$ satisfy \eqref{coordinationcondition} and have support $(-a,a)\subset \mathbb{R}$, with the assumption that $\rho'$ has only one zero in this interval at $z=0$ (see \cref{fig:rhoprime}). If $u\in L^1(\Omega)\cap L^\infty(\Omega)$ is a solution to $g[u]=0$ in $\Omega$, then it is a simple function of the form 
        \[u=\sum_{k=1}^mc_k\chi_{S_i}\]
        for some finite $m$ where $\cup_{k=1}^m S_k=\Omega$ and $S_k\cap S_j=\emptyset$. Also, $|c_k-c_j|\geq a$ for all $k\neq j$.
    \end{theorem}

    \begin{figure}
        \centering
        \begin{tikzpicture}
            \draw[<->](-7,0)--(7,-0);
            \draw[<->](0,-3)--(0,3);

            \node (a) at (-0.5,3){$\rho'(z)$};
            \node (b) at (7,-0.5){$z$};

            \draw(-5,0)--(-5,-0.25);
            \draw(5,0)--(5,0.25);
                
            \node (c) at (-5,-0.5){$-a$};
            \node (f) at (5,0.5){$a$};

            \draw plot [smooth] coordinates {(-5,0) (-4,0.5) (-3,2) (-2,1) (-1,1.5) (0, 0) (1,-0.75) (1.5,-0.6) (2,-1) (2.5,-1.5) (3,-2) (3.5,-2)(4,-1) (5,0)};
        \end{tikzpicture}
        \caption{A diagram showing an appropriate $\rho'$ for \cref{THM:StationarySolution}. $\rho'$ must be non-zero on $(-a,a)\setminus\{0\}$ and satisfy \eqref{coordinationcondition} so that on $(-a,0)$ $\rho'$ is positive and on $(0,a)$, $\rho'$ is negative. $\rho'$ is Lipschitz continuous, but no more regularity is required for this argument.}
        \label{fig:rhoprime}
    \end{figure}

    These results give us some insight into coordination dynamics at equilibrium. Although we cannot rule out the existence of a set of players with measure zero that do not adhere to one of a finite set of strategies, we can say that, in a weak sense, at equilibrium, when no outside influences are acting on the system (i.e., no boundary data or inhomogeneity), the solution will be piecewise constant. Notice that, if $U$ is the cumulative distribution function of strategy across the entire population (i.e., $U(z)=\mu(\{x\in \Omega;u(x)\leq z\}) $ then $U$ is piecewise constant, which is exactly the result of Friedman and Ostrov in \cite{FRIEDMAN2010}.
    
    Two important corollaries are stated below. The first is key to completing \cref{lem:PositiveEigenvalues} because it implies that, even with a Kernel that is not translation invariant, the only stationary solution to the linear nonlocal diffusion equation is a constant function. The corollary requires a short proof, which is included in Appendix \ref{APP:Proofs} as \cref{appcor:linear}. The second corollary gets to the heart of the question about strategic gradients and is immediate from the results of \cref{THM:StationarySolutions,THM:StationarySolution}.
    \begin{corollary}\label{cor:linear}
        If $K$ is as in \cref{THM:StationarySolutions} and $\rho'(z)\neq 0$ away from $0$, Then if $g[u]=0$ $u$ is constant. 
    \end{corollary}
    
    \begin{corollary}\label{COR:ContinuousStationarySolution}
        If $u$ is a continuous, bounded stationary solution on a bounded domain with $\rho$ having only 1 zero in the interior of its support at $z=0$ and $K$ as in \cref{THM:StationarySolutions} or \cref{THM:StationarySolution}, then $u$ is constant.
    \end{corollary}
    
    All of this means that strategic gradients can not occur under these conditions. We cannot see a slow change in strategy across space when every player is interacting at least some small amount in their immediate neighborhood. This means that when we consider examples of sympatric speciation or dialect change across a region, we can hypothesize that there are other important elements forcing the existence of a strategic gradient. These may be the result of a different strategy revision protocol, which is unlike the one described in \cref{PROP:ModelWellFounded} on which this model is founded. It may also be because of some intrinsic benefit of a particular strategy related to the location in the relational structure.
    
    As discussed, these stationary solutions are not necessarily Nash equilibria almost everywhere or even Local Nash equilibria almost everywhere. The best that we can say is that the set of stationary solutions is dense with LNE in the $L^\infty$ sense (and the $L^1$ sense if $\Omega$ is bounded). Moreover, we can say something even stronger about this dense set of stationary solutions. Not only are they LNE, but they are also stable in the $L^\infty$ sense, meaning that even coalitions of players cannot change their strategies by less than $\epsilon$ to improve their payoff. 

    \begin{theorem}[LNE are dense among stationary solutions]\label{thm:LNE}
        Let $K$ and $\rho$ be as in \cref{THM:StationarySolutions}. For any $\epsilon$, if $u$ is a stationary solution with $g[u]=0$ in $L^1(\Omega)\cap L^\infty(\Omega)$, then there is a $v$ such that $g[v]=0$, $v$ is an LNE, $v$ is stable in the $L^\infty$ sense, and $\|v-u\|_{L^1(\Omega)}<\epsilon$.
    \end{theorem}

    So stationary solutions are something stronger than an LNE but not quite strong enough to be a Nash equilibrium. Those stationary solutions which are not equilibria must, therefore, include players who are ``stuck'' in local but not global payoff maxima. With this interpretation, we can say that all stationary solutions are Local $\epsilon-$Nash equilibria almost everywhere. This is apparent in the proof of \cref{thm:LNE}. This exceedingly weak kind of equilibrium is defined as a strategy profile for which almost no player can improve their payoff by more than $\epsilon$ by changing their strategy by less than $\delta$.  For those readers interested in this game on discrete very large graphs, this means that when $K$ is considered as a graphon, there is a graph large enough with adjacency matrix approximated by $K$ so that a discretized version of $u$ is a Local $\epsilon-$Nash equilibrium of the game on the very large graph. See the result of Carmona et al. for details \cite{Carmona2022}.  

    Even with this characterization of stationary solutions, we cannot achieve a general convergence result. This characterization helps us describe what the set of stationary solutions looks like in $L^1(\Omega)\cap L^\infty(\Omega)$. Frustratingly, even the $L^\infty$ stability of the dense set of stationary solutions cannot be used to show that solutions converge to a single stationary solution because even if solutions subsequentially converge to elements of an $\omega$-limit set, this convergence would be in $L^1(\Omega)$, not $L^\infty(\Omega)$. The best that we can do at present is note that in the case that $\rho'(z)\neq 0 $ away from $0$, our convergence result \cref{thm:conv} is consistent with our characterization of stationary solutions. In the case that $\rho'(z)$ has compact support, \cref{cor:intColapse} tells us that the total volume occupied by pairs of individuals with strategies which are in the interior of the support of $\rho'$, weighted by the strength of their interactions integrated over time, goes to zero. This is a necessary condition for convergence to a stationary solution but unfortunately not a sufficient one. 
    
    \subsection{Counter examples with weakened conditions}\label{subsec:counterex}
    We must also acknowledge that these results can be shown only with some assumptions on $K$ and $\rho$. We can give several examples of strategic gradients when an assumption is broken. An easy example is when the boundedness of the solution and of the domain is violated. There are clear examples of solutions to $g[u]=0$ which are unbounded and not piecewise constant. 
    \begin{example}\label{EX:NonconstantStationarySolution}
        Let $k(x,y)=J(|x-y|)$ be a radial, translation-invariant kernel. Let $\rho$ be even so $\rho'$ is odd. For the domain $\mathbb{R}$, the solution $u(x)=x$ satisfies $g[u]=0$. We can see this by noting that $\rho'$ is odd so for any $x$, $\rho'(u(x)-u(y))=\rho'(x-y)$ is odd in $y$. $J$ is even and translation invariant so again for any $x$ it is even in $y$. Thus $g[u](x)=\int_{\Rn}J(x-y)\rho'(x-y)dy=0$ for all $x\in \Omega$. 
    \end{example}

    Although an unbounded patch is never satisfied, it can be approximately satisfied when the scale of individuals is much smaller than the scale of the patch. This assumption, and what happens when it is broken, is discussed further in \cref{SEC:conclusion}.
    Another example may be if $\rho'$ is not supported in an interval around $z=0$. In this case, regardless of $K$, we can construct a continuous non-constant stationary solution.
    \begin{example}
        $K$ is translation invariant and has support in $B_\delta(0),$ and $\rho$ satisfies 
        \[
            \rho(z) =\begin{cases}
            1&|z|<r_0\\
            0&|z|\geq r_1
        \end{cases}
        \]
    where $r_0 <r_1$ and  $0\leq \rho(z) \leq 1$ when $r_0\leq |z|\leq r_1$.
    than any $u$ which is globally Lipschitz with Lipschitz constant $\frac{r_1}{\delta}$ or which satisfies $\sup(u)-\inf(u)<r_0$ is a stationary solution.
    \end{example}

    In this example, we still see that as the kernel $K$ increases its support, the strategic gradients are limited. Under the condition where every player in the domain interacts with every other player with some non-zero probability, stationary solutions will still satisfy the condition that every strategy being played is perfectly recognizable by every other strategy. Another way to state this is that letting $\rho$ be constant near 0 allows for trivial changes in the strategy profile at equilibrium, but we would still not observe any strategic gradient with payoff implications. 

    Lastly, we have a pathological kernel to show that when some assumptions are violated ((a)(c) and (d)), then there may be an infinite number of strategies present at equilibrium. 
    \begin{example}
        Consider the domain $(0,1)$ and consider a kernel $K(x,y)$ so that $K(x,y)=1$ if $\lfloor\frac{1}{x}\rfloor=\lfloor\frac{1}{y}\rfloor$ and $0$ elsewhere. Let $u(x)=\exp\{-\lfloor\frac{1}{x}\rfloor\}$. Then $u$ is a stationary solution. 
    \end{example}

    It is an easy computation to see that this is indeed a stationary solution. More to the point, we can also see that the image of $u$ is countably infinite. Although it is still discrete and has no strategic gradient, the result of \cref{THM:StationarySolutions} does not hold. 

    Although we can find counterexamples when the hypotheses are not satisfied, these results about stationary solutions are still helpful because they give us a way of narrowing down the search for Nash equilibria in the classical game. By \cref{PROP:StationarySolutions} and by \cref{THM:StationarySolutions,THM:StationarySolution}, we know that any bounded Nash equilibrium of the game with appropriate $K$ and $\rho'$ on a bounded domain will have a discrete and finite image when at most a set of measure 0 is excluded from the domain.

    \subsection{Other remarks}\label{subsec:other}
    \subsubsection{Boundary value problem}
    It may also be interesting to consider how these results change with the addition of boundary data. Although we do not give a complete proof, we can describe the set of stationary solutions with homogeneous Dirichlet boundary data. In this setting, homogeneous Dirichlet boundary data means that $u=c$ on $\mathbb{R}^n\setminus \Omega$. Finding stationary solutions here is equivalent to finding a particular subset of stationary solutions in an expanded domain $\Omega'=\bigcup_{x\in \Omega}supp(K(x,\cdot))$. If $u$ is a stationary solution in $\Omega'$, then it is also stationary in $\Omega$. To ensure that it satisfies the boundary conditions, we need only look for those stationary solutions in the expanded domain of the form 
    \[u=\sum_{k=1}^m c_k\chi_{S_k}\text{ where } c_k=c\text{ and }\Omega'\setminus\Omega\subset S_k \text{ for some } k\]
    
    This means that with homogeneous Dirichlet boundary data, stationary solutions are still piecewise constant. We do not expect the same to be true with the addition of inhomogeneous boundary data, nor do we expect this to hold if we consider the inhomogeneous problem discussed in the supplement.

    \subsubsection{Comparison to numerical results}
    To compare these analytical results with the numerical experiments, consider the maximum possible number of strategic bands present in the case that $K$ is supported on all of $\Omega\times \Omega$. From \cref{THM:StationarySolution}, we can see that if $b$ is the number of points in the image (which appear as bands on the bifurcation diagrams), $R:=\sup_{x\in\Omega}u(x)-\inf_{x\in\Omega}u(x)$ is the oscillation of $u$, and $r$ is the radius of $supp(\rho')$, then $b\leq\lfloor\frac{R}{r}\rfloor+1$. Notably, this means that if $\rho'$ is supported on all of $\mathbb{R}$ and has only one zero at $z=0$, then a stationary solution will be constant except possibly at a set of measure zero. This upper bound is sharp, as we can construct stationary solutions which satisfy $b=\lfloor\frac{R}{r}\rfloor+1$ easily.
    \[u(x)=\sum_{k=1}^{b}rk\chi_{S_k}=\begin{cases}
        r&x\in S_1\\
        2r&x\in S_2\\
        \vdots&\\
        br&x\in S_{b}
    \end{cases}\]
    Where $\{S_i\}_{i=1}^b$ partitions $\Omega$. 

    The difference between the analytical limit $b\leq \lfloor\frac{R}{r}\rfloor +1$ and the experimentally determined limit $b\leq \lfloor\frac{R}{2r}\rfloor+1$ poses an interesting area of further research and could have to do with the stability of solutions. The issues of convergence to a single stationary solution prevent us from filling the gap between the theoretical limit found in \cref{THM:StationarySolution} and the empirically observed limit from \cref{SEC:NumericalExperiments} rigorously. Attempts to study the stability of stationary solutions, even in the Cauchy setting or in a periodic domain, are foiled by the nonlinearity.  We can conjecture that the difference in these limits comes from the difference between stable stationary solutions and general stationary solutions, but cannot yet prove that all stationary solutions that achieve the theoretical limit of $b= \frac{R}{r} +1$ are unstable in $L^1$.

    \section{Conclusion}\label{SEC:conclusion}
    In this study, we have extended the game-theoretic treatment of structured coordination to allow for continuous pure strategies and provide novel critical insights, extending the application areas beyond those for which only the traditional, discrete pure strategies of discrete player spaces are appropriate. With certain reasonable hypotheses on the components of the fitness function, we were able to show rigorously that, through a myopic-best-response-like update rule, the situation can be modeled in continuous time by way of a nonlinear nonlocal equation, similar to existing nonlocal diffusion problems. 

    The nonlinearity in this model prevents the use of Fourier analysis or comparison principles in proving our results, but through elementary analysis and PDE techniques we were able to determine short-time existence and uniqueness for the general setting. With some additional requirements on $\rho$, we can strengthen this to find global existence and uniqueness. In addition to these results, we found that solutions with Lipschitz initial data remain Lipschitz continuous although the modulus of continuity may increase exponentially. In the special case of the Cauchy problem, the Lipschitz constant does not depend on the shape of the kernel, $K$, as long as it is translation-invariant.

    After giving several trivial analytical examples, we showed that simple numerical methods are stable and convergent, which allows us to visually examine solutions to the initial value problem without boundary data prescribed. Using these results, we were able to carry out several numerical examples which supported the analytical results about stationary solutions. 

    The effort involved in developing this model and showing that it is well-posed is worth it because it allows us to use previously unavailable techniques to ask questions about coordinating systems. In a general way, the question of how the relational structure of a coordinating system determines the game-theoretic outcome can be addressed with different tools using this model. Standard tools from nonlocal PDE theory can be used to ask questions about how the shape of $K$ changes the behavior of solutions. One such question, which is not addressed here but is an area of future study, asks how solutions will behave in the limit as the support of $K$ vanishes. In this so-called ``zero horizon" limit, we can test the extreme case where players only associate with the players directly next to them. This zero-horizon limit will turn a nonlocal equation into a PDE whose solutions can be analyzed (probably in a weak sense because of the expected discontinuities). Some of the analysis necessary to consider this setting, like strengthened regularity results in the case of a translation-invariant kernel, is included in the supplement. 
    Because our study focused on classical solutions, we had assumptions on $K$ which are stronger than the assumptions generally used in games defined on graphons. These graphons may be a key step to defining and understanding weak solutions of the game and are an area of future study. Moreover, the connection to the graphon literature means that the solutions of this nonlocal equation can be used to approximate game dynamics on very large graphs in a certain sense \cite{Medvedev2014,Carmona2019}. Describing exactly when that approximation is appropriate would be a step forward in describing network coordination and optimization. Questions like these get at our understanding of how the relational structure of a system determines its coordinating behavior. Further questions about relational structure can be asked and answered by use of the analytical results here and in \cite{Caffarelli2011, Medvedev2014}, and the numerical methods, which are provably useful because of the regularity results for this system. 

    Concerning the existence of strategic gradients in coordinating systems, we have shown that coordination under myopic best response is discretizing away from consensus and diffusive close to consensus. Away from consensus, this means that strategic gradients might blow up as time tends to $\infty$. Close to consensus, this means strategic gradients are suppressed to zero as time tends to $\infty$. In either scenario, coordinating dynamics alone seem not to support a strategic gradient unless the odd symmetry can be maintained. In the application areas, this suggests that strategic gradients observed in all sorts of coordinating systems with continuous strategies are the result of external factors, not coordination alone.

    In pursuit of further understanding of this setting, where perhaps some intrinsic benefit is conferred to an individual depending on their strategy and position in the relational structure, we can further develop an inhomogeneous version of this same model. Inclusion of an inhomogeneity can capture the ways that coordinating systems interacting with their environment might organize themselves. Although the full analysis of the inhomogeneous case is not carried out here, in the supplement
    we discuss the well-posedness of the inhomogeneous model and the existence of solutions to the IVP. Looking forward to other applications of this work, the myopic best response strategy revision protocol is not always a good model for how strategies change over time. Adapting this work to work with different strategy revision protocols in the continuous setting, like evolutionary games as in \cite{Hwang2013} or other imitative dynamics as in \cite{CHEUNG2016}.

    The main result of this manuscript is not just the fact that the strategic gradients must collapse as time goes to infinity, but that they do so in a predictable way related to the radius of recognition. This has many implications for a wide range of complex systems, but the general understanding that we can take away from it is that coordinating systems tend to form strategic communities (communities where every player is using the same strategy) that are distinctly separate from the strategic communities around them. This can be observed as polarization in political science \cite{Levin2021}, or it can be observed through identifying with regional linguistic quirks \cite{van2023}. Coordination dynamics alone promote the formation of these communities. These communities can represent many different things across many different fields (e.g., languages, political parties, phenotypes, etc.), but the insight remains the same. The gradients seen between strategic communities cannot be due to coordination alone. 
    
    Not only do these results help us understand the model as a way of discussing coordination in space, but they also represent advances in our understanding of nonlinear, nonlocal diffusion problems. On the modeling side, we see examples of solutions seeming to converge towards non-constant equilibria, the continuous analog of the non-trivial equilibria in the discrete case. Moreover, when discontinuities in strategy emerge, we can determine how quickly and how severely they can appear.  By characterizing stationary solutions, we can significantly advance our understanding of both the classical game and nonlinear nonlocal diffusion problems. The concentration of strategies appears as a smoothing in some parts of the domain and tends to create discontinuities in other parts of the domain. Further study into nonlinear nonlocal diffusion problems is required to describe the asymptotic behavior of solutions more generally and the stability of stationary solutions. From this study, it is clear that understanding coordination as a nonlinear nonlocal diffusion problem allows us to examine the system in exciting and novel ways, making accessible insights that were previously impossible.  

\section*{Acknowledgments}
The authors would like to thank Dr. Simon Levin and Dr. Ioannis Sgouralis for their thoughtful reading of this manuscript, as well as the two anonymous reviewers for their extremely helpful comments.

\section*{Declarations}

\subsection*{Funding statement}
This research received no outside funding

\subsection*{Conflicts of interest}
The authors have no conflicting interests to declare

\subsection*{\label{DataAvail}Code availability statement}
All of the software that was used to run the numerical simulations and generate the figures in this manuscript can be found at the repository linked here: \url{https://github.com/feffermanlab/JSM_2024_ContinuousCoordination/releases/tag/v1.0.0}. It is also cited as \cite{McAlister2025}.

\subsection*{Author Contribution Statement}
Both J.S.M. and T.A.M. contributed to the formal analysis and investigation. J.S.M. and N.H.F. conceptualized the work. J.S.M. wrote the software and wrote the original draft.

\bibliographystyle{plain}
\bibliography{coordinationReferences}

\begin{appendices}
\section{Proofs of the main results}\label{APP:Proofs}

\setcounter{section}{3}
\renewcommand{\thesection}{\arabic{section}}

\begin{appendixLemma}[$g$ is well defined]\label{APPLEM:welldefined}
        The non local operator $g[u](x,t)=\int_\Omega K(x,y)\rho'(u(x)-u(y))dy$ is well defined from $C^0_b(\Omega_T,\mathbb{R})$ into $C^0_b(\Omega_T,\mathbb{R})$.
    \end{appendixLemma}
    This lemma appears as \cref{LEM:welldefined} in the main text.
    \begin{proof}
        To show that $g$ is well-defined in this space, let $u\in C^0_b(\Omega_T)$, and we will show that $g[u]$ is continuous on $\Omega_T$ and bounded. To show that $g[u]$ is bounded, first note that, when $u$ is bounded, because $\rho\in C^{1,1}$, surely there is a $W$ such that $|\rho'(u(x,t_1)-u(y,t_2))|<W$ for all $x,y\in \Omega$ and $t_1,t_2\in [0,T]$ by a continuous on compact argument. This means that $|g[u](x,t)|\leq \int_\Omega |K(x,y)|Wdy\leq W\|K(x,\cdot)\|_{L^1(\Omega)}$. By assumption $\|K(x,\cdot)\|_{L^1(\Omega)}$  is uniformly bounded so $\sup_{x\in\Omega}\|K(x,\cdot)\|_{L^1(\Omega)}\leq D$ for some $D<\infty$. Thus, we have shown that $\|g[u]\|_\infty<WD$, so $g[u]$ is bounded whenever $u$ is continuous and bounded.

        Consider the sequence ${(x_n,t_n)}_{n=1}^\infty\subset\Omega_T$ with $(x_n,t_n)\rightarrow (x,t)\in\Omega$. Furthermore, consider the difference 

        \begin{equation*}
        \begin{split}
        |g[u](x_n,t_n)&-g[u](x,t)|\\
        &=\bigg|\int_\Omega K(x_n,y)\rho'(u(x_n,t_n)-u(y,t_n))-K(x,y)\rho'(u(x,t)-u(y,t))dy\bigg|\\
            &\leq \int_\Omega|(K(x_n,y)-K(x,y))\rho'(u(x_n,t_n)-u(y,t_n))|dy\\
            &\quad +\int_\Omega |K(x,y)(\rho'(u(x_n,t_n)-u(y,t_n))-\rho'(u(x,t)-u(y,t)))|dy\\
            &\leq I_1+I_2
            \end{split}
        \end{equation*}
        We will consider each integral separately. Let $\epsilon>0$. Note that there is an $R$ such that $\int_{\Omega\setminus B_{R}(0)}K(x,y)dy<\frac{\epsilon}{8W}$. Now consider $x_n$ and $y$ inside $\overline{B_R(0)}.$ $\rho\in C^{1,1}$ so $\exists \delta$ such that $|x-y|<\delta \Rightarrow |\rho'(x)-\rho'(y)|<\frac{\epsilon}{4D}$. Because $u$ is continuous, $\exists M_1$ such that $|u(x_n,t_n)-u(x,t)|<\delta/2$ whenever $n>M_1$. Furthermore, In $\Omega\cap \overline{B_R(0)}$ there is surely a $M_2$  such that $n>M_2\implies |u(y,t_n)-u(y,t)|<\delta/2$ for any $y\in \Omega\cap\overline{B_R(0)}$. It is an easy application of the triangle inequality to see that this implies $|(u(x_n,t_n)-u(y,t_n))-(u(x,t)-u(y,t))|<\delta$ when $n>\max\{M_1,M_2\}$ and thus, when $n$ is sufficiently large
        
       \[|\rho'(u(x_n,t_n)-u(y,t_n))-\rho'(u(x,t)-u(y,t))|\leq \frac{\epsilon}{4D} \quad\forall y\in\Omega\cap \overline{B_R(0)}.\]

        Therefore, we can see that
        \begin{equation*}
            \begin{split}
                I_2&\leq 2W\int_{\Omega\setminus B_R(0)}K(x,y)dy+\int_{\Omega\cap B_R(0)}K(x,y)\frac{\epsilon}{4D}dy\\
                &\leq 2W\frac{\epsilon}{8W}+D\frac{\epsilon}{4D}=\frac{\epsilon}{2}
            \end{split}
        \end{equation*}

        Now we consider $I_1$. It is clear that
        
       \[I_1 \leq W\int_{\Omega}|K(x_n,y)-K(x,y)|dy.\] 
        From the assumed continuity of $\|K(x,\cdot)\|_{L^1(\Omega)}$ we get immediately that $\exists M_3$ such that for $n>M_3$, 
        $\|K(x_n,\cdot)- K(x,\cdot)\|_{L^1(\Omega)}\leq \frac{\epsilon}{2W}$.

        Thus, if $M=\max\{M_1,M_2,M_3\}$, then $n>M\implies$
        
       \[|g[u](x_n,t_n)-g[u](x,t)|\leq I_1+I_2\leq \frac{\epsilon}{2}+W\frac{\epsilon}{2W}=\epsilon\]
        This clearly works for any $\epsilon>0$, so we have shown that $g[u]$ is continuous at $(x,t)$ for any $(x,t)\in \Omega _T$. Therefore we have proved that $g:C^0_b(\Omega_T,\mathbb{R})\rightarrow C^0_b(\Omega_T,\mathbb{R})$ is well defined.  
    \end{proof}

    \begin{appendixLemma}[Lipschitz continuity of $g$] \label{APPLEM:Lipschitzg}
        For any bounded (in the sup norm sense) subset $X_R:=\{u\in C^0_b(\Omega_T);\|u\|_\infty \leq R\}\subset C^0_b(\Omega_T)$, there exists a $C^g\geq 0 $ such that for every $u,v\in X_R$ 
        \begin{equation*}
            \|g[u](\cdot,t)-g[v](\cdot,t)\|_{\infty} \leq C^g\|u(\cdot,t)-v(\cdot,t)\|_{\infty} 
        \end{equation*}
    \end{appendixLemma}
    This lemma appears as \cref{LEM:Lipschitzg} in the main text.
    \begin{proof}
        Let $u,v\in X_R$. Then we observe that for each $0<t<T,$
        \begin{equation*}
            \begin{split}
              \|g[u](\cdot,t)-g[v](\cdot,t)\|_\infty   &=\bigg\|\int_\Omega K(x,y)(\rho'(u(x,t)-u(y,t))-\rho'(v(x,t)-v(y,t))dy\bigg\|_\infty\\
                &\leq \sup_{x\in \Omega}\int_\Omega|K(x,y)||\rho'(u(x,t)-u(y,t))-\rho'(v(x,t)-v(y,t))|dy\\
                &\leq \sup_{x\in \Omega}\|\rho'(u(x,
                t)-u(\cdot,t))-\rho'(v(x,t)-v(\cdot,t))\|_\infty \int_\Omega K(x,y)dy. 
            \end{split}
        \end{equation*}

        Recall that $\rho\in C^{1,1}(\mathbb{R})$  and $\rho'$ has Lipschitz constant $L_\rho$ for the interval $[-2R,2R]$. Now note that for any $(x,t)\in\Omega_T$
        \begin{equation*}
            \begin{split}
                \sup_{y\in \Omega}|\rho'(u(x,t)&-u(y,t))-\rho'(v(x,t)-v(y,t))|\\
                &\leq \sup_{y\in \Omega}L_\rho|u(x,t)-u(y,t)-v(x,t)+v(y,t)|\\
                &\leq \sup_{y\in \Omega}L_\rho (|u(x,t)-v(x,t)|+|u(y,t)-v(y,t)|)\\
                &\leq L_\rho(|u(x,t)-v(x,t)|+\|u(\cdot,t)-v(\cdot, t)\|_\infty)
            \end{split}
        \end{equation*}
        so naturally 
        
       \[\sup_{x\in \Omega}\|\rho'(u(x,t)-u(\cdot,t))-\rho'(v(x,t)-v(\cdot,t))\|_\infty \leq 2L_\rho\|u(\cdot,t)-v(\cdot,t)\|_{\infty}\]
        This, and the fact that $\|K(x,\cdot)\|_{L^1(\Omega)}\leq D$ uniformly for some finite $D$, gives us the result that for any $t\in[0,T]$
        
       \[\|g[u](\cdot,t)-g[v](\cdot,t)\|_{\infty}\leq 2L_\rho D\|u(\cdot,t)-v(\cdot,t)\|_{\infty}\]
        Clearly, then we have $C^g=2L_\rho D$ and 
        
       \[\sup_{t\in [0,T]}\|g[u](\cdot,t)-g[v](\cdot,t)\|_{\infty}\leq C^g\sup_{t\in [0,T]}\|u(\cdot,t)-v(\cdot,t)\|_{\infty}.\] It is important to notice that $L_\rho$ may depend on the choice of $R$ so $C^g$ depends on $R$.
        \end{proof}
    \setcounter{appendixTheorem}{2}
    \begin{appendixTheorem}[Short time existence and uniqueness] \label{APPTHM:ShortTimeExistenceAndUniqueness}
        The initial value problem $u_t=g[u]$ in $\Omega_\tau$ has a unique continuous and bounded solution in $\Omega_\tau$ for some $\tau$, when $u(x,0)=u_0\in C_b^0(\Omega)$, $\rho\in C^{1,1}(\mathbb{R})$, and $K\in C^0_b(\Omega;L^1(\Omega))$.
    \end{appendixTheorem}
    The statement of this theorem appears as \cref{THM:ShortTimeExistenceAndUniqueness} in the main text.
    \begin{proof}
        Let $\Omega_T=\Omega\times [0,T]$ with $T$ to be chosen later. Equip the function space $C_b^0(\Omega_T)$ with the standard sup norm $\|u\| = \sup_{t\in[0,T]}\|u(\cdot,t)\|_\infty$. 
        Now, for some $R>\|u_0\|_\infty$, let $E_{R,T}:=\{u\in C^0_b(\Omega_T,\mathbb{R});u(x,0)=u_0, \|u\|\leq R\}$. Observe that $E_{R,T}$ is nonempty as the map $(x,t)\mapsto u_0(x)$ belongs in  $E_{R,T}$. Moreover, observe that $E_{R,T}$ is complete with respect to the sup norm, so we will be able to proceed with a Banach Fixed Point Theorem (BFPT) argument. It is clear that a solution to the IVP will also satisfy
        
       \[u(x,t)=u_0(x)+\int_0^tg[u](x,s)ds.\]
        Let $\Theta:C_b^0(\Omega_T,\mathbb{R})\rightarrow C_b^0(\Omega_T,\mathbb{R})$, where $\Theta u = u_0+\int_0^tg[u]ds$, and notice that if this operator has a unique fixed point in $E_{R,T}$, then we have a unique solution to the IVP. 

        First, we show that $\Theta:E_{R,T}\rightarrow E_{R,T}$ for some $T$. It is easy to say that $\Theta u\in C^0_b(\Omega_T,\mathbb{R})$ and that $\Theta u(x,0)=u_0(x)$.  Because of \cref{APPLEM:welldefined}, we know that $g[u]\in C_b^0(\Omega_T)$, so its time antiderivative is obviously continuous in space and time. 
        
        In order to show that $\|\Theta u\| \leq R$, we note that $g[u]$ is bounded whenever $u$ is bounded. In particular, $g[u]\leq WD$ where $W$ is the bound for $\rho'$ on  $[-2R,2R]$ which, of course, depends on $R$, and $D$ is the uniform bound on $\|K(x,\cdot)\|_{L^1(\Omega)}$. Therefore we can always find a $T_R$ such that $\|\Theta u(\cdot, t)\|_\infty \leq \|u_0\|_\infty +\int_0^tWD \leq R$ for all $t\in [0,T_R]$. Let $T<T_R$ and we have that $\Theta:E_{R,T}\rightarrow E_{R,T}$. 
        
        Now that we have shown that $\Theta$ indeed maps from $E_{R,T}$ to $E_{R,T}$ for some $T$, we need only show that $\Theta$ is a contraction in $E_{R,T}$ for some, possibly smaller, value of $T$. Let $u,v \in E_{R,T}$ and note that
        \begin{equation*}
        \begin{split} 
            \|\Theta u-\Theta v\| &= \sup_{t\in[0,T]} \bigg\|\int_0^tg[u](\cdot,s)-g[v](\cdot,s)ds\bigg\|_{\infty}\\
            &\leq \sup_{t\in[0,T]}\int_0^t\|g[u]-g[v]\|ds
            \end{split} 
        \end{equation*}
        By \cref{APPLEM:Lipschitzg} we get immediately that
        
       \[\|\Theta u-\Theta v\|\leq C^gT\|u-v\|.\] Notice here that $C^g$ depends on $R$ because the Lipschitz constant for $g$ is defined for a particular compact subset.  
        With this Lipschitz constant for $\Theta$, when $T\leq\frac{1}{2C^g}$ we know that $\Theta$ is a contraction. Thus, by the BFPT, there is a unique $u\in E_{R,T}$ such that $\Theta u = u$.  Thus, there is a continuous and bounded $u$ so that $u_t=g[u]$ on $\Omega_T$ and $u(0,t)=u_0\in C^0_b(\Omega)$ so long as $T<\min\{T_R,\frac{1}{2C^g}$\}. This completes the proof of short-time existence and uniqueness by way of the Banach Fixed Point Theorem. We can use an extension principle to get longer existence time. Note that for a given $u_0$ we can choose an $R_1$ and find a resulting $T_1$ so that there is a unique solution on $[0,T_1]$. Take $u(\cdot, T_1-\epsilon)$ for some $\epsilon>0$ as our initial condition, take a new $R_2$ and resulting $T_2$ to find a new solution on $[T_1-\epsilon, T_2]$. Wherever these solutions overlap, they must be identical because of the uniqueness proved here. We cannot use this extension to say that the solution is global in time because we have no lower bound on the minimal existence time. As $R$ grows, the resulting $T$ may shrink quickly enough so that we cannot extend the solution beyond some finite time $\tau$.     
    \end{proof}

        \setcounter{appendixLemma}{3}

        \begin{appendixLemma}[Finite time blow up]\label{APPLEM:FiniteTimeBlowup}
        If  $T<\infty$ is the maximal time of existence for a solution $u$ to the IVP $u_t=g[u]$ with $u(\cdot,0)=u_0\in C^0_b(\Omega)$ then $\|u(\cdot, t)\|_\infty \rightarrow \infty $ as $t\to T$.
        \end{appendixLemma}
        This lemma is given as \cref{LEM:FiniteTimeBlowup} in the main text.
        \begin{proof}
        Suppose that $u$ is a solution to the IVP $u_t=g[u]$ with $u(\cdot, 0)=u_0\in C^0_b(\Omega)$ that has a maximal time of existence $T<\infty$. Moreover, by way of contradiction, suppose there is a bounded subset $E_R:=\{v\in C^0_b(\Omega);\|v\|_\infty\leq R\}$ such that for all $t\in [0,T)$ $u(\cdot, t)\in E_R$. Note that $E_{R}$ is a closed subset of $C^0_b(\Omega).$

        Consider now a sequence of times $t_n$ such that $t_n\to T$ as $n\rightarrow \infty$. This is necessarily a Cauchy sequence in the reals. We will show that $u(\cdot, t_n)$ is a Cauchy sequence in $E_R$ with respect to the sup norm. Note that for any $\delta>0$ there is an $N$ so that $n,m>N\implies |t_n-t_m|<\delta$. 
        \begin{equation*}
            \begin{split}
                \|u(\cdot,t_n)-u(\cdot, t_m)\|_\infty&=\bigg\|\int_0^{t_n}g[u](\cdot, s)ds-\int_0^{t_m}g[u](\cdot,s)ds\bigg\|_\infty\\
                &=\bigg\|\int_{t_n}^{t_m}g[u](\cdot,s)ds\bigg\|_\infty\\
                &\leq \delta \sup_{s\in[t_n,t_m]}\|g[u](\cdot,s)\|_\infty
            \end{split}
        \end{equation*}
        By assumption, $\|u\|_\infty$ is bounded for all time $t\in [t_n,t_m]\subset [0,T)$ uniformly in time by $R$. By \cref{APPLEM:welldefined} we know that if $u$ is bounded by $R$ then there is an $R_g$ so that $\|g[u]\|_\infty\leq R_g$ and, crucially, this upper bound only depends on the choice of $R$ (see proof of \cref{APPLEM:welldefined}).  Thus we can say that $\sup_{s\in[t_n,t_m]}\|g[u](\cdot,s)\|_\infty\leq R_g$. Thus we have that
        
       \[\|u(\cdot, t_n)-u(\cdot, t_m)\|_\infty\leq \delta R_g\]

        Therefore, for any $\epsilon>0$, we let $\delta=\frac{\epsilon}{2R_g}$ and from this we get an appropriate $N$ so that $|t_n-t_m|<\delta$ and thus $\|u(\cdot,t_n)-u(\cdot, t_m)\|\leq \epsilon$. Thus $u(\cdot, t_n)$ is a Cauchy sequence in the closed subset of a Banach space, $E_R$, and so it has a limit $\tilde{u}_0\in E_R$. (Note carefully that $\tilde{u}_0$ is continuous because the convergence is uniform). Now consider another arbitrary sequence of times $\tau_k$ with $\lim \tau_k\rightarrow T$ and we will show that $u(\cdot, \tau_k)\rightarrow \tilde{u}_0$. Of course, $u(\cdot,t)$ is differentiable in time on $[0,T)$ and that derivative is uniformly bounded by $R_g$ as above, so $|u(x,t)-u(x,\tau)|\leq R_g|t-\tau|$. Thus, we can write 
        \begin{equation*}
            \begin{split}
                \|u(\cdot, \tau_k)-\tilde{u}_0\|&\leq \|u(\cdot, \tau_k)-u(\cdot, t_n)\|_\infty+\|u(\cdot,t_n)-\tilde{u}_0\|_\infty\\
                &\leq R_g|\tau_k-t_n|+\|u(\cdot,t_n)-\tilde{u}_0\|_\infty
            \end{split}
        \end{equation*}
        Both $\tau_k$ and $t_n$ approach $T$ from the left and so if $\epsilon>0$ there is an $n_1$ such that $k,n>n_1\implies |t_k-t_n|<\frac{\epsilon}{2R_g}$. Furthermore, we already showed that $\|u(\cdot,t_n)-\tilde{u}_0\|\rightarrow 0$ so there is an $n_2$ such that $ n>n_2\implies \|u(\cdot,t_n)-\tilde{u}_0\|_\infty \leq \frac{\epsilon}{2}$. Therefore we get that when $k>\max\{n_1,n_2\}$ then surely
        
       \[\|u(\cdot, \tau_k)-\tilde{u}_0\|\leq \epsilon.\]
        This is true for any sequence of $\tau\rightarrow T$ from the left thus we write that $\lim_{t\to T^-}u(\cdot,t)=\tilde{u}_0\in C^0_b(\Omega)$.

        Because we have a $\tilde{u}_0\in C^0_b(\Omega)$, by \cref{APPTHM:ShortTimeExistenceAndUniqueness} we can find a solution $\tilde u$ to the IVP $\tilde u=g[\tilde u]$ with $\tilde u(\cdot, T)=\tilde u _0$ on some interval $[T,T+\eta).$ Now let $\hat u(\cdot, t) = u(\cdot,t)$ when $t\in [0,T)$ and $\hat u(\cdot, t)=\tilde u(\cdot ,t)$ when $t\in [T,T+\eta)$. We will show now that $\hat u\in C^0_b(\Omega_{T+\eta}).$ The boundedness is immediate from the hypothesis and from the details of the proof of \cref{APPTHM:ShortTimeExistenceAndUniqueness}. Also, we know that $\hat{u}$ is continuous on $[0,T)$ and on $(T,T+\eta)$. To see that is also continuous at $T$ observe that $\lim_{t\rightarrow T^-}\hat{u}(\cdot, t)=\hat{u}(\cdot, T)$ by the previous result and that $\lim_{t\rightarrow T^+}\hat{u}(\cdot, t)=\hat{u}(\cdot, T)$ by the details in the proof of \cref{APPTHM:ShortTimeExistenceAndUniqueness}. This convergence is uniform and so we can say that  $\hat u\in C^0_b(\Omega_{T+\eta}).$
        
        Lastly, we show that $\hat{u}$ is a solution to the IVP, that is, in the integral form, 
        
       \begin{equation}\label{integralequation}\hat{u}(x,t)=u_0(x)+\int_0^tg[\hat{u}](x,s)ds.\end{equation}
        This is obviously true when $t<T$. Moreover, we can see that when $t=T$ we can write
        \begin{equation}\label{corEq2}
            \begin{split}
            u_0(x)+\int_0^Tg[\hat{u}](x,s)ds&=u_0(x)+\int_0^{T-\epsilon}g[{u}](x,s)ds+\int_{T-\epsilon}^Tg[{u}](x,s)ds\\
                &=u(x,T-\epsilon)+\int_{T-\epsilon}^Tg[u](x,s)ds
            \end{split}
        \end{equation}
        Recall that, because $\|\hat{u}\|_\infty=\|u\|_\infty\leq R$ when $t<T$,  $\|g[\hat{u}]\|_\infty=\|g[u]\|_\infty\leq R_g$ when $t<T$. Therefore if we take the limit as $\epsilon \rightarrow 0$ in equation \eqref{corEq2} we get $u_0(x)+\int_0^Tg[u](x,s)ds=\lim_{t\rightarrow T^-}(x,t)=\tilde{u}_0(x)=\hat{u}(T)$. 

        Finally, when $t>T$ we observe that 
        \begin{equation*}
            \begin{split}
                u_0(x)+\int_0^tg[\hat{u}](x,s)ds&= u_0(x)+\int_0^Tg[\hat{u}](x,s)ds+\int_T^tg[\hat{u}](x,s)ds\\
                &=\tilde{u}_0+\int_T^tg[\tilde{u}](x,s)ds\\
                &=\tilde{u}(x,t)=\hat{u}(x,t)
            \end{split}
        \end{equation*}

        Therefore $\hat{u}\in C^0_b(\Omega_{T+\eta})$ and $\hat{u}$ satisfies that integral equation \eqref{integralequation}. Thus $\hat{u}$ extends the solution to the IVP $u_t=g[u]$ with $u(\cdot,0)=u_0$ beyond the assumed maximal time of existence. This contradicts the assumption that $T$ was the maximal time of existence. Thus, if there is a finite maximal time of existence, the solution must leave every compact subset of $C^0_b(\Omega_T)$. Because the limit $u(\cdot,t_n)$ is continuous for $t_n\leq T<\infty$ we know that $\|u(\cdot, t)\|_\infty\to \infty$ as $t\to T$.
        \end{proof}

        \begin{appendixLemma}[Weak maximum principle]\label{APPLEM:MaximumPrinciple}
        If $u$ solves $u_t = g[u]$ in $\Omega_T$ with $u(\cdot, 0)=u_0\in C^0_b(\Omega)$ and if $\rho$ satisfies \eqref{coordinationcondition}, then 
       \[\|u(\cdot, t_2)\|_\infty\leq\| u(\cdot , t_1)\|_\infty \] whenever $t_1\leq t_2$. 
    \end{appendixLemma}
    This lemma is included as \cref{LEM:MaximumPrinciple} in the main text.
    \begin{proof}
        Notice that it is no loss of generality to assume $t_1=0$, so we will prove that $\|u(\cdot,t)\|_\infty\leq \|u_0\|_\infty$ for all $t\in [0,T)$ where $T$ is the maximal existence time which may be infinite. Let $\epsilon>0$ and let $v=\|u_0\|_\infty +\epsilon$. Now observe that if $u$ solves $u_t=g[u]$ and $\tilde{u}:=u-\epsilon t$, then $\tilde{u}_t=g[\tilde{u}]-\epsilon$. This is because $g[u-\epsilon t]=g[u]$ (More generally, $g$ is invariant under vertical shifts, even if they are time-dependent). Suppose that $\tilde{u}(x^\star,t^\star) = v$ for the first time at some $x^\star\in \Omega$ and some $t^\star>0$. That means that when  $t<t^\star$, $\tilde{u}(y,t)<v$ for all $y\in \Omega$ and by continuity of $u$, $\tilde{u}(y,t^\star)\leq v$ for all $y\in\Omega$. Using the fact that $\rho'(z)\leq 0$ when $z\geq 0$, we obtain that    
        \[g[\tilde{u}](x^\star,t^\star)=\int_\Omega K(x,y)\rho'(\tilde u(x^\star,t^\star)-\tilde u(y,t^\star))dy\leq 0, \] 
        and therefore, $\partial_t\tilde u(x^\star,t^\star)\leq 0-\epsilon<0$. However, because $\tilde{u}$ is continuously differentiable in time (see proof of \cref{APPTHM:ShortTimeExistenceAndUniqueness}) and because $\tilde u=v$ for the first time at $t=t^\star$, we can say $\partial_t\tilde u(x^\star,t^\star)\geq 0$. This is a contradiction, so we can say that $\tilde{u} < v$, which means $u< \|u_0\|_\infty+\epsilon+\epsilon t$. This is true for any $\epsilon$ so let $\epsilon\rightarrow 0$ and we see that $\|u(\cdot, t)\|_\infty<\|u_0\|_\infty$ for all finite time. This proves the result.
    \end{proof}
    
    \setcounter{appendixTheorem}{6}
    \begin{appendixTheorem}[Regularity for the coordination game without boundary conditions]\label{APPTHM:Regularity}
        Let $\Omega\subseteq \mathbb{R}^n$. Suppose $u$ solves the Initial Value Problem $u_t=g[u]$ with the coordination assumption \eqref{coordinationcondition} and with $u_0\in C^{0,1}_b(\Omega)$ with a uniform Lipschitz constant $L_0$. If $K$ is uniformly Lipschitz in the first variable with respect to the $L^1$ norm, then for any finite time $t$, $u(\cdot, t)\in C^{0,1}(\Omega)$. Moreover, if $\Omega$ is open and bounded with a $C^1$ boundary,  $u\in W^{1,\infty}(\Omega)$. Furthermore, for some positive $c$ and $C$, 
        
       \[\|D_xu(\cdot, t)\|_{L^\infty(\Omega)}\leq (L_0+Ct)e^{ct} \]
        Where $C$ depends only on $K$ and $u_0$ and $c$ depends only on $K$ and $\rho$.
    \end{appendixTheorem}
    The statement of this theorem appears as \cref{THM:Regularity} in the main text.
    \begin{proof}
        Let $\phi(h,x,t)=u(x+h,t)-u(x,t)$ for some $h\in  \mathbb{R}^n$ so that both $x$ and $x+h$ are in $\Omega$, and observe that by simply subtracting the two solutions from one another.
        \begin{equation}\label{regproofeq1}
            \phi(h,x,t)=\phi(h,x,0)+\int_0^t\left(g[u](x+h,s)-g[u](x,s)\right)ds
        \end{equation}
        Let $L_\rho$ be the Lipschitz constant for $\rho'$ on the interval $[-2\|u_0\|_\infty,2\|u_0\|_\infty]$. Recall that, by \cref{APPLEM:MaximumPrinciple}, $u\in [-\|u_0\|_\infty,\|u_0\|_\infty]$ for all $x,t\in \Omega_T$ so we know that  $u(x,t)-u(y,t)\in [-2\|u_0\|_\infty,2\|u_0\|_\infty]$ regardless of $x$ and $y$. Also note that $\rho'$ is bounded by some $W$ on the same interval. Now observe that this integrand is
        \begin{equation*}
            \begin{split}
                |g[u]&(x+h,s)-g[u](x,s)|\\
                &=\bigg|\int_\Omega K(x+h,y)\rho'(u(x+h,s)-u(y,s))-K(x,y)\rho'(u(x,s)-u(y,s))dy\bigg|\\
                &\leq \int_\Omega K(x+h,y)|\rho'(u(x+h,s)-u(y,s))-\rho'(u(x,s)-u(y,s))|dy\\
                &\quad \quad +\int_\Omega [K(x+h,y)-K(x,y)]|\rho'(u(x,s)-u(y,s))|ds\\
                &\leq \int_\Omega K(x+h,y)L_\rho|u(x+h,s)-u(x,s)|dy\\
                &\quad\quad+W\|K(x+h,\cdot)-K(x,\cdot)\|_{L^1(\Omega)}\\
                &\leq L_\rho\int_\Omega K(x+h,y)|\phi(x,h,s)|dy+WL_Kh\\
                & \leq L_\rho\sup_{z\in \Omega}\|K(z,\cdot)\|_{L^1(\Omega)}\|\phi(h,\cdot,s)\|_{\infty}+WL_kh\\
            \end{split}
        \end{equation*}
        let $c=L_\rho\sup_{z\in \Omega}\|K(z,\cdot)\|_{L^1(\Omega)}$ and $C=WL_k$ and make this replacement into equation \eqref{regproofeq1}  and take the sup norm over $\Omega$ to see that
        
       \[\|\phi(h,\cdot,t)\|_\infty \leq \|\phi(h,\cdot,0)\|_\infty +\int_0^t[c\|\phi(h,\cdot, s)\|_\infty +Ch]ds\]
        
        It is an easy application of Gr\"onwall's inequality to see that
        \begin{equation*}
             \|\phi(h,\cdot,t)\|_\infty
             \leq(\|\phi(h,\cdot, 0)\|_\infty+Cht)e^{ct}
        \end{equation*}
        Note also that $\|\phi(h,\cdot,0)\|_\infty\leq |h|L_0$, so the difference from any $x\in \Omega$ on any finite time interval $[0,T]$ is controlled by
        
        \[
            |u(x,t)-u(x+h,t)|\leq |h|(L_0+CT)e^{cT}
        \]

        Notice that we have shown that for each $t$, $u(\cdot,t)$ is globally Lipschitz continuous in $\Omega$ with Lipschitz constant $L_T=(L_0+CT)e^{cT}$. Note that this constant is uniform in $t$. This means that, when $\Omega$ is open and bounded, we can use the characterization of $W^{1,\infty}(\Omega)$ \cite{evansPDE} to conclude that $u$ is weakly differentiable and the weak derivative is bounded by the Lipschitz constant.
        
         \[\|D_xu(\cdot,t)\|_{L^\infty(\Omega)}\leq L_T=(L_0+CT)e^{cT}\]
         
    \end{proof}
    
    \setcounter{appendixCorollary}{7}
    \begin{appendixCorollary}\label{APPCOR:GlobalLip}
        For $\Omega \subseteq \mathbb{R}^n$, suppose that $u$ solves the IVP $u_t=g[u]$ under the assumptions in \cref{APPTHM:Regularity} on a finite time domain $[0,T]$. In this case, $u$ is globally Lipschitz in $\Omega\times[0,T]$
    \end{appendixCorollary}
    The statement of the corollary appears as \cref{COR:GlobalLip} in the main text.
    \begin{proof}
        By \cref{APPTHM:Regularity}, $u$ is Lipschitz continuous in space and the Lipschitz constant is uniform on a bounded time interval. Moreover, we know that $u(x,\cdot)$ is continuously differentiable because $u_t(x,t)=g[u](x,t)$, which is continuous in space and time by \cref{APPLEM:welldefined}. $g[u]$ is bounded above by a constant related to the form of $K,\rho$ and the bounds on $u$. Under the hypothesis of \cref{APPTHM:Regularity}, $\|u(\cdot, t)\|_\infty\leq \|u_0\|_\infty$ so we obtain  
        \begin{equation*}
            \begin{split}
                |u(x,t)-u(y,s)|&\leq |u(x,t)-u(y,t)|+|u(y,t)-u(y,s)|\\
                &\leq C_T|x-y|+C_{0}|t-s|\\
                &\leq L(|x-y|+|t-s|), 
            \end{split}
        \end{equation*}
        where $L$ is a constant which depends only on $T, K, \rho$ and $u_0$. Of course, we can adjust this norm to say that 
        \[|u(x,t)-u(y,s)|\leq C \sqrt{|x-y|^2+|t-s|^2}\]
        so that we can express it using the Euclidean norm in $\Omega_T\subset \mathbb{R}^{n+1}$
    \end{proof}
    \setcounter{appendixLemma}{8}

    \begin{appendixLemma}[Solutions are dissipative]\label{applem:dis}
        Suppose that $K$ is symmetric, $\rho$ is even, satisfying \eqref{coordinationcondition}, and $\rho'$ has only a single zero in the interior of its support. If $u$ is a solution to the IVP $u_t=g[u]$ and $u(\cdot, u)=u_0\in L^2(\Omega)$, then $\|u(\cdot, T)\|_{L^2(\Omega)}\leq \|u_0\|_{L^2(\Omega)}$ for all $T$.
    \end{appendixLemma}
    This result is given as \cref{lem:dis} in the main text.
    \begin{proof}
        The proof is done through a straightforward computation. Set \[ E(t) = \|u\|^2_{L^2}=\int_\Omega u(x,t)^2dx.\]
        Then 
        \begin{equation}
            \begin{split}
                E'(t)&=\int_\Omega 2 u(x,t)\frac{\partial}{\partial t}u(x,t)dx\\
                &=2\int_\Omega\int_\Omega K(x,y)\rho'(u(x,t)-u(y,t))u(x,t)dydx\\&=\int_\Omega\int_\Omega K(x,y)\rho'(u(x,t)-u(y,t))u(x,t)dydx\\&\quad\quad +\int_\Omega\int_\Omega K(x,y)\rho'(u(x,t)-u(y,t))u(x,t)dydx. 
            \end{split}
        \end{equation}
        Now we use the symmetry of $K$ and the odd symmetry of $\rho'$ to do a symmetrization argument
        \begin{equation}
        \begin{split}
           E'(t)&=\int_\Omega\int_\Omega K(x,y)\rho'(u(x,t)-u(y,t))u(x,t)dydx\\&\quad\quad +\int_\Omega\int_\Omega K(y,x)\rho'(u(y,t)-u(x,t))u(y,t)dydx\\
            &=\int_\Omega\int_\Omega K(x,y)\rho'(u(x,t)-u(y,t))u(x,t)dydx\\&\quad\quad -\int_\Omega\int_\Omega K(x,y)\rho'(u(y,t)-u(x,t))u(y,t)dydx\\
            &=\int_\Omega\int_\Omega K(x,y)\rho'(u(x,t)-u(y,t))(u(x,t)-u(y,t))dydx. 
            \end{split}
        \end{equation}
        Lastly, we notice that $\rho'(z)z\leq 0$ to say that $E'(t)\leq 0$. 
       This implies that $E(t)$ is a nonincreasing function. 
        This completes the proof.
    \end{proof}
    \setcounter{appendixCorollary}{9}
     \begin{appendixCorollary}[Collapse of the intermediate range]\label{Appcor:intColapse}
        Suppose $\rho$ is as in Lemma  \ref{lem:dis} with support in $[-a,a]$. Let $I$ be any closed subinterval in $(0,a)$ and let $E_I(t):=\{(x,y)\in \Omega\times\Omega; |u(x,t)-u(y,t)|\in I\}$. Then for any $\epsilon$ there is a $T$ so that 
        \[\int_T^\infty\iint_{E_I(t)}K(x,y)dxdydt<\epsilon\]
    \end{appendixCorollary}
    This result is given as \cref{cor:intColapse} in the main text.
    \begin{proof}
        Consider the closed subinterval $I$ in $(0,a)$ and notice that for $z\in I$, $\rho'(z)z<-\eta<0$ for some $\eta$. Equivalently $\frac{\rho'(z)z}{-\eta}\geq 1$
        \begin{equation}
        \begin{split}
            \iint_{E_I(t)}K(x,y)dydx& \leq \iint_{E_I(t)}K(x,y)\frac{\rho'(u(x,t)-u(y,t))(u(x,t)-u(y,t))}{-\eta}dydx\\
            &\leq\int_\Omega\int_\Omega K(x,y)\frac{\rho'(u(x,t)-u(y,t))(u(x,t)-u(y,t))}{-\eta}dydx\\
            \end{split}
        \end{equation}
        By the same symmetrization argument we made previously in the proof of Lemma \ref{lem:dis}, we can say that 
        \begin{equation}
            \begin{split}
              \iint_{E_I(t)}K(x,y)dydx&\leq\frac{-2}{\eta}\int_\Omega\int_\Omega K(x,y)\rho'(u(x,t)-u(y,t))u(x,t)dydx\\  
              &\leq \frac{-2}{\eta}\int_\Omega u_t(x,t)u(x,t)dx\\
              &\leq\frac{-2}{\eta}\int_\Omega \frac{\partial}{\partial t}u^2dx\\
              &\leq\frac{-2}{\eta}\frac{\partial}{\partial t}\|u\|_{L^2}^2
            \end{split}
        \end{equation}
        This means that if we integrate this quantity over any time interval $t\in [0,T]$, we get 
        \begin{equation}
        \begin{split}
            \int_0^T\iint_{E_I(t)}K(x,y)dydx&\leq \frac{-2}{\eta}\int_0^T\frac{\partial}{\partial t}\|u\|_{L^2}^2 dt\\
            &\leq\frac{-2}{\eta}\left[\|u(\cdot,T)\|_{L^2}^2-\|u(\cdot, 0)\|_{L^2}^2\right]\\
            &\leq C
            \end{split}
        \end{equation}
        Because this integral has a uniform bound for all $T$. We know that it is increasing in $T$, so it must converge to a limit which is at most $C$ as $T\to \infty$. This means that the tail of the integral can be made arbitrarily small  
    \end{proof}
    \setcounter{appendixLemma}{10}
        \begin{appendixLemma}[Nonlocal Poincar\'e inequality]\label{applem:PositiveEigenvalues}
       Consider a symmetric kernel $K$ with \[\inf_{x\in\Omega}\|K(x,\cdot)\|_{L^1(\Omega)}=\xi>0, \]
      for some $\xi>0$. Suppose also that there exists a $\delta>0$ such that $B_\delta (x)\subset supp(K(x,\cdot))$ for all $x$. Then for any $c>0$, and any bounded domain $\Omega$, the quantity
        \[\inf_{u\in L^2(\Omega),\fint_{\Omega} u\, dx =0}\frac{c\int_\Omega\int_\Omega K(x,y)(u(x)-u(y))^2dxdy}{\|u\|_{L^2}^2}>0\]
    \end{appendixLemma}
    This result is given as \cref{lem:PositiveEigenvalues} in the main text.
    \begin{proof}
        The proof is a modification of the proof of Proposition 3.4 in \cite{NonlocalDiffusionProblems}. 
        It is obvious that $\beta\geq0$, so we must show that it is strictly positive. Consider the subspace $H\subset L^2(\Omega)$ which is the orthogonal complement to the constant functions in $\Omega$. (Thus $\fint h=0$ for all $h\in H$). Now consider $T:H\to H$ as
        \begin{equation}
            Tu(x)=c\int_\Omega K(x,y)(u(x)-u(y))dy = 2c\|K(x,\cdot)\|_{L^1}u(x)-c\int_\Omega K(x,y)u(y)dy
        \end{equation}
        We note that this is a symmetric operator. To see this, do two symmetrization procedures as taken in Lemma \ref{applem:dis} and Corollary \ref{Appcor:intColapse} $\langle Tu,v\rangle_{L^2}=\langle u,Tv\rangle_{L^2}$. As in Prop 3.4 of \cite{NonlocalDiffusionProblems}, $T$ is the sum of a compact operator ($c\int_\Omega K(x,y)u(y)dy$) and an invertible operator ($c\|K(x,\cdot)\|_{L^1}u(x))$, which is invertible when $\inf_{x\in\Omega}\|K(x,\cdot)\|_{L^1(\Omega)}>0$). We can evaluate the spectrum of this operator using Proposition 6.9 from Brezis \cite{Brezis2011}. Because of symmetry (equivalent to self-adjointness here), $\sigma(T)\subset[m,M]$ where $m=\inf_{u\in H,\|u\|_{L^2}=1}\langle Tu,u\rangle_{L^2}$ and $M=\sup_{u\in H,\|u\|_{L^2}=1}\langle Tu,u\rangle_{L^2}$.

        Notice that
        \begin{equation}
            \begin{split}
                m&=\inf_{u\in H,\|u\|_{L^2}=1}2c\int_\Omega\int_\Omega K(x,y)(u(x)-u(y))u(x)dx\\
                &=\inf_{u\in H}\frac{c\int_\Omega\int_\Omega K(x,y)(u(x)-u(y))^2dxdy}{\|u\|_{L^2}^2}
            \end{split}
        \end{equation}
        Finally, we assert that $m>0$ since otherwise $T$ has a non-trivial Kernel in $H$, so there would be a non-constant $u$ for which $\int K(x,y)(u(x)-u(y))dy=0$. We will prove later (as a corollary to \cref{THM:StationarySolutions} that there are no non-constant solutions to $\int K(x,y)(u(x)-u(y))dy=0$ when $K$ satisfies the conditions of the hypothesis. The proof is complete.   
    \end{proof}
    \setcounter{appendixTheorem}{11}
    \begin{appendixTheorem}[Convergence of solutions in a special case]\label{appthm:conv}
        Let $K(x,y)$ be a symmetric kernel with $\inf_{x\in \Omega}\|K(x,\cdot)\|_{L^1(\Omega)}=\xi>0$ and $(B_\delta(x)\cap\Omega)\subset supp(K(x,\cdot))$ for some positive $\delta$ for all $x$. If $\rho$ is even, and $\rho'(z)z<0$ away from $0$, then the solution $u$ to the IVP $u_t=g[u]$ with $u(\cdot, 0)=u_0\in C^0_b(\Omega)$ satisfies 
        \[\lim_{t\to \infty}\left\|u(\cdot,t)-\fint_\Omega u_0\right\|_{L^2(\Omega)}= 0 \]
    \end{appendixTheorem}
    This result is given as \cref{thm:conv} in the main text.
    \begin{proof}
        We first argue that $\frac{d}{dt}\fint_\Omega u(x,t)dx =0$. This is done through a simple computation which takes advantage of the symmetries in $K$ and $\rho$. First, notice that $\rho'(-z)=-\rho'(z)$. Then, after switching the variables of integration, notice that $K(x,y)=K(y,x)$ to see that
        \begin{equation}
        \begin{split}
            \frac{d}{dt}\fint_\Omega u(x,t)dx&=\frac{1}{Vol(\Omega)}\int_\Omega \int_\Omega K(x,y)\rho'(u(x,t)-u(y,t))dydx\\
            &=\frac{-1}{Vol(\Omega)}\int_\Omega\int_\Omega K(x,y)\rho'(u(y,t)-u(x,t))dydx\\
            &=\frac{-1}{Vol(\Omega)}\int_\Omega\int_\Omega K(y,x)\rho'(u(x,t)-u(y,t))dxdy\\
            &=\frac{-1}{Vol(\Omega)}\int_\Omega\int_\Omega K(x,y)\rho'(u(x,t)-u(y,t))dydx\\
                &=-\frac{d}{dt}\fint_\Omega u(x,t)dx             
        \end{split}
        \end{equation}
        Thus $\frac{d}{dt}\fint_\Omega u(x,t)dx=0$, and therefore $\fint_\Omega u(x,t)dx=  \fint_\Omega u(x,0)dx=\fint_\Omega u_0(x)dx$. Now, because $g[u]$ is invariant under vertical translation, it is no loss of generality to consider only $u_0\in C^0_b$ so that $\fint_\Omega u_0(x)dx=0$. Thus, for all $t\geq 0$, $\fint_\Omega u(x,t)dx = 0.$

    Now that we have established this, we can consider our assumptions about $\rho$ more closely. Because $u_0$ is bounded, we know that $|u(x,t)|\leq M$ for all $(x,t)\in \Omega\times[0,\infty)$. Therefore, $|u(x,t)-u(y,t)|\leq2M$ and we can consider $\rho$ only on the closed interval $I:=[-2M,2M]$. Because $\rho'(z)z<0$ away from $z=0$ we know that there is a $c_1>0$ such that $\rho'(z)z\leq -c_1 z^2$. With this information, we claim the following:
    \begin{equation}\label{beta-defn}
        \beta=\inf_{u\in L^2(\Omega)}
\frac{-\int_\Omega\int_\Omega K(x,y)\rho'(u(x)-u(y))(u(x)-u(y))dydx}{\int_\Omega u(x)^2dx}>0\end{equation}
    The non-negativity of this quantity is obvious because $-\rho'(z)z\geq 0$. We need only show that it is bounded away from $0$. In order to do this, we notice that 
    \begin{equation}
        -\int_\Omega\int_\Omega K(x,y)\rho'(u(x)-u(y))(u(x)-u(y))dydx\geq c_1\int_\Omega\int_\Omega K(x,y)(u(y)-u(x))^2dydx
    \end{equation}
    and so we know that 
    \begin{equation}
        \beta\geq \gamma:=\inf_{u\in L_2(\Omega)}\frac{c_1\int_\Omega\int_\Omega K(x,y)(u(y)-u(x))^2dydx}{\int_\Omega u(x)^2dx}
    \end{equation}
    We can refer to Lemma \ref{applem:PositiveEigenvalues}, which
    is applicable because $\fint_\Omega u(x, t)dx = 0$. 
    This result relies on the fact that $K$ satisfies each of the properties in the hypothesis, and that there is no information outside the bounded domain $\Omega$.

    With this information, we can do a quick computation
    \begin{equation}
        \begin{split}
            H(t)&=\frac{1}{2}\int_\Omega u(x,t)^2dx\\
            H'(t)&=\int_\Omega u(x,t)\frac{\partial}{\partial t}u(x,t)dx\\
            &=\int_\Omega\int_\Omega K(x,y)\rho'(u(x,t)-u(y,t))u(x,t)dydx\\
            &=\frac{1}{2}\int_\Omega \int_\Omega K(x,y)\rho'(u(x,y)-u(y,t))(u(x,t)-u(y,t))dydx\\
        \end{split}
    \end{equation}
    Note crucially that $H'(t)\leq 0$. Moreover, we observe that 
    \[\frac{-H'(t)}{H(t)}\geq\beta>0, \]
    where we use $\beta$ from \eqref{beta-defn}.  That is, 
    $H'(t)\leq -\beta H(t)$. 
    A quick application of Gr\"onwall's inequality will tell us that $H(t)\leq H(0)e^{-\beta t}$ and because $\beta>0$, we know that $H(t)\to 0$ and $t\to \infty$. Thus we have shown that if $\fint_\Omega u_0(x)dx=0$ then 
    \[\lim_{t\to \infty}\|u(\cdot,t)\|_{L^2(\Omega)}=0. \]
    To complete the proof, we need only consider a solution $u$ as 
    \[u=v+\fint_\Omega u(x)dx\]
    Because we have already shown the conservation of mass, we know that $v_t=g[v]$ and $\fint_\Omega v(x)dx=0$. Thus we know that $\|v(\cdot, t)\|_{L^2}\to 0$ and so $\|u(\cdot, t)\|_{L^2}\to \fint _\Omega u_0(x)dx$. This completes the proof.
    \end{proof}

    \setcounter{section}{5}
    \setcounter{appendixLemma}{0}
    \begin{appendixLemma}[Consistency of forward Euler scheme]\label{APPLEM:Consistency}
        The numerical scheme \eqref{ForwardEuler} is consistent to order $\tau+h$ for the nonlocality \eqref{nonlocality} with a bounded $w\in C^{1,1}$ with no boundary data when $K\in C^0_b(\Omega;C^{0,1}_0)$ with a uniform bounds on $\|K(x,\cdot)\|_\infty$ and on the Lipschitz constant for $K(x,\cdot)$.
    \end{appendixLemma}
    This lemma appears as \cref{LEM:Consistency} in the main text
    \begin{proof}
          First we compute the error of the right-hand numerical quadrature
          
        \begin{equation}\label{quadrature}\mathcal{G}^h[w](\mathbf{x},t_i):=\sum_{\mathbf{y}\in ^-\Omega^h}K(\mathbf{x},\mathbf{y})\rho'(w(\mathbf{x},t_i)-w(\mathbf{y},t_i))h^n\end{equation}
        for a bounded $w\in C^{0,1}$. Let $\mathbf{y}\in \,^-\Omega^h$ and observe that in the hyperrectangle $\omega_{\mathbf{y}}:=\prod_{k=1}^n[\mathbf{y}_k,\mathbf{y}_k+h]$, $|w(y,t_i)-w(\mathbf{y},t_i)|\leq Lc_nh$ for all $y\in \omega_{\mathbf{y}}$, where $L$ is the Lipschitz constant for $w$ and $c_n$ is a constant depending on the dimension of $\Omega$.  

        As such, for any $y\in \omega_{\mathbf{y}},$ $\mathbf{x}\in \Omega^{h},$ and $\mathbf{y}\in \,^-\Omega^h$, we have $|\rho'(w(\mathbf{x},t_i)-w(y,t_i))-\rho'(w(\mathbf{x},t_i)-w(\mathbf{y},t_i))|\leq L_\rho Lc_nh$ where $L_\rho$ is the Lipschitz constant for $\rho'$ on the interval containing the compact range of $w$, $[-2\|w\|_\infty,2\|w\|_\infty]$. By assumption,  $\|K(\mathbf{x},\cdot)\|_\infty\leq B<\infty$ uniformly over $\Omega^h$ and has a Lipschitz constant, $L_K<\infty$ which is an appropriate Lipschitz constant for $K(x,\cdot) \, \forall x\in \Omega$.  This, along with the fact that $\rho'$ attains its maximum on $[-2\|u\|_\infty,2\|u\|_\infty]$, which we call $W$,  allows us to do the following computation. First note that $K(\mathbf{x},\mathbf{y})\rho'(u(\mathbf{x},t_i)-u(\mathbf{y},t_i))h^n=\int_{\omega_y}K(\mathbf{x},\mathbf{y})\rho'(u(\mathbf{x},t_i)-u(\mathbf{y},t_i))dy$. Thus we write,
        \begin{equation*}
            \begin{split}
                \bigg|\int_{\omega_\mathbf{y}}&K(\mathbf{x},y)\rho'(u(\mathbf{x},t_i)-u(y,t_i))dy-K(\mathbf{x},\mathbf{y})\rho'(u(\mathbf{x},t_i)-u(\mathbf{y},t_i))h^n\bigg|\\
                &\leq\int_{\omega_\mathbf{y}}\bigg|K(\mathbf{x},y)\rho'(u(\mathbf{x},t_i)-u(y,t_i))dy-K(\mathbf{x},\mathbf{y})\rho'(u(\mathbf{x},t_i)-u(\mathbf{y},t_i)\bigg|dy\\
                &\leq \int_{\omega_\mathbf{y}}|K(\mathbf{x},y)-K(\mathbf{x},\mathbf{y})||\rho'(u(\mathbf{x},t_i)-u(y,t_i))|dy\\
                &\quad \quad +\int_{\omega_{\mathbf{y}}}|K(\mathbf{x},\mathbf{y})||\rho'(u(\mathbf{x},t_i)-u(y,t_i))-\rho'(u(\mathbf{x},t_i)-u(\mathbf{y},t_i))|dy\\
                &=:I_1+I_2
            \end{split}
        \end{equation*}
        As before, we will deal with each integral separately. The second integral is controlled by $I_2\leq L_\rho Lc_nh\int_{\omega_\mathbf{y}}|K(\mathbf{x},\mathbf{y})|dy\leq B L_\rho Lc_nh^{n+1}$. Now notice that for $I_1$
          \begin{equation*}
              \begin{split}
                  I_1&\leq W\int_{\omega_\mathbf{y}}|K(\mathbf{x},y)-K(\mathbf{x},\mathbf{y})|dy\\
                    &\leq W\int_{\omega_\mathbf{y}}L_Khdy\\
                    &\leq WL_Kh^{n+1}
                \end{split}
          \end{equation*}
          Thus, the original difference is controlled by 
          
         \[C_1h^{n+1}:=(WL_K+BL_\rho Lc_n)h^{n+1}\]
        Now when we sum across every $\mathbf{y}\in ^-\Omega^h_T$ we see that 
          \begin{equation*}
              |g[w](\mathbf{x},t_i)-\mathcal{G}^h[w](\mathbf{x},t_i)|\leq C_1h^{n+1}\frac{1}{h^{n}}
          \end{equation*}
        Therefore, we have shown that 
        
       \[g[w]=\mathcal{G}^h[w](\mathbf{x},t_i)+\mathcal{O}(h)\]

        Now, suppose $w$ is continuously differentiable in time. It is easy to see that the forward difference 
        
       \[\frac{\partial}{\partial t}w(\mathbf{x},t_i)=\frac{w(\mathbf{x},t_i)-w(\mathbf{x},t_{i+1})}{\tau}+\mathcal{O}(\tau).\] Thus we can say that the numerical scheme \eqref{ForwardEuler} is consistent to order $\tau+h$ for functions $w$ which are bounded, $C^1$ in time, and $C^{0,1}$ in space. 
    \end{proof}

     \begin{appendixLemma}[Lipschitz continuity of $\mathcal{G}^h$]\label{APPLEM:DiscreteLipschitz}
        For any  bounded (in the sup norm sense) subset of $X_R:=\{w\in \mathcal{V}(\Omega^{(h,\tau)}_T);\|w\|_{\ell^\infty}\leq R\}$, there is a $C^\mathcal{G}\geq 0$ such that for every $w_1, w_2\in X_{R}$
        
       \[\|\mathcal{G}[w_1](\cdot, t_i)-\mathcal{G}[w_2](\cdot, t_i)\|_{\ell^\infty\left(\Omega^h\right)}\leq C^\mathcal{G}\|w_1(\cdot, t_i)-w_2(\cdot, t_i)\|_{\ell^\infty\left(\Omega^h,\right)}\]
        so long as $K\in C^0_b(\Omega;C^{0,1}_0)$ with a uniformly bounded $\|K(x,\cdot)\|_\infty$ and $\mathcal{G}$ is defined as in \eqref{quadrature}.
    \end{appendixLemma}
    This lemma is given as \cref{LEM:DiscreteLipschitz} in the main text.
    \begin{proof}
        Consider $w_1, w_2\in X_R$ and observe that $\|\mathcal{G}[w_1]-\mathcal{G}[w_2]\|_{\ell^\infty}$
        \begin{equation*}
        \begin{split}
            &\leq \max_{\mathbf{x}\in \Omega^h}\left\{\sum_{\mathbf{y}\in ^-\Omega^h}|K(\mathbf{x},\mathbf{y})||\rho'(w_1(\mathbf{x},t_i)-w_1(\mathbf{y},t_i))-\rho'(w_2(\mathbf{x},t_i)-w_2(\mathbf{y},t_i))|h^n\right\}\\
            &\leq \max_{\mathbf{x}\in \Omega^h}\left\{\|\rho'(w_1(\mathbf{x},t_i)-w_1(\cdot,t_i))-\rho'(w_2(\mathbf{x},t_i)-w_2(\cdot,t_i))\|_{\ell^\infty(\Omega^h)}\sum_{y\in^-\Omega^h}|K(\mathbf{x}-\mathbf{y})|h^n\right\}
            \end{split}
        \end{equation*}
        Notice that, because the domain in question has volume $vol(\Omega)=1$, we have \begin{equation*}
        \begin{split}
        \sum_{\mathbf{y}\in^-\Omega^h}|K(\mathbf{x},\mathbf{y})|h^n &\leq \sum_{\mathbf{y}\in \Omega^{h}}\|\pi^hK(\mathbf{x},\mathbf{y})\|_{\ell^\infty(\Omega^{h})}h^n\\
        &\leq\|\pi^h K(\mathbf{x},\cdot)\|_{\ell^\infty(\Omega^h)}\sum_{\mathbf{y}\in^-\Omega^h}h^n\\
        &\leq \|K(\mathbf{x},\cdot)\|_{\ell^\infty(\Omega^h)}vol(\Omega)\\
        &\leq B:=\sup_{\mathbf{x}\in\Omega^h}\|K(\mathbf{x},\cdot)\|_{\ell^\infty(\Omega^h)}
        \end{split}
        \end{equation*}
        for all $\mathbf{x}\in \Omega^h$ and so we proceed with a very similar argument about $\rho'$ as in \cref{APPLEM:Lipschitzg}.

        \begin{equation*}
            \begin{split}
                \|\rho'(w_1(\mathbf{x},t_i)-&w_1(\cdot, t_i))-\rho'(w_2(\mathbf{x},t_i)-w_2(\cdot,t_i))\|_{\ell^\infty(\Omega^h)}\\
                &\leq L_\rho\|w_1(\mathbf{x},t_i)-w_1(\cdot,t_i)-w_2(\mathbf{x},t_i)+w_2(\cdot,t_i)\|_{\ell^\infty(\Omega^h)}\\
                &\leq L_\rho\left(|w_1(\mathbf{x},t_i)-w_2(\mathbf{x},t_i)|+\|w_1(\cdot,t_i)-w_2(\cdot,t_i)\|_{\ell^\infty(\Omega^h)}\right)
            \end{split}
        \end{equation*}
        So naturally we see, as before, that
        \begin{equation*}
        \begin{split}
        \max_{\mathbf{x}\in \Omega^h}\|\rho'(w_1(\mathbf{x},t_i)-&w_1(\cdot, t_i))-\rho'(w_2(\mathbf{x},t_i)-w_2(\cdot,t_i))\|_{\ell^\infty(\Omega^h)}\\
        &\leq 2L_\rho\|w_1(\cdot, t_i)-w_2(\cdot, t_i)\|_{\ell^\infty(\Omega^h)}
        \end{split}
        \end{equation*}
        Therefore, we get the result that
        
       \[\|\mathcal{G}[w_1](\cdot ,t_i)-\mathcal{G}[w_2](\cdot, t_i)\|_{\ell^\infty(\Omega^h)}\leq 2L_\rho M\|w_1(\cdot, t_i)-w_2(\cdot, t_i)\|_{\ell^\infty(\Omega^h)}\]
        Let $C^\mathcal{G}=2L_\rho M$ to complete the proof. Again $L_\rho$ may depend on $R$ so $C^\mathcal{G}$ will depend on $R$. 
    \end{proof}

    \setcounter{appendixTheorem}{2}
    \begin{appendixTheorem}[Convergence of the forward Euler Scheme]\label{APPTHM:Convergence}
    In a domain with discretization \eqref{hypercube}, the IVP $u_t=g[u]$ on $\Omega$, with $u(x,0)=u_0\in C^{0,1}_b(\Omega)$, and with $\rho$ satisfying \eqref{coordinationcondition}, the numerical scheme \eqref{ForwardEuler} is convergent in the case where $K\in C^0_b(\Omega;C^{0,1}_0)$ with uniformly bounded $\|K(x,\cdot)\|_\infty$ and a uniform global Lipschitz constant.
    \end{appendixTheorem}

    The statement of this theorem is included as \cref{THM:Convergence} in the main text.
    \begin{proof}
        Let $\pi_{h,\tau}:C^0(\Omega_T)\rightarrow \mathcal{V}(\Omega^{(h,\tau)}_T)$ discretize $u$, a solution of the IVP $u_t=g[u]$ with $u(x,0)=u_0\in C^{0,1}_b(\Omega)$. Let $w$ be a grid function which solves the numerical scheme \eqref{ForwardEuler} with $w(\mathbf{x},0)=\pi_h u_0$. Further let $e=\pi_{h,\tau} u- w$. Now notice that 
        \begin{equation}
            \begin{split}
            \frac{e(\mathbf{x},t_{i+1})-e(\mathbf{x},t_{i})}{\tau}&=\frac{u(\mathbf{x},t_{i+1})-u(\mathbf{x},t_{i})}{\tau}- \mathcal{G}^h[w](\mathbf{x},t_i)\\
            &=u_t(\mathbf{x},t_i)-\mathcal{G}^h[w](\mathbf{x},t_i)+\mathcal{O}(\tau)\\
            &=g[u](\mathbf{x},t_i)-\mathcal{G}^h[w](\mathbf{x},t_i)+\mathcal{O}(\tau)\\
            &=\mathcal{G}^h[u](\mathbf{x},t_i)-\mathcal{G}^h[w](\mathbf{x},t_i)+\mathcal{O}(\tau+h)
            \end{split}
        \end{equation}
        by \cref{APPLEM:Consistency}. $\mathcal{G}^h$ is not linear, so the stability of the method does not complete the proof. Instead, we use \cref{APPLEM:DiscreteLipschitz} to show that
        \begin{equation*}
            \begin{split}
        e(\mathbf{x},t_{i+1}) &=e(\mathbf{x},t_{i})+\tau\left(\mathcal{G}^h[u](\mathbf{x},t_i)-\mathcal{G}^h[w](x,t_i)+\mathcal{O}(\tau+h)\right)\\
        \|e(\cdot, t_{i+1})\|_{\ell^\infty(\Omega^h)}&\leq\|e(\cdot, t_{i})\|_{\ell^\infty(\Omega^h)}+ \tau \|\mathcal{G}^h[u](\mathbf{x},t_i)-\mathcal{G}^h[w](x,t_i)\|_{\ell^\infty(\Omega^h)}+\tau|\mathcal{O}(\tau+h)|\\
        &\leq\|e(\cdot, t_{i})\|_{\ell^\infty(\Omega^h)}+ \tau C^\mathcal{G}\|\pi^h u(\mathbf{x},t_i)-w(x,t_i)\|_{\ell^\infty(\Omega^h)}+\tau|\mathcal{O}(\tau+h)|\\
        &\leq (1+\tau C^\mathcal{G})\|e(\cdot, t_i)\|_{\ell^\infty(\Omega^h)}+\tau|\mathcal{O}(\tau +h)|
        \end{split}
        \end{equation*}

        We now employ Gr\"onwall's lemma in the discrete forward difference case to say that
        \begin{equation*}
            \|e(\cdot, t_{i+1})\|_{\ell^\infty(\Omega^h)}\leq (1+\tau C^\mathcal{G})^i\|e(\cdot, 0)\|_{\ell^\infty(\Omega^h)}+\frac{1}{C^\mathcal{G}}((1+\tau C^\mathcal{G})^i-1)C(\tau+h)
        \end{equation*}
        Naturally, $e(\cdot, 0)\equiv 0$ and for finite time $T$ we have a maximum $i\leq T/\tau$ so we can say that
        \begin{equation*}
            \|e\|_{\ell^\infty(\Omega^{(\tau,h)}_T)}\leq \tilde{C}((1+\tau C^\mathcal{G})^\frac{T}{\tau}-1)(\tau+h)
        \end{equation*}
        Observe that $C^\mathcal{G}$ does not depend on $h$ or $\tau$. It only depends on $R$, but because of \cref{APPLEM:MaximumPrinciple} and its discrete analogue 
        , $R$ is the same at each time step and thus so is $C^\mathcal{G}$. Also observe that $\lim_{\tau\rightarrow 0}(1+\tau C^\mathcal{G})^\frac{T}{\tau}$ exists, and so this quantity is bounded for any sufficiently small $\tau$ and indeed as $(\tau, h)\rightarrow \mathbf{0}$, $\|e\|_{\ell^\infty}\rightarrow 0$.
    \end{proof}

    \setcounter{section}{7}
    \setcounter{appendixTheorem}{1}

    \begin{appendixTheorem}[Stationary solutions are simple functions]\label{APPTHM:StationarySolution2}
        Let $\Omega$ be a bounded domain in $\mathbb{R}^n$. Let $K(x,y)\in C_b^0(\overline\Omega;L^1(\Omega))$ with $\inf\|K(x,\cdot)\|_{L^1(\Omega)}=\xi>0$ and there a $\delta$ so that $B_{\delta}(x)\subset supp(K(x,\cdot))$ for all $x$.  Finally, let $\rho'$ satisfy \eqref{coordinationcondition} and have support $(-a,a)\subset \mathbb{R}$, with the assumption that $\rho'$ has only one zero in this interval at $z=0$ (see \cref{fig:rhoprime}). If $u\in L^1(\Omega)\cap L^\infty(\Omega)$ is a solution to $g[u]=0$ in $\Omega$, then the image, $u(\Omega)$, is a finite set of points except possibly at a set of measure 0.
    \end{appendixTheorem}   
    The statement of this theorem is given as \cref{THM:StationarySolutions} in the main text.
    \begin{proof}
        Let $M_0=\sup_\Omega u(x)$. Whether or not it is attained, there is a sequence $x_k$ such that $u(x_k)\to M_0$. Moreover, because $\Omega$ is bounded, there is a subsequence $\{x_j\}\subset \{x_k\}$ which converges to a point in $\overline\Omega$ which we call $x_0$. Thus $u(x_j)\to M_0$ as $x_j\to x_0$.

        For each $j$ separate $\Omega$ into the following sets
        \begin{equation*}
        \begin{split}
            \Omega^+_j&:=\{x\in \Omega, u(x)>u(x_j)\}\\
            \Omega^o_j&:=\{x\in \Omega; u(x)=u(x_j)\}\\
            \Omega^-_j&:=\{x\in \Omega; u(x)<u(x_j)\}
        \end{split}
        \end{equation*}
        and so $g[u]$ is partitioned into three parts
        \begin{equation*}  
            \begin{split} 
                g[u](x_j)& =\underbrace{\int_{\Omega^+_j}K(x_j,y)\rho'(u(x_j)-u(y))dy}_{:=I_j^+\geq 0}+\underbrace{\int_{\Omega^o_j}K(x_j,y)\rho'(u(x_j)-u(y))dy}_{:=I_j^o=0}\\
                &\quad +\underbrace{\int_{\Omega^-_j}K(x_j,y)\rho'(u(x_j)-u(y))dy}_{:=I_j^-\leq0}\end{split}\end{equation*}
        Note that for every $j$, $I^+_j=-I^-_j$ because $g[u]=0$. 

        Let $\beta>0$ and note that we can always find a $J_\beta$ so that $j>J_\beta$ implies $u(x_j)\geq M_0-\frac{\beta}{C^kL_\rho}$ and $|x_j-x_0|<\beta$ where $L_\rho$ is the Lipschitz constant for $\rho'$. Because $\rho'$ is Lipschitz continuous and $\rho'(0)=0$, if $\frac{-\beta}{C^KL_\rho}<u(x_j)-u(y)<0$ for all $y\in \Omega_j^+$, then $\rho'(u(x_j)-u(y))\leq \frac{\beta}{C^K}$ for all $y\in \Omega^+_j$. This means that for any $\beta$ we can find a $J_\beta$ so that when $j>J_\beta$, we have $|x_j-x_0|<\beta$ and 
        \[I_j^+\leq \|K(x_j,\cdot)\|_{L^1(\Omega)}\frac{\beta}{C^K}\leq\beta.\]\
        Likewise when $j>J_\beta$ we know that $|I_j^-|\leq \beta$.

        Let $\epsilon>0$ and Now assume that there is a set $S$ so that $\mu(S\cap supp(K(x_0,\cdot))>0$. This is equivalent to saying that $\int_\Omega K(x,y) \chi_S(y)dy=\eta>0$. Further assume that $u(S)\in(M_0-a+\epsilon,M_0-\epsilon)$. 
        
        There is a $J_1$ such that $u(x_j)>M_0-\frac{\epsilon}{2}$ for all $j>J_1$. This means that for all $j>J_1$, because of the properties of $\rho'$ we know that $\frac{\epsilon}{2}<u(x_j)-u(y)<a-\frac{\epsilon}{2}$ for all $y\in S$. Therefore, there is a $\zeta_\epsilon$ such that $\rho'(u(x_j)-u(y))<\zeta_\epsilon$ for all $y\in S$ when $j>J_1$. 
        
        There is a $J_2$ such that $j>J_2$ implies that $\|K(x_j,\cdot)-K(x_0,\cdot)\|_{L^1(\Omega)}<\xi/2$. More specifically this means that for any $S$ we can write $\|K(x_j,y)\chi_S-K(x_0,y)\chi_S\|_{L^1(\Omega)}\leq \xi/2$ 
        
        Lastly, there is a $J_3$ such that $|I_j^+|=|I_j^-|<\zeta_\epsilon \frac{\xi}{4}$ for all $j\geq J_3$. 

        This means that for $j>J=\max\{J_1,J_2,J_3\}$ we know that $|I^-_j|<\zeta_\epsilon\frac{\xi}{4}$ and 
        \begin{align*}
            I^-_J& =\int_{S}K(x,y)\rho'(u(x_j)-u(y))dy\\
                &\geq \zeta_\epsilon\int_SK(x,y)dy\\
                &\geq \zeta_\epsilon(\|K(x_0,\cdot)\|_{L^1(S)}-\xi/2)\\
                &\geq\zeta_\epsilon\frac{\xi}{2} 
        \end{align*}
        This contradiction means that no such $S$ can exist. The same can be said for any $\epsilon$, and thus we conclude that for any set $S$ with $\mu(S\cap supp(K(x_0,\cdot)))>0$, $u(y)=M_0$ or $u(y)\leq M_0-a$ for all $y\in S$. This means that, if $x\in \Omega\cap supp(K(x_0,\cdot))$ with $u(x)=M_0$ then for almost every player $y\in \Omega$ either $K(x,y)=0$ or $\rho'(u(x)-u(y))=0$. Equivalently, let $\Omega_0^0=\{y\in\Omega\cap supp(K(x_0,\cdot)); u(y)=M_0\}$ then, for almost every $x\in \Omega$, $\int_{\Omega_0^0}K(x,y)\rho'(u(x)-u(y))dy=0.$

        This means that we can repeat the same process on the reduced domain $\Omega_1=\Omega\setminus\Omega_0^0$ and get a similar result. We say that $M_1=\esssup \{u(x); x\in\Omega_1\}$ and we note that, for almost every $x\in \Omega_1$ \[g[u](x)=\int_\Omega K(x,y)\rho'(u(x)-u(y))dy=\int_{\Omega_1} K(x,y)\rho'(u(x)-u(y))dy\] 

        Of course we still have that $g[u]=0$ so $g_1[u]=\int_{\Omega_1}K(x,y)\rho'(u(x)-u(y))dy=0$ almost everywhere. We repeat our process by finding a subsequence $\{x_j\}\subset\Omega_1$ such that $u(x_j)\to M_1$ and $x_j\to x_1\in \overline{\Omega_1}$. As before, we say
        
        \begin{equation*}  
            \begin{split} 
                g[u](x_j)& =\underbrace{\int_{\Omega^+_{1,j}}K(x_j,y)\rho'(u(x_j)-u(y))dy}_{:=I_{1,j}^+\geq 0}+\underbrace{\int_{\Omega^o_{1,j}}K(x_j,y)\rho'(u(x_j)-u(y))dy}_{:=I_{1,j}^o=0}\\
                &\quad +\underbrace{\int_{\Omega^-_{1,j}}K(x_j,y)\rho'(u(x_j)-u(y))dy}_{:=I_{1,j}^-\leq0}+\underbrace{\int_{\Omega_{0}^0}K(x_j,y)\rho'(u(x_j)-u(y))dy}_{=0}\end{split}\end{equation*}
        and again $|I^+_{1,j}|=|I^-_{1,j}|\to 0$. In the same way, we find a contradiction to this convergence when we assume that there is an $S$ with $u(S)\subset (M_1-a+\epsilon, M_1-\epsilon)$ and $\mu(S\cap supp(K(x_1,\cdot)))>0.$ This means that almost everywhere in $supp(K(x_1,\cdot))$, either $u(x)=M_1$ or $u(x)\leq M_1-a$. Let $\Omega_1^0=\{x\in supp(K(x_1,\cdot)), u(x)=M_1\}$ and we can make the same conclusions as before. For almost every $x\in \Omega$, $\int_{\Omega_1^0}K(x,y)\rho'(u(x)-u(y))dy=0$.

        We repeat this process and decompose $\Omega$ into a set $\{\Omega_k^0; k=0,1,...\}$ each of which have $u(x)=M_k$ in $\Omega_k^0$. Now we need only show that this process terminates after a finite number of steps. In order to do this, we recall that $\Omega$ is bounded, that $\|u\|_{L^\infty(\Omega)}<\infty$, and that $B_\delta(x)\subset supp(K(x,\cdot))$ for all $x$.

        Note that if $|M_j-M_k|<a$ then $d(\Omega_j^0,\Omega_k^0)>\delta$ in the sense that $\inf_{(x,y)\in\Omega_j^0\times \Omega_k^0}(|x-y|)>\delta$. Because $\overline\Omega$ can be covered with a finite number of balls of radius $\delta$, there can only be a finite number of $M_j$ within $(M_k-a,M_k+a)$ for any $k$. Likewise, $[\inf u,\sup u]$ can be covered with intervals of radius $a/2$. In one such interval, there can only be a finite number of $M_k$, and there are only a finite number of such intervals, so there are only a finite number of $M_k$ in $u$.

        Thus we have shown that if $u$ is a stationary solution, $\rho$ satisfies \eqref{coordinationcondition} and behaves as in \cref{fig:rhoprime}, and $K$ satisfies the following conditions: (a) $K(x,y)\in C_b^0(\overline\Omega;L^1(\Omega))$ (b) $\inf\|K(x,\cdot)\|_{L^1(\Omega)}=\xi>0$ (c) there a $\delta$ so that $B_{\delta}(x)\subset supp(K(x,\cdot))$ for all $x$, then the stationary solution is a simple function of the form 
        \[u= \sum_{k=0}^mM_k\chi_{\Omega_k^0}\]

    \end{proof}

    \begin{appendixTheorem}[Stationary solutions when $\Omega\subset supp(K(x,\cdot))$]\label{APPTHM:StationarySolution}
        Let $\Omega$ be a bounded domain in $\mathbb{R}^n$. Let $K(x,y)\in C_b^0(\Omega;L^1(\Omega))$ so that for any $x$, $\Omega\subseteq supp(K(x,\cdot))$ and $\lambda \leq K(x,y)$. Finally, let $\rho'$ satisfy \eqref{coordinationcondition} and have support $(-a,a)\subset \mathbb{R}$, with the assumption that $\rho'$ has only one zero in this interval at $z=0$, see \cref{fig:rhoprime}. If $u\in L^1(\Omega)\cap L^\infty(\Omega)$ is a solution to $g[u]=0$ in $\Omega$, then it is a simple function of the form 
        \[u=\sum_{k=1}^mc_k\chi_{S_i}\]
        for some finite $m$ where $\cup_{k=1}^m S_k=\Omega$ and $S_k\cap S_j=\emptyset$. Also, $|c_k-c_j|\geq a$ for all $k\neq j$.
    \end{appendixTheorem}
    The statement of this theorem is given as \cref{THM:StationarySolution} in the main text.
    \begin{proof}
        First note that by the assumption on $K$, $C^K\geq \|K(x,\cdot)\|_{L^1(\Omega)}>\lambda Vol(\Omega)=:C_K$ for all $x$. 
        Let $M_0=\sup_{\Omega}u(x)$. Whether or not it is attained, there is a sequence of $x_k$ so that $u(x_k)\to M_0$ as $k\to \infty$. For each $k$, we partition the domain into three parts
        \begin{equation*}
        \begin{split}
            \Omega^+_k&:=\{x\in \Omega, u(x)>u(x_k)\}\\
            \Omega^o_k&:=\{x\in \Omega; u(x)=u(x_k)\}\\
            \Omega^-_k&:=\{x\in \Omega; u(x)<u(x_k)\}
        \end{split}
        \end{equation*}
        and so $g[u]$ is partitioned into three parts
        \begin{equation*}  
            \begin{split} 
                g[u](x_k)& =\underbrace{\int_{\Omega^+_k}K(x_k,y)\rho'(u(x_k)-u(y))dy}_{:=I_k^+\geq 0}+\underbrace{\int_{\Omega^o_k}K(x_k,y)\rho'(u(x_k)-u(y))dy}_{:=I_k^o=0}\\
                &\quad +\underbrace{\int_{\Omega^-_k}K(x_k,y)\rho'(u(x_k)-u(y))dy}_{:=I_k^-\leq0}\end{split}\end{equation*}
        Note that for every $k$, $I^+_k=-I^-_k$ because $g[u]=0$. 

        Let $\beta>0$ and note that we can find a $K_\beta$ so that $u(x_k)>M_0-\frac{\beta}{C^KL_\rho}$ for all $k>K_\beta$ where $L_\rho$ is the Lipschitz constant for $\rho'$. Because $\rho'$ is Lipschitz continuous and $\rho'(0)=0$, if $-\frac{\beta}{C^KL_\rho}<u(x_k)-u(y)<0$ for all $y\in \Omega^+_k$, then $\rho'(u(x_k)-u(y))\leq \frac{L\beta}{C^KL}=\frac{\beta}{C^K}$ for all $y\in \Omega^+_k$. 
        This means that for any $\beta$ we can find a $K_\beta$ so that when $k>K_\beta$,
        \[I_k^+<\|K(x_k,\cdot)\|_{L^1(\Omega)}\frac{L\beta}{C^KL}\leq\frac{C^KL\beta}{C^KL}=\beta. \]
        This means that for every $\beta$ sufficiently small, we must see that there exists a $K$ (indeed, the same $K_\beta$ should do) such that $|I_k^-|\leq\beta$ for all $k>K_\beta$.

        Now, consider $\epsilon>0$ sufficiently small and suppose there is an $S\subset\Omega$ with positive measure $\mu(S)>0$ and satisfying the inequality $M_0-a+\epsilon<u(s)<M_0-\epsilon$. Let $K_{\epsilon}$ be the index so that $k>K_{\epsilon}\implies M_0-u(x_k)<\frac{\epsilon}{2}$. Note that when $k>K_{\epsilon}$, $S\subset \Omega ^-_k$. Moreover, because $\frac{\epsilon}{2}\leq u(x_k)-u(s)\leq a-\epsilon$, and $\rho'$ has no zeros in this compact interval, there is a $C_\epsilon$ such that $\rho'(u(x)-u(s))\le C_\epsilon<0$. This means that when $k>K_{\epsilon}$ we can bound $I^-_k$ away from $0$
        \[
        -I^-_k\geq -\int_SK(x,y)\rho'(u(x_k)-u(y))dy\geq -\lambda C_\epsilon \mu(S)>0.
        \]

        This contradicts the fact that $I^-_k\to 0$ as $k\to \infty$ so we conclude that for any subset $S\subset \Omega$ with $M_0-a+\epsilon_0\leq u(S)\leq M_0-\epsilon_0$, we know that $\mu(S)=0$. Restated, we have shown that in the bounded domain $\Omega$, 
        \begin{equation}\label{EQ:StationaryBaseCaseEpsilon} u(s)\geq M_0-\epsilon\quad\text{ or }\quad u(s)\leq M_0-a+\epsilon\quad \text{ almost everywhere.}\end{equation} 

        Importantly, every step taken to arrive at \eqref{EQ:StationaryBaseCaseEpsilon} can still be done when $\epsilon$ is made smaller, and so in the limit as $\epsilon\to 0$ we can conclude that
        
        \begin{equation}\label{EQ:StationaryBaseCase} u(s) = M_0\quad\text{ or }\quad u(s)\leq M_0-a\quad \text{ almost everywhere.}\end{equation} 
        To see the convergence argument, observe that if \eqref{EQ:StationaryBaseCase} was not true we would surely have that $|\{s\in \Omega;M_0-a<u(s)<M_0\}|>0$ and we can write this as the countable union 
        \[|\{s\in \Omega;M_0-a<u(s)<M_0\}|=\bigg|\bigcup_{\epsilon=\frac{1}{n}}\{s\in\Omega; \epsilon-a<u(s)<M_0<\epsilon\}\bigg|>0.\]
        For this to be positive, it must be positive for at least one $\epsilon$, and this could contradict \eqref{EQ:StationaryBaseCaseEpsilon}.

        Now, consider $\{u(x);x\in \Omega, u(x)\neq M_0\}$. If this set has measure 0, then we have completed the proof. If it has positive measure, let $M_1:=\esssup\{u(x);x\in \Omega, u(x)\neq M_0\}$ and note that $M_1\leq M_0-a$. Let $x_k\in \Omega$ so that $u(x_k)\leq M_1$ and $u(x_k)\to M_1$. Now partition the domain into 
        \begin{equation*}
            \begin{split}
            \Omega^{++}_k&:=\{x\in \Omega; u(x)\geq u(x_k)+a\}\\
                \Omega^+_k&:=\{x\in \Omega, u(x_k)<u(x)<u(x_k)+a\}\\
            \Omega^o_k&:=\{x\in \Omega; u(x)=u(x_k)\}\\
            \Omega^-_k&:=\{x\in \Omega; u(x)<u(x_k)\}
            \end{split}
        \end{equation*}
        and again we can split the integral into
        \begin{equation*}
            \begin{split}
                g[u](x_k)& =\underbrace{\int_{\Omega^{++}_k}K(x_k,y)\rho'(u(x_k)-u(y))dy}_{:=I_k^{++}= 0}+\underbrace{\int_{\Omega^+_k}K(x_k,y)\rho'(u(x_k)-u(y))dy}_{:=I_k^+\geq 0}\\
                &\quad+\underbrace{\int_{\Omega^o_k}K(x_k,y)\rho'(u(x_k)-u(y))dy}_{:=I_k^o=0} +\underbrace{\int_{\Omega^-_k}K(x_k,y)\rho'(u(x_k)-u(y))dy}_{:=I_k^-\leq0}
            \end{split}
        \end{equation*}

        In this way, we are left with the same situation as before. $I^{++}_k=0$ because when $u(y)\geq u(x_k)+a$ then $u(x_k)-u(y)\leq -a$ and so $\rho'(u(x_k)-u(y))=0$ for all $y\in \Omega_k^{++}$. Thus, we note that $I^+_k=-I^-_k$ for all $k$. Again, for any $\beta$ we can find a $K_\beta$ so that $I^+_k<\beta$ for all $k>K_\beta$. This is because we proved in the first part of the proof that $u(x)\notin (M_0-a,M_0)$ except possibly on a set of measure $0$. Thus the entire contribution to $I^+_k$ comes from $x\in \Omega^+_k$ with $u(x_k)<u(x)\leq M_1$, and we can repeat our previous argument exactly. Thus $u(x)\notin (M_1,M_1+a)$ except possibly at a set of measure $0$. Because the nonlocality will not see a set of measure $0$, it is an identical argument to extend the result in \eqref{EQ:StationaryBaseCase} to say that in $\Omega$,
        \[u(s)=M_0\quad \text{ or } \quad u(s)=M_1 \quad\text{ or } \quad u(s)\leq M_1-a \quad\text{ almost everywhere}\]

        This argument can be repeated indefinitely, but notice that between each of the resulting bands in $u(\Omega)$ there is a gap of at least $a$. This means that, after a finite number of repetitions, we will cover the entire range of the bounded solution $u$. This means that, except for a possible set $Q$ which has measure 0, the image of $u(\Omega\setminus Q)$ is a finite set of points. Further, we can partition $\Omega$ into $A_0,...,A_m$ where $A_k:=\{x\in \Omega; u(x)=M_k\}$ (by the process described above each of these $A_i$ will have positive measure) and represent the solution $u$ as a simple function
        \[u(x)= \sum_{k=0}^m M_k\chi_{\{A_i=k\}}(x)\in L^\infty(\Omega)\] 
    \end{proof}
    \setcounter{appendixCorollary}{3}
    \begin{appendixCorollary}\label{appcor:linear}
        If $K$ is as in \cref{THM:StationarySolutions} and $\rho'(z)\neq 0$ away from $0$, Then if $g[u]=0$ $u$ is constant. 
    \end{appendixCorollary}
    This result is stated as \cref{cor:linear} in the main text.
    \begin{proof}
        By the result of \cref{APPTHM:StationarySolution2}, we know that $u$ is a simple function. Suppose that it is not constant and pick a point $x$ so that $u(x)=\esssup u$ and so that there is a $R\subset B_\delta(x)$ with positive measure where $u(R)=c< u(x)$. Such a point is always possible to find if $u$ is not constant. Because $u(x)\geq u(y)$ for almost every $y$, $\rho'(u(x)-u(y))\leq 0$ for almost every $y$. However
        \[\int_\Omega K(x,y)\rho'(u(x)-u(y))dy=\int_RK(x,y)\rho'(u(x)-c)dy+\int_{R^c}K(x,y)\rho'(u(x)-u(y))dy\]
        The first term is strictly negative because $R\subset B_\delta(x)$ and $\rho'(u(x)-c)<0$. The second is clearly non-positive. Therefore $g[u](x)\neq 0$. This same argument can be made for $x$ in $B_{\eta}(x)$ for some $\eta$, and so $g[u]\neq 0$ almost everywhere.  
    \end{proof}
    \setcounter{appendixTheorem}{5}
    \begin{appendixTheorem}[LNE are dense among stationary solutions]\label{APPTHM:LNE}
        For any $\epsilon$. If $u$ is a stationary solution with $g[u]=0$ in $L^1(\Omega)\cap L^\infty(\Omega)$, then there is a $u'$ such that $g[u']=0$, $u'$ is stable in the $L^\infty$ sense. $u'$ is an LNE, and $\|u'-u\|_{L^1(\Omega)<\epsilon}$.
    \end{appendixTheorem}
    This theorem is stated as \cref{thm:LNE}

    \begin{proof}
        Let $\epsilon>0$ and let $u$ be a stationary solution. Thus by \cref{APPTHM:StationarySolution2} we can say that $u=\sum_{k=1}^m c_k\chi_{S_k}$ where $\{S_k\}_{k=1}^m$ partitions $\Omega$. Now let $\delta>0$ to be determined later and let $u'=(1+\delta)u$. Thus $u'=\sum_{k=1}^m(1+\delta)c_k\chi_{S_k}$. First, we will show that for any $\delta$, $u'$ is a stationary solution. Then we will show that for $\delta<\delta^*$, $u'$ is a local Nash equilibrium. We will also show that $u'$ is stable in the $L^\infty$ sense. Finally we will show that there is a $\delta$ such that $\|u-u'\|_{L^1(\Omega)}<\epsilon.$

        First observe that for any $x$ in $u'$, $g[u'](x)=\sum_{k=1}^m\int_{S_k}K(x,y)\rho'((1+\delta)(c_j-c_k))$ where $x\in S_j$. Because $u$ is a stationary solution, by the proof of \cref{APPTHM:StationarySolution2} we know that, for almost every $x,y\in \Omega$ either $K(x,y)=0$ or $\rho'(u(x)-u(y))=0$. Because of the properties of $\rho'$, if $\rho'(u(x)-u(y))=0$ then either $|u(x)-u(y)|=0$ or $|u(x)-u(y)|>a$. Therefore $|(1+\delta)(u(x)-u(y))|=0$ or $|(1+\delta)(u(x)-u(y))|>(1+\delta)a>a$ Therefore, for almost every $x,y\in \Omega$, $K(x,y)=0$ or $\rho'((1+\delta)(u(x)-u(y)))=0$. Therefore, for almost every $x\in \Omega$, $g[u'](x)=0$ and $u'$ is a stationary solution.

        Now we will show that $u'$ is a local Nash equilibrium. Crucially, we mean that the unique representative of the equivalence class $u'=\sum_{k=1}^m(1+\delta)c_k\chi_{S_k}$ for which the equality holds pointwise is an LNE. Consider a focal individual, $x_0$, Let $\epsilon<\delta a/2$ and consider $\hat u=u'+\eta\chi_{\{x_0\}}$. We compute
        \begin{equation}\begin{split}\label{LNE1}
            w(x_0|\hat u)-w(x_0|u')&=\int_\Omega K(x_0,y)(\rho(\hat u(x_0)-\hat u(y)-\rho(u'(x_0)-u'(y)))dy\\
            &=\sum_{k=1}^m\int_{S_k}K(x_0,y)(\rho((1+\delta)(c_j-c_k)+\eta)-\rho((1+\delta)(c_j-c_k)))dy
        \end{split}\end{equation}
        where $j$ is such that $x_0\in S_j$. Notice that if $|c_j-c_k|<a$ then $K(x,y)=0$ for all $x\in S_j$ and $y\in S_k$ by the proof of \cref{APPTHM:StationarySolution2}. This means that if $K(x_0,y)>0$ for any $y\in S_k$ then $|c_j-c_k|>a$. In this case $|(1+\delta)(c_j-c_k)|>(1+\delta)a$ and  $|(1+\delta)(c_j-c_k)+\eta|>(1+\frac{\delta}{2})a$. Because both of these are outside of the support of $\rho'$ we can say $\rho((1+\delta)(c_j-c_k)+\eta)-\rho((1+\delta)(c_j-c_k))=0$ whenever $|c_j-c_k|>a$. The result is that the integrand of equation \eqref{LNE1} is $0$ for all $k\neq j$ and
        \[w(x_0|\hat u)-w(x_0|u')=\int_{S_j}K(x,y)(\rho(\eta)-\rho(0))\]
        Because of our assumptions on $\rho$, we know that $\rho(\eta)<\rho(0)$, so whenever $\eta<\frac{\delta}{2}$, no player can change their strategy by less than $\eta$ to improve their payoff, and $u'$ is an LNE. 
        
        The LNE concept is restrictive to the kinds of perturbations allowed. To make a claim about a wider class of perturbations, let $\hat u = u'+\eta \varphi$ where $0<\eta<\delta a/2$ and $\varphi\in L^1(\Omega)\cap  L^\infty(\Omega)$ with $\|\varphi\|_{L^\infty}<1$. 
        \begin{equation*}
            \begin{split}
                g[\hat u](x)&=\int_\Omega K(x,y)\rho'(u'(x)-u'(y)+\eta(\varphi(x)-\varphi(y)))dy\\
                &=\sum_{k=1}^m\int_{S_k}K(x,y)\rho'((1+\delta)(c_j-c_k)+\eta(\varphi(x)-\varphi(y)))dy
            \end{split}
        \end{equation*}
        where $j$ is such that $x\in S_j$. Notice that for $k\neq j$ $|(1+\delta)(c_j-c_k)|>(1+\delta)a$ and $|\eta(\varphi(x)-\varphi(y))|<\delta a/2$ so $|(1+\delta)(c_j-c_k)+\eta(\varphi(x)-\varphi(y))|>(1+\frac{\delta}{2})a>a$ so the integrand is $0$ almost everywhere for all $k\neq j$. Thus
        \begin{equation*}
            \begin{split}
                g[\hat u](x)&=\int_{S_j}K(x,y)\rho'(\eta(\varphi(x)-\varphi(y)))dy
            \end{split}
        \end{equation*}

        Now notice that this means that $\frac{\partial}{\partial t}\eta\varphi(x)=\int_{S_j}K(x,y)\rho'(\eta(\varphi(x)-\varphi(y)))dy$ and thus $\varphi$ solves an IVP of the form $\varphi_t=g[\varphi]$ so it obeys the maximum principle. This proves $L^\infty$ stability of $\Omega$ because if $\|\hat u(\cdot, 0)-u'\|_{L^\infty(\Omega)}<\frac{\delta a}{2}$ then $\|\hat u(\cdot, t)-u'\|_{L^\infty(\Omega)}<\frac{\delta a}{2}$ for all $t>0$. 

        We have shown that for any $u$ we can construct a related $u'$ with parameter $\delta>0$ such that $u'$ is stationary, an LNE, and $L^\infty$ stable. We need only show that such a $u'$ can be constructed with $\|u'-u\|_{L^1(\Omega)}<\epsilon$ for every $\epsilon$. Let $\epsilon>0$ and let $\delta<\frac{\epsilon}{Vol(\Omega)\|u\|_{L^\infty}}$ and notice that $u'$ constructed with parameter $\delta$ satisfies \[\|u'-u\|_{L^\infty(\Omega)}=\esssup\sum_{k=1}^m((1+\delta)c_k-c_k)\chi_{S_k}\leq\max_k\{\delta c_k\}=\delta \|u\|_{L^\infty}\leq \frac{\epsilon}{Vol(\Omega)} \]
        Thus, $\|u'-u\|_{L^1(\Omega)}\leq Vol(\Omega)\|u'-u\|_{L^\infty(\Omega)}\leq \epsilon$. This completes the proof.
    \end{proof}    
    
\end{appendices}

\end{document}